\documentclass[a4paper, 11pt, oneside]{Thesis}  
\graphicspath{{Figures/}}  

\usepackage[square, numbers, comma, sort&compress]{natbib}  
\usepackage{verbatim}  
\usepackage{vector}  
\hypersetup{urlcolor=blue, colorlinks=true}  

\usepackage{amsthm}
\usepackage{amsfonts}
\usepackage{amsmath} 
\usepackage[normalem]{ulem}

\usepackage{enumerate}
\usepackage{mathtools}
\renewcommand{\ref}[1]{\eqref{#1}}

\DeclarePairedDelimiter\floor{\lfloor}{\rfloor}
\usepackage{hyperref}

\usepackage{algorithm}
\usepackage{algorithmic}

\newcommand{\Y}{\mathbb{Y}}
\usepackage{multicol}

\newcommand{\rcv}{\texttt{RCV1}}
\newcommand{\ctype}{\texttt{Covertype}}
\newcommand{\astro}{\texttt{Astro-ph}}
\newcommand{\aset}{\texttt{Alpha}}

\def\mbf{\mathbf}
\def\mbs{\boldsymbol}
\def\mc{\mathcal}

\def\spec{\sigma_1}
\def\lmax{\lambda_{\mathrm{max}}}
\newcommand{\ds}{\hat{\rho}^2}
\newcommand{\dists}{\rho}

\DeclareMathOperator*{\argmin}{argmin}

\DeclareMathOperator{\Tr}{Tr}

\newcommand{\eqdef}{\stackrel{\text{def}}{=}}
\newcommand{\R}{\mathbb{R}}                      
\newcommand{\Prob}{\mathbb{P}}                   
\newcommand{\E}{\mathbb{E}}                      

\newcommand{\calF}{\mathcal{F}}

\newcommand{\commfrac}{\nu}

\newcommand{\x}{ {\bf x}}
\newcommand{\vv}{ {\bf v}}

\newcommand{\X}{ {\bf X}}
\newcommand{\Z}{ {\bf Z}}
\newcommand{\Q}{ {\bf Q}}

\newcommand{\w}{  {\bf w}}

\newcommand{\z}{\mathbf{z}}

\newcommand{\vc}[2]{#1^{(#2)}}                   

\newcommand{\removed}[1]{}
\newcommand{\step}{\eta}

\newcommand{\norm}[1]{\left\| #1 \right\|}

\newcommand{\trans}{{\top}}

\newcommand{\specnorm}{\rho}
\newcommand{\reg}{\mu}

\newcommand{\bP}{\mathbf{P}}

\newcommand{\g}{\mathbf{g}}

\newcommand{\sgi}[2]{\beta_{#1,#2}}
\newcommand{\sgt}[1]{\gamma_{#1}}
\newcommand{\SGI}[1]{\mbs{\beta}_t}
\newcommand{\SGT}{\mbs{\gamma}}

\begin{document}

\frontmatter	  

\title  {Data Dependent Convergence for Distributed Stochastic Optimization}
\authors  {\texorpdfstring
            {\href{http:www.avleenbijral.com}{Avleen Singh Bijral}}
            {Avleen Singh Bijral}
            }
\addresses  {\groupname\\\deptname\\\univname}  
\date       {\today}
\subject    {}
    
\maketitle

\setstretch{1.3}  

\fancyhead{}  
\rhead{\thepage}  
\lhead{}  

\pagestyle{fancy}  

\clearpage  

\pagestyle{empty}  

\null \null \null \null \null \null \null \null \null \null \null
\textbf{\textit{``I don't like work... but I like what is in work - the chance to find yourself. Your own reality - for yourself, not for others - which no other man can ever know.''}}

\begin{flushright}
Joseph Conrad
\end{flushright}

\clearpage  

\addtotoc{Abstract}  
\abstract{
\addtocontents{toc}{\vspace{1em}}  

In this dissertation we propose alternative analysis of distributed stochastic gradient descent (SGD) algorithms that rely on spectral properties of the data covariance. As a consequence we can relate questions pertaining to speedups and convergence rates for distributed SGD to the data distribution instead of the regularity properties of the objective functions. More precisely we show that this rate depends on the spectral norm of the sample covariance matrix. An estimate of this norm can provide practitioners with guidance towards a potential gain in algorithm performance. For example many sparse datasets with low spectral norm prove to be amenable to gains in distributed settings. Towards establishing this data dependence we first study a distributed consensus-based SGD algorithm and show that the rate of convergence involves the spectral norm of the sample covariance matrix when the underlying data is assumed to be independent and identically distributed (homogenous). This dependence allows us to identify network regimes that prove to be beneficial for datasets with low sample covariance spectral norm. Existing consensus based analyses(\cite{dualAveraging}, \cite{nedicDistributedOptimization}, \cite{distrStochSubgrOpt}) prove to be sub-optimal in the homogenous setting. Our analysis method also allows us to find data-dependent convergence rates as we limit the amount of communication. Spreading a fixed amount of data across more nodes slows convergence; in the asymptotic regime we show that adding more machines can help when minimizing twice-differentiable losses. Since the mini-batch results don't follow from the consensus results we propose a different data dependent analysis thereby providing theoretical validation for why certain datasets are more amenable to mini-batching. We also provide empirical evidence for results in this thesis.

}
\clearpage
\acknowledgements{
\addtocontents{toc}{\vspace{1em}}  

I may not have been completely successful in reaching the high bar set by my advisor Nati Srebro, but I did learn a lot from this experience, firstly as his student and also from the excellent courses he taught. The central ideas of this thesis would not have borne fruit without the intense scrutiny and the insightful feedback I received. I remain forever thankful to him.

I would also like to thank Anand Sarwate for being an outstanding second advisor and for teaching me much about research and about writing papers. Many ideas in this thesis were the result of long discussions with Anand and without countless red marks left on the several drafts these ideas would not have reached a conclusion. I also owe a debt to the instructors of several courses I had the pleasure of taking - Greg Shaknarovich,  Laci Babai, Alex Eskin and many others. The period at TTI-C would not have been the same without the company of my friends: Feng, Jian, Hao, Payman, Behnam, Somaye, Ankan, Shubendu, Taehwan, Andy, Karthik, Zhiyong and Jianzhu. 

The staff at TTI-C was always available for help with any administrative tasks. Many thanks to Chrissy, Adam and Liv. 

Without the unconditional support and encouragement of my father, mother and brother this would have never taken shape. It was towards the end of my time at graduate school that I met Amanda, but from then to the completion of this thesis she has been a steady source of love and encouragement. Her patience and presence was crucial and I am indebted to her. Finally, without our dog Nayeli, writing this wouldn't have been half as fun.

}
\supervisor{Nathan Srebro}
\examiner{Anand D. Sarwate}
\examiner{Gregory Shaknarovich}
\examiner{Lim Lek-Heng}
\clearpage  

\setstretch{1.3}  

\clearpage  

\pagestyle{fancy}  

\lhead{\emph{Contents}}  
\tableofcontents  

\lhead{\emph{List of Figures}}  
\listoffigures  

\lhead{\emph{List of Tables}}  
\listoftables  

\setstretch{1.5}  


\setstretch{1.3}  


\addtocontents{toc}{\vspace{2em}}  

\mainmatter	  
\pagestyle{fancy}  



\chapter{Introduction} 
\label{Chapter1}
\lhead{Chapter 1. \emph{Introduction}} 

\section{Overview-Data Dependent Distributed Stochastic Optimization}\label{section:overVi}

Stochastic convex optimization in machine Learning and statistics often refers to the problem   
\begin{align}
F(\w) \eqdef \E_{\x\sim\mc{P}}\left[\ell(\w^{\trans}\x)\right]
\label{eq:optForm}
\end{align}
where the stochasticity is in the access model. We have access to stochastic (sub)gradients $\mbf{\hat{g}}$ such that $\E\left[\mbf{\hat{g}}\right] \in \partial F(\w)$, the subgradient set at $\w$. The goal in Machine Learning and often Statistics is to solve for a regularized version of problem \eqref{eq:optForm} given a sample $\x_1,...,\x_N$ from the distribution $\mc{P}$.

The methods used to solve problem \eqref{eq:optForm} are randomized variations of optimization techniques that are often more feasible in large scale settings . As an example stochastic gradient descent (SGD) for empirical risk minimization with its inexpensive updates involving unbiased estimates of the gradient performs significantly better than the batch gradient descent since the batch gradient computation involves one complete iteration over the entire training sample. As a consequence SGD has found immense applications in large scale machine learning. See \citet{SSSC11:pegasos} and \citet{TakacBRS:13icml} for some of the more prominent examples.

However, despite all the advantages of stochastic gradient descent it is inherently a sequential method. For applications involving very large datasets the need for parallelization becomes imminent. This drawback, aided by developments in the world of cheaply deployed networks motivated the development of several distributed optimization methods with stochastic extensions, notably those of \citet{nedicDistributedOptimization}, \citet{DistStronglyConvex}, \citet{dualAveraging}, \citet{incrementalBertsekas} etc. In all these methods an ensemble of loss functions is assumed to be distributed across a network of $m$ compute nodes and communication protocols are proposed that tie in with standard optimization updates. As expected, the convergence guarantees involve parameters of the underlying communication graph either through the mixing rate of a related Markov random walk on the graph (See \citet{dualAveraging}) or parameters directly depending on the entries of the Markov matrix corresponding to this random walk (\cite{nedicDistributedOptimization},\cite{distrStochSubgrOpt}). 

On the downside these methods require communicating at every iteration and the communication cost for high dimensional optimization variables can offset the advantage obtained by distributing the computation. One might then wonder if the communication requirements can be relaxed. At the very end of this low communication spectrum is a strategy where the nodes perform their local computations and only communicate at the end by avergaing the iterates from all nodes. This simple averaging was recently analyzed by \citet{ZhangDW:12} and shown to be a viable strategy for a class of loss functions with strong regularity properties. In the intermediate regime one could propose more general communication protocols to offset the cost of frequent communication. Communicating a fixed proportion of the total number of iterations is a simple example. 

A different approach to distribute SGD could be to parallelize the computation of (sub)gradients. Instead of sampling a single estimate of the (sub)gradient we could sample $b$ estimates (mini batch) of the (sub)gradient and return an averaged estimate by processing the computation of each of the $b$ (sub)gradients on separate compute nodes. It is easy to see that this strategy conforms to the star network topolgy. In the analysis the next step would be to relate the objective error to variance reduction (\citet{duchiBartlett}) or the smoothness properties of the objective function (\citet{CotterEtAl}).

Most of the distributed stochastic optimization strategies alluded to take on an immense significance only in the context of machine learning and statistics and the central object of all machine learning and statistics research and practice is data and its properties. How then are these properties conducive to better performance of these methods? Are there data sources (distributions) for which these methods give better error guarantees or are more amenable to relaxed communication regimes? It is somewhat appealing to provide distribution-free guarantees, methods that warrant consistent performance, independent of the data at hand. But as we shall see in detail in the ensuing chapters, they sometimes come at a cost. The assumptions required for the error guarantees require an increasing degree of regularity (as in smoothness and boundedness) of the loss functions and the data (See \cite{ZhangDW:12} and \cite{CotterEtAl}). This leaves out several important non-smooth problems (e.g. SVMs) and furthermore mask the question whether some data distributions in general are more accomodating to better parallelization performance. This is a fact of considerable importance to the practicioner.

In this thesis we attempt to bridge the gap between the convergence properties and communication requirements of a class of distributed stochastic optimization problems and the underlying data distribution. We will establish data dependent convergence rates for several distributed strategies
\removed{\begin{itemize}
\label{itemize:distStrat}
\item Consensus based stochastic optimization with communication at every round.
\item Intermittent communication protocols.
\item Mini batch based SGD.
\end{itemize}}
Specifically, we establish the dependence of convergence rates for a class of $\ell_2$-regularized problems on the spectral properties of the sample covariance matrix. This, as we shall see will have important implications empirically. Additionally we shall provide empirical evidence of the dependence of the average at the end strategy on the spectral norm. 

In Section \eqref{section:prelim} we formally introduce the stochastic optimization problem including the assumptions made, additonal assumptions if required will be described in the relevant sections in later chapters.  In Section (\ref{section:commIter}) and (\ref{section:commInfreq}) we explore the consensus SGD paradigm and discuss existing work. Section (\ref{section:mBatch}) describes a mini batch approach to distributed stochastic optimization alongwith known results.  Finally Section (\ref{section:commEnd}) describes the minimal communication approach also known as one shot averaging.  

\section{Preliminaries}\label{section:prelim}
In this section we establish the basic assumptions and notation for the remainder of the thesis

\textbf{Data model.} Let $\hat{\mc{P}}$ be the empirical distribution corresponding to an i.i.d sample $S = \{\x_1, \x_2, \ldots, \x_N\}$ in $\mathbb{R}^d$ such that $\x \sim \mc{P}$ satisfies $\norm{\x} \le 1$ almost surely.  Let $\mbs{\hat{\Sigma}} = \E_{\x \sim \hat{\mc{P}}}[ \x \x^{\trans} ]$ be the sample covariance matrix.  Our goal is to express the performance of our algorithms in terms of $\rho^2 = \spec(\mbs{\hat{\Sigma}})$ where $\spec(\cdot)$ denotes the maximum singular value.  

\textbf{Problem.}  Our problem is to minimize:
\begin{align}
	J(\w) \eqdef \frac{\sum_{i=1}^{N}\ell(\w^{\trans}\x_i)}{N} + \frac{\mu}{2} \norm{\w}^2.
	\label{eq:optForm}
	\end{align}
where
	\begin{align}
	\w^* \eqdef \argmin_{\w} J(\w)
	\label{eq:globalopt}
	\end{align}
\removed{We defer studying the gap between the iterates and the minimizer of the population objective in \eqref{eq:optForm} to future work.}
We will denote the subgradients of $J(\w)$ by $\nabla J(\w) \in \partial J(\w)$.

In our analysis we will make the following assumptions about the individual functions $\ell(\w^{\trans}\x)$:
\begin{itemize}
 \item The loss functions $\{ \ell(\cdot) \}$ are convex
\item The loss functions $\{ \ell(\cdot) \}$ are $L$-Lipschitz for some $L > 0$.
\end{itemize}
Any further assumptions will be made where necessary. Note that $J(\w)$ is $\mu$-strongly convex due to the $\ell_2$-regularization. For binary classification problems $\x$ is assumed to be scaled by its label in $\{-1,+1\}$.

\textbf{Network Model.} We consider a model in which minimization in \eqref{eq:globalopt} is carried out by $m$ computational nodes (cpu cores or machines).  These nodes are arranged in a network whose topology is given by a graph $\mc{G}$ -- an edge $(i,j)$ in the graph means nodes $i$ and $j$ can communicate.  A matrix $\bP$ is called graph conformant if $\bP_{ij} > 0$ only if the edge $(i,j)$ is in the graph.  We will consider algorithms which use a doubly stochastic and graph conformant sequence of matrices $\bP(t)$. Note that the mini batching paradigm corresponds to the star topology for $\mc{G}$.

\section{Consensus Based Optimization}\label{section:commIter}

In consensus-based optimization methods information about the local iterates is exchanged every iteration with the neighboring nodes as defined by a communication network. Such network optimization problems appear in a variety of applications such as multi-agent coordination and estimation problems in sensor networks. The chief requirement in these problems is the need to distribute computation across several, usually inexpensive, compute nodes and communicate lightly in a robust fashion. In machine learning such a setting is useful, for example when dealing with very large scale datasets. Subsets of data can then be distributed across a network and the compute nodes exchange information to solve the underlying loss minimization problem.

 
 Several authors have proposed distributed algorithms involving nodes computing local gradient steps and averaging iterates, gradients, or other functions of their neighbors~\citep{nedicDistributedOptimization,dualAveraging,distrStochSubgrOpt}.  By alternating local updates and consensus with neighbors, estimates at the nodes converge to the optimizer of $J(\cdot)$. 

\citet{nedicDistributedOptimization} propose a subgradient method and analyze the consensus problem when the objective $J(\dot)$ is unconstrained convex (not necessarily smooth) and give a constant error convergence bound $\frac{m}{\eta}\mc{O}\left(\frac{1}{T}\right) + const$ with a constant step-size. The analysis was then generalized and extended to the constrained case \citep{distrStochSubgrOpt} with (sub)gradients corrupted by stochastic noise, still obtaining a constant error convergence and a $\mc{O}(m^4)$ scaling with respect to the network. In contrast \citet{DistStronglyConvex} extended the primal consensus framework to the strongly convex regime and obtained $\mc{O}\left(\frac{\log(Tm)}{(1-\lambda_2(\bP))T}\right)$ convergence for a epoch decreasing step-size. Various extensions of the consensus problems have been proposed for the case of time varying networks \cite{DistOptTimeVarying} and asynchronous communication protocols \cite{DistOptAsyn}.

\citet{dualAveraging} proposed a consensus version of Nesterov's dual averaging algorithm \cite{NesterovDA} for a convex objective. The method essentially performs the consensus step in the dual space and employs a network dependent step-size to obtain a $\mc{O}\left(\frac{\log(Tm)}{\sqrt{(1-\lambda_2(\bP))T}}\right)$ guarantee. A tradeoff analysis for distributed dual averaging, relating the network size to the frequency of communication was presented by \citet{TsianosNIPS2012} wherein recipes for choosing the network size to offset low communication were presented. 

Next we discuss in detail some of the representative work for consensus based optimization algorithms.

\subsection{Consensus Primal Averaging}\label{subsec:DPA}

In this setting the topology of the network connecting the compute nodes is described by a graph $\mc{G}=\{\mc{V},\mc{E}\}$ where $\mc{V}$ is the set of the vertices and the $\mc{E}$ is edge set. Moreover there is a doubly stochastic matrix $\bP \in \Re^{m \times m}$ satisfying the graph constraints $\bP_{ij} > 0$ if $\{i,j\} \in \mc{E}$ and $\bP_{ij} = 0$ otherwise. The edge weights and the graph topology have a significant impact on the convergence of the distributed methods to be described and we will expand upon this topic in Chapter \eqref{Chapter5}.

The general synchronized primal averaging framework is as described in Algorithm \eqref{algorithm:D-SGD} where the choice of step-size, $\eta_t$ and the doubly stochastic matrix $\bP(t)$ will lead to different algorithms.

\begin{algorithm}[!htb]
 \caption{Distributed Stochastic Gradient Descent - Primal Averaging}
\begin{algorithmic}\label{algorithm:D-SGD}
\STATE \COMMENT{Each $i \in [m]$ executes}
\STATE Initialize $\w_i(0)=\mathbf{0}$ for all nodes $i \in [m]$ 
\FOR{$t=1$ {\bfseries to} $T$ }
\STATE Compute $\g_i(t)$ an unbiased estimate of the (sub)gradient $\nabla J(\w_i(t))$ 
\STATE $\w_i(t+1) = \sum_{j=1}^{m} \bP_{ij}(t)\w_j(t) - \eta_t \g_i(t)$
\ENDFOR
\STATE Return for any $i$
\begin{align*}
\bar{\w}_T(i)=\frac{1}{T}\sum_{t=1}^{T} \w_i(t)
\end{align*}
\end{algorithmic}
\end{algorithm}

Algorithm \eqref{algorithm:D-SGD} in essence computes an averaged primal estimate and then takes a stochastic (sub)gradient step. With some universal and specific assumptions we can recover different algorithms
\begin{enumerate}[I]
\item Universal 
\begin{enumerate}\label{enumerate:uniAssum} 
\item  Graph $\mc{G}=\{\mc{V},\mc{E}\}$ is connected.
\item The matrix $\bP(t)$ is doubly stochastic.
\item  $J(\w)$ is convex.
\end{enumerate}
\item Distributed subgradient method of \citet{nedicDistributedOptimization}.
\begin{enumerate}\label{enumerate:nedicAssum}
\item There exists a scalar $0<\alpha <1$ such that for all $i \in [m]$
\begin{enumerate}
\item $\bP_{ii}(t) \geq \alpha$ for all $t$.
\item $\bP_{ji}(t) \geq \alpha$ for all $t$ and all $j$ such that $\{j,i\} \in \mc{E}$.
\item The step-size $\eta_t = \eta$ is constant.
\end{enumerate}
\end{enumerate}
\item Distributed strongly convex optimization of \citet{DistStronglyConvex}.
\begin{enumerate}\label{enumerate:rabbatAssum}
\item $J(\w)$ is strongly convex.
\item $\bP(t)$ is fixed for all $t$.
\item The step-size is fixed for a epoch and is halved after each epoch.
\end{enumerate}
\end{enumerate}

Assumptions (\ref{enumerate:uniAssum}) are necessary for the convergence of the local iterates to the global optimum for all the primal and dual averaging methods, while assumptions (i) and (ii) from \ref{enumerate:nedicAssum} specific to \cite{nedicDistributedOptimization} imply a minimum amount of communication at every iteration.


In Chapter \eqref{Chapter3} we will describe a distributed strongly convex algorithm with a decreasing step-size and a corresponding data dependent guarantee.

\subsection{Consensus Dual Averaging}

In \cite{dualAveraging} Duchi et al. extend the dual averaging algorithm of Nesterov (\cite{NesterovDA}) using the distributed framework described in section (\ref{subsec:DPA}). The method is given in Algorithm (\ref{algorithm:Dual-Av})

\begin{algorithm}[!htb]
 \caption{Distributed Stochastic Dual Averaging}
\begin{algorithmic}\label{algorithm:Dual-Av}
\STATE \COMMENT{Each $i \in [m]$ executes}
\STATE Initialize $\z_i(0)=\mathbf{0}$ for all nodes $i \in [m]$ 
\FOR{$t=1$ {\bfseries to} $T$ }
\STATE Compute $\g_i(t)$ an unbiased estimate of the (sub)gradient $\nabla J(\w_i(t))$ 
\STATE $\z_i(t+1) = \sum_{j=1}^{m} \bP_{ij}(t)\z_j(t) + \g_i(t)$
\STATE $\w_i(t+1) = \argmin_{\w \in \mc{W}} \left \{ \z_i(t+1)^{\trans}\w + \frac{1}{\eta_t} \psi(\w) \right \}$
\ENDFOR
\STATE Return for any $i$
\begin{align*}
\bar{\w}_T(i)=\frac{1}{T}\sum_{t=1}^{T} \w_i(t)
\end{align*}
\end{algorithmic}
\end{algorithm}

The algorithm holds at each iteration $t$ for all $i$ a pair of vectors $(\w_i(t),\z_i(t))$ and in stark constrast to algorithm (\ref{algorithm:Dual-Av}), at time $t+1$ computes the new dual parameter from a weighted average of the dual parameter of its neighbors and then computes the next local iterate by a projection defined by the proximal function $\psi$ and the step-size $\eta_t$.

\subsection{Convergence Guarantees}

Table \eqref{tab:convGuarantee} shows convergence rates for primal and dual averaging algorithms. Note that \citet{dualAveraging} only look at the convex case and hence the sub-optimal dependence on $T$. But since we are mostly interested in the interplay of network and the data, the order of convergence is less important to our study

\begin{center}
\resizebox{\columnwidth}{!}{%
\begin{tabular}{ |c|c|c|c| } 
 \hline
 \textbf{Algorithm} & \textbf{Type} & \textbf{Step-size} & \textbf{Error}\\ 
    \citet{nedicDistributedOptimization} & Primal Avg. & $\eta_t = \eta$ & $J(\bar{\w}_i(T)) - J(\w^{*}) = \frac{m}{\eta}\mc{O}\left(\frac{1}{T}\right) + const. $\\ 
      \citet{DistStronglyConvex} & Primal Avg. & $\eta_t = \text{Epoch Decaying}$ & $J(\bar{\w}_i(T)) - J(\w^{*}) = \mc{O}\left(\frac{\log(Tm)}{(1-\lambda_2(\bP))T}\right)$ \\
            \citet{dualAveraging} & Dual Avg. & $\eta_t = \frac{c(1-\sqrt{\lambda_2(\bP)})}{L\sqrt{t}}$ & $J(\bar{\w}_i(T)) - J(\w^{*}) = \mc{O}\left(\frac{\log(Tm)}{\sqrt{(1-\lambda_2(\bP))T}}\right)$ \\
 \hline
\end{tabular} \label{tab:convGuarantee} 
}
\end{center}

All the results in Table \eqref{tab:convGuarantee} suggest that the error worsens with network size $m$. For a network topology that is not completely connected this is expected since as we spread a finite dataset across more and more machines it takes longer for machines to receive information about a fresh sample. Consequently the convergence guarantees do not reflect the setting when the data is homogenous (for e.g. when data has the same distribution), specifically error increases as we add more machines. This is counterintuitive, especially in the large scale regime, since this suggests that despite homogeneity the methods perform worse than the centralized setting (all data on one node).

This thesis (Chapter \ref{Chapter3}) provides a first analysis of a consensus based stochastic gradient method in the homogenous setting  and demonstrate that there exist regimes 
where we benefit from having more machines in any network.

\section{General Communication Strategies}\label{section:commInfreq}

A simple way to reduce the cost of frequent communication for Algorithm \eqref{algorithm:D-SGD} or sidestep the limitations of the underlying network is to incorporate more general protocols (e.g. intermittent  and asynchronous) through the Markov matrix $\bP(t)$ (in Algorithm \eqref{algorithm:D-SGD}) in the distributed stochastic optimizaiton regime. A simple example being the intermittent regime where the nodes communicate at a fixed frequency and perform communication free local updates the rest of the time. The generic communication protocol was also analyzed in the context of distributed dual averaging \citep{dualAveraging}.

A tradeoff analysis, relating the network size to the frequency of communication was presented in \cite{TsianosNIPS2012} wherein recipes for choosing the network size to offset low communication were presented. In Chapter \eqref{Chapter4} we present instead data dependent error rates for general communication protocols and example show that certain distributions are more amenable to communicating intermittently.

To mitigate the effect of limited communication, in Chapter \eqref{Chapter4} we propose and analyze a mini-batched extension to reduce communication costs. We interpret this as an intermediate regime between full communication and one-shot communication~\citep{ZhangDW:12},\citep{ShamirSrebroZhang:14icml}. Finally, we show that for twice-differentiable losses having more machines always helps (via a variance reduction) in the infinite data regime, using results of Bianchi et al.~\cite{BianchiFortHachem:13IEEETrans}.

\section{Mini Batches}\label{section:mBatch}

In stochastic gradient descent (SGD), the algorithm processes points sequentially and updates its estimate of the optimum after sampling each point as shown in Algorithm \eqref{algorithm:SGD}
\begin{algorithm}[!htb]
 \caption{Stochastic Gradient Descent}
\begin{algorithmic}\label{algorithm:SGD}
\STATE Initialize $\w_0=\mathbf{0}$
\FOR{$t=1$ {\bfseries to} $T$}
\STATE Compute $\g_t=\nabla f(\w(t);\x_{i_t})$, the unbiased estimate of $\nabla F(\w(t))$ at $\x_{i_t}$.
\STATE $\eta_t = \frac{c}{\lambda t}$
\STATE $\w(t+1) = \w(t) - \eta_t \g(t)$
\ENDFOR
\STATE Return
\begin{align*}
\bar{\w}(T)=\frac{1}{T}\sum_{t=1}^{m} \w(t)
\end{align*}
\end{algorithmic}
\end{algorithm}

To benefit from multiple compute nodes a common practice is to perform the computation of a mini batch of gradients in parallel. This corresponds to
\begin{align}
\g(t) = \frac{\sum_{i=1}^ b \g_i(t)}{b}
\end{align} 
in Algorithm \eqref{algorithm:SGD}. 

Recently the use of mini-batches in stochastic gradient descent,
as well as stochastic dual averaging and stochastic mirror descent, when minimizing a smooth loss
function has been considered (\cite{AgarwalD:11nips}, \cite{CotterSSS:11nips}). These works establish parallelization speedups for smooth loss minimization with mini-batches using accelerated variants of Algorithm \eqref{algorithm:D-SGD}. However, these results do not apply to non-smooth (but strongly convex) objective functions (for e.g. SVM). In Chapter \eqref{Chapter5} we prove data dependent bounds for the parallelization speedups for convex and strongly convex (not necessarily smooth) objectives.

\section{One Shot Averaging}\label{section:commEnd}
Parallelization using mini-batch and distributed primal averaging bears
a very high communication cost: the nodes must still communicate
the gradient estimates or the primal vectors every iteration.  The overall communication might be linear
in the total size of the data set, which may be prohibitive
in some distributed environments.
Instead, several authors have suggested that independent optimization
at each node followed by averaging the resulting predictors works well
in  practice~\cite{MannMMSW:09,McDonaldHM:2010,ZinkevichWSL:10,ZhangDW:12}.
We refer to such procedures as \textit{average-at-the-end}.  This
approach sits on the opposite end of the communication spectrum, with
each machine sending only a single message at the end of the
computation.  

Zhang et al.~\cite{ZhangDW:12} recently presented an analysis for the average-at-the-end procedure for twice-smooth loss functions (with bounded first, second and third derivatives) and additional bounded moment assumptions. In the statistical setting, the average-at-the-end approach can be thought of as simply averaging random draws from a probability distribution on predictors induced by the stochastic optimization algorithm.  

\subsection{\textit{Average-at-the-end}}

The setting consists of a dataset of $N=mn$ samples which is divided uniformly among a network of $m$ computing nodes. For each node there exists a local empirical objective
\begin{align*}
 F_{S_j}(\w)=\frac{1}{n}\sum_{\x \in S_j} f(\w;\x)
\end{align*}
where $j=1\ldots m$.
\begin{algorithm}[!htb]
\caption{SAVGM}
\begin{algorithmic}\label{algorithm:savgm}
\STATE For each $i\in [m]$ the node $i$ computes its approximate local minimizer
\begin{align*}
\w_i^{n} \approx \arg \min_{\w\in \mathcal{C}} \frac{1}{n}\sum_{\x \in S_j} f(\w;\x)
\end{align*}
by running algorithm \ref{algorithm:SGD} for $n$ iterations with step-size $\eta_t=c/(\lambda t)$ for some constant $c>1$.\\
\STATE {\bfseries Output:} 
\begin{align*}
\bar{\w}^n=\frac{1}{m}\sum_{i=1}^{m} \w^{n}_i
\end{align*}
\end{algorithmic}
\end{algorithm}

For algorithm \ref{algorithm:savgm} it is then shown that under the following assumptions
\begin{itemize}
\item \textbf{Assumption 1:} The parameter space $\mathcal{C} \subset \Re^d$ is a compact convex set and each individual loss function $f(\w;\x_i)$ is convex. A standard assumption across several optimization problems.
\item \textbf{Assumption 2:} There exists a function $L:\mathcal{X} \to \Re^{+}$ such that
\begin{align*}
\norm{\nabla^2 f(\w;\x) - \nabla^2 f(\w^{*};x)} \leq L(\x) \norm{\w-\w^{*}}
\end{align*}
and $\E[L^2(X)] \leq L^2$ and $\norm{\cdot}$ refers to the matrix operator norm. This assumption implies Lipschitz continuity of the Hessian near the optimal value.
\item \textbf{Assumption 3:} There are finite constants $G$ and $H$ such that 
\begin{align*}
\E[\norm{\nabla f(\w;X)}^4]\leq G^4 \text{, } \E[\norm{\nabla^2 f(\w^{*};X)}^4]\leq H^4 \ \forall \w \in \mathcal{C}.
\end{align*}
\item \textbf{Assumption 4:} The complete sample function $F(\w)$ is $\lambda$-strongly convex over the space $\mathcal{C}$ implying
\begin{align*}
\nabla^2 F(\w) \succeq \lambda I
\end{align*}
for all $\w \in \mathcal{C}$. This last assumption necessitates a non-negligible curvature of the global optimization function.
\end{itemize}

the following result holds (Theorem $3$ from \cite{ZhangDW:12})

\begin{theorem}\label{thm:savgmres}
\begin{align*}
\E \norm{\bar{\w}^n-\w^{*}}^2 \leq \frac{\alpha G^2}{\lambda}\frac{1}{mn} + \frac{\beta^2}{n^{3/2}}
\end{align*}
where $\alpha = 4c^2$ and
\begin{align}
\beta = \max\left \{\frac{cH}{\lambda}, \frac{c\alpha^{3/4}G^{3/2}}{(c-1)\lambda^{5/2}} \left (\frac{\alpha^{1/4}LG^{1/2}}{\lambda^{1/2}} + \frac{4G+HR}{\rho^{3/2}} \right )\right \}
\end{align}
\end{theorem}

Theorem \ref{thm:savgmres} shows that that algorithm (\ref{algorithm:savgm}) attains the optimal $\mathcal{O}(1/N)$ convergence rate under assumptions $1-4$ when the following the number of compute nodes $m$ satisifies the following condition
\begin{align}
m \leq \left ( \frac{\alpha G^2}{\beta^2 \lambda} \right )^{2/3} \sqrt[3]{N}
\label{eq:machinesCond}
\end{align}
Condition \eqref{eq:machinesCond} on $m$ can be improved to $\mathcal{O}(\sqrt{N})$ if we assume the existence of infinite moments in assumption $3$. 

Theorem \eqref{thm:savgmres} for algorithm \eqref{algorithm:savgm} leads to a distribution free guarantee on the number of machines the dataset can be parititioned into and still have an optimal error guarantee. This is good news for distributed algorithms since one doesn't have to worry about synchronization issues, communication overheads or network failure issues. With $m$ instantiations of algorithm \eqref{algorithm:savgm} such that $m$ satisfies condition \eqref{eq:machinesCond}, and a \textit{averaging-at-the-end} strategy we can get near optimal performance.

As discussed previously in Section \eqref{section:overVi} these error guarantees come at the price of assuming rather well behaved loss functions and data distributions. The proofs of results such as Theorem \eqref{thm:savgmres} rely heavily on higher order Taylor series expansions in turn implying strong smoothness assumptions (See appendix \citep{ZhangDW:12}). This strategy clearly does not bode well for a significant and important part of the machine learning and statistics stratosphere, for e.g. support vector machines and estimation problems involving Huber losses (\cite{huberLoss}) do not satisfy the smoothness assumption ($2$). Even the widely used, once differentiable surrogate for the Hinge loss, the squared Hinge loss fails to satisfy the twice differentiable conditions. 

In Chapter \eqref{Chapter6} we will provide empirical insight into how the convergence rates appears to be dependent on the spectral norm of the data distributions. The theoretical analysis supporting this observation is currently incomplete and we hope to we able to prove the result in a future work.

\section{Experimental Setup}

\subsection{Data sets and Network Settings} 

\begin{table}[t]
\centering
\caption{Data sets and parameters for experiments}
\begin{tabular}{l|c|c|c|c|c}
data set & training  & test  & dim. & $\lambda$ & $\specnorm^2$ \\
\hline
\rcv & $781,265$ & $23,149$ & $47,236$ & $10^{-4}$ & $0.01$ \\
\astro & $29,882$ & $32,487$ & $99,757$ & $5\times 10^{-5}$ & $0.01$ \\
\aset & $100,000$ & $50,000$ & $500$ & $10^{-4}$ & $0.35$ \\
\ctype & $522,911$ & $58,001$ & $47,236$ & $10^{-6}$ & $0.21$ 
\end{tabular}
\end{table}

The data sets used in our experiments are summarized in Table \ref{tab:data}. \ctype \ is the forest covertype dataset~\cite{covertype} used in \cite{SSSC11:pegasos} obtained from the UC Irvine Machine Learning Repository~\cite{Lichman:2013}, \astro \ is comprised of abstracts of papers from physics also of \cite{SSSC11:pegasos}, rcv1 is from the Reuters collection obtained from libsvm collection \cite{ChangL:11libsvm}. Alpha is a dense dataset obtained from the Pascal large scale learning challenge~\cite{alpha}.  The \rcv \ and \astro \ data sets have small values of $\hat{\rho}^2$, whereas \aset \ and  \ctype \ have larger values of $\hat{\rho}^2$. \label{tab:data} In all the experiments we looked at $\ell_2$-regularized classification objectives for problem \eqref{eq:optForm}.   Each plot is averaged over $5$ runs except where specified.


The data consists of pairs $\{(\x_1,y_1), \ldots, (\x_N,y_N)\}$ where $\x_i \in \mathbb{R}^d$ and $y_i \in \{-1,+1\}$.   We performed experiments on two objectives.  To analyze the effect of communication chose the hinge loss $\ell(\w^{\trans}\x ) = (1 - \w^{\trans} \x y)_{+}$.  
The values of the regularization parameter $\reg$ are chosen from to be the same as those in Shalev-Shwarz et al.~\cite{SSSC11:pegasos}. 

We simulated networks of compute nodes of varying size ($m$) arranged in a $k$-regular graph with $k = \floor{0.25m}$ or a fixed degree ($k=20$). Note that the dependence of the convergence rate of consensus based optimization procedures on the properties of the underlying network has been investigated before and we refer the reader to Agarwal and Duchi~\cite{AgarwalD:11nips} for more details. In this paper we experiment only with $k$-regular graphs. The weights on the Markov matrix $\bP$ are set by the max-degree random walk \citep{DiaconisBoyd}).
	\[
	P_{ij} = \left\{ \begin{array}{ll}
		\min\{ 1/d_i, 1/d_j \} & (i,j) \in \mc{E} \\
		\sum_{(i,k) \in \mc{E}} \max\{ 0, 1/d_i - 1/d_k\} & i = j \\
		0 & (i,j) \notin \mc{E}
		\end{array}
		\right.
	\]
where $d_i$ is the degree of node $i$.  
Each node is randomly assigned $n=\floor{N/m}$ points. For better performance we could optimize the weights on each edge~\citep{DoubStoch::nips09}.

\section{Summary}

In this chapter we introduced different computational paradigms for SGD including consensus, mini batching and one shot averaging. We described the problem setup and discussed existing works with their limitations, setting up the analysis and discussion in the subsequent chapters.



\chapter{Preliminary Results} 
\label{Chapter2}
\lhead{Chapter 2. \emph{Preliminary Results}} 

In this chapter we establish results related to spectral norm of randomly sampled principal gram submatrices. These will find a place in the analysis of algorithms presented in the subsequent chapters. We delegate a chapter to these results as they stand on their own and are potentially useful, even outside the confines of this thesis.

The first part of this chapter describes results that bound the spectral norm of submatrices of inner products of points sampled from some arbitrary distribution $\mc{P}$. We look at both sampling with and without replacement methods.

\section{Spectral Norm of Sampled Gram Submatrices}

In this section we establish bounds on expected spectral norm of principal submatrices of points sampled from a distribution $\mc{P}$ in terms of the spectral norm of the covariance matrix of the distribution. These results form the backbone of the convergence proofs to be presented in the later chapters. Though they are general enough in the current form to be used in other applications.

\subsection{Bound on Principal Gram Submatrices }

We establish the following inequality which follows by applying the Matrix Bernstein inequality of Tropp~\cite{Tropp:12tail}.


\begin{theorem}\label{theorem:specnorm}
Let $\mc{P}$ be a distribution on $\mathbb{R}^d$ with second moment matrix $\mbs{\Sigma} = \E_{\Y \sim \mc{P}}[ \Y \Y^{\trans}]$ such that $\norm{\Y_k} \le 1$ almost surely.  Let $\rho^2 = \spec(\mbs{\Sigma})$.  Let $\Y_1, \Y_2, \ldots, \Y_K$ be an i.i.d. sample from $\mc{P}$ and let
	\[
	\Q_K = \sum_{k=1}^{K} \Y_k \Y_k^{\trans}
	\]
be the empirical second moment matrix of the data.  Then for $K > \frac{4}{3 \rho^2} \log d$,
	\begin{align}
	\E\left[ \frac{ \spec( \Q_K ) }{K} \right] 
	&\le 5 \rho^2 
	\end{align}
and for $K > \max\{ \frac{4}{3 \rho^2} \log d, \frac{8 \sqrt{2}}{3 \rho^2} \sqrt{d} \}$,
	\begin{align}
	\E\left[ \frac{ \spec( \Q_K )^2 }{K^2} \right] 
	&\le 14 \rho^4.
	\end{align}	
\end{theorem}

\begin{proof}
Let $\Y$ be the $d \times K$ matrix whose columns are $\{\Y_k\}$.  Define $\X_k = \Y_k \Y_k^{\trans} - \mbs{\Sigma}$.  Then $\E[ \X_k ] = \mbf{0}$ and
	\begin{align*}
	\lmax(\X_k) &= \lmax\left( \Y_k \Y_k^{\trans} - \mbs{\Sigma} \right) \\
	&\le \norm{ \Y_k }^2 \\
	&\le 1,
	\end{align*}
because $\mbs{\Sigma}$ is positive semidefinite and $\norm{ x_i } \le 1$ for all $i$.  Furthermore,
	\begin{align*}
	\spec\left( \sum_{k=1}^{K} \E\left[ \X_k^2 \right] \right) 
	&= t \spec\left(  
		\E\left[ \Y_k \Y_k^{\trans} \Y_k \Y_k^{\trans} \right] - \mbs{\Sigma}^2 \right) \\
	&\le t \spec\left( \E\left[ \norm{Y_k}^2 \Y_k \Y_k^{\trans} \right] \right) 
		+ t \spec\left( \mbs{\Sigma} \right)^2 \\
	&\le t (\rho^2 + \rho^4) \\
	&\le 2 t \rho^2
	\end{align*}
since $\rho \le 1$.
	
Applying the Matrix Bernstein inequality of Tropp~\cite[Theorem 6.1]{Tropp:12tail}
	\begin{align}
	\Prob\left( \spec\left( \sum_{k=1}^{K} \mbf{X}_k \right) \ge r \right) 
	&\le \left\{ 
		\begin{array}{ll}
		d \exp\left( - 3 r^2/( 16 K \rho^2) \right) & r/K \le 2 \rho^2 \\
		d \exp\left( - 3 r/8 \right) & r/K \ge 2 \rho^2
		\end{array}
		\right.
	\label{eq:bernstein1}
	\end{align}
Now, note that
	\begin{align*}
	\spec\left( \sum_{k=1}^{K} \X_k  \right) = \spec\left( \sum_{k=1}^{K} \Y_k \Y_k^{\trans} - \mbs{\Sigma}  \right),
	\end{align*}
so $\spec\left( \sum_{k=1}^{K} \mbf{X}_k \right) \ge r$ is implied by 
	\[
	\left| \frac{1}{t} \spec\left( \sum_{k=1}^{K} \Y_k \Y_k^{\trans} \right) - \spec\left( \mbs{\Sigma}  \right) \right| \ge r.  
	\]
Therefore
	\begin{align}
	\Prob\left( \left| \frac{ \spec(\Q_K) }{ K } - \rho^2 \right| \ge r' \right)  \
	&\le \left\{ 
		\begin{array}{ll}
		d \exp\left( - 3 K r'^2/( 16 \rho^2) \right) & r' \le 2 \rho^2 \\
		d \exp\left( - 3 K r'/8 \right) & r' \ge 2 \rho^2
		\end{array}
		\right.
	\label{eq:bernstein2}
	\end{align}
Integrating \eqref{eq:bernstein2} yields
	\begin{align*}
	\E\left[ \frac{ \spec( \Q_K ) }{K} \right]
	&= \int_{0}^{\infty} \Prob\left( \frac{ \spec(\Q_K) }{ K } \ge x \right) dx \\
	&\le 3 \rho^2 + \int_{3 \rho^2}^{\infty} \Prob\left( \frac{ \spec(\Q_t) }{ K } - \rho^2 \ge x - \rho^2 \right) dx \\
	&\le 3 \rho^2 + \int_{2 \rho^2}^{\infty} \Prob\left( \frac{ \spec(\Q_t) }{ K } - \rho^2 \ge r' \right) dr' \\
	&\le 3 \rho^2 + \int_{2 \rho^2}^{\infty} d \exp\left( - \frac{3}{8} K r' \right) dr' \\
	&= 3 \rho^2 + \frac{8}{3} \cdot \frac{d}{K} \exp\left( - \frac{3}{4} \rho^2 K \right)
	\end{align*}
For $K > \frac{4}{3 \rho^2} \log d$,
	\begin{align*}
	\E\left[ \frac{ \spec( \Q_K ) }{K} \right] 
	&\le 3 \rho^2 + \frac{8}{3} \cdot \frac{3}{4} \cdot \frac{\rho^2}{\log d} \\
	&\le 5 \rho^2.
	\end{align*}	
	
Turning to the second inequality,
	\begin{align*}
	\E\left[ \frac{ \spec( \Q_K )^2 }{K^2} \right]
	&= \int_{0}^{\infty} \Prob\left( \frac{ \spec(\Q_K) }{ K } \ge \sqrt{x} \right) dx \\
	&\le 9 \rho^4 + \int_{9 \rho^4}^{\infty} \Prob\left( \frac{ \spec(\Q_K) }{ K } \ge \sqrt{x} \right) dx \\
	&= 9 \rho^4 + \int_{3 \rho^2}^{\infty} \Prob\left( \frac{ \spec(\Q_K) }{ K } \ge y \right) 2 y dy \\
	&= 9 \rho^4 + \int_{3 \rho^2}^{\infty} \Prob\left( \frac{ \spec(\Q_K) }{ K } - \rho^2 \ge y - \rho^2 \right) 2 y dy \\
	&= 9 \rho^4 + \int_{2 \rho^2}^{\infty} \Prob\left( \frac{ \spec(\Q_K) }{ K } - \rho^2 \ge r' \right) 2 (r' + \rho^2) dy \\
	&\le 9 \rho^4 + 2 d \int_{2 \rho^2}^{\infty} r' \exp\left( - 3 K r'/8 \right ) dr' 
		+ 2 \rho^2 d \int_{2 \rho^2}^{\infty} \exp\left( - 3 K r'/8 \right ) dr' \\
	&\le 9 \rho^4 + \frac{128 d}{9 K^2} 
		+ \frac{16}{3} \rho^2 \cdot \frac{d}{K} \exp\left( - \frac{3}{4} \rho^2 K \right) \\
	\end{align*}
where in the last line we used the formula for the mean of an exponential random variable.  This shows that for $K > \max\{ \frac{4}{3 \rho^2} \log d, \frac{8 \sqrt{2}}{3 \rho^2} \sqrt{d} \}$,
	\begin{align*}
	\E\left[ \frac{ \spec( \Q_K )^2 }{K^2} \right] &\le 14 \rho^4.
	\end{align*}
	
	Proceeding similarly we can show that there exists a fixed constant $c_0$ such that
	\begin{align}
	\E \left [ \frac{\spec(\Q_K)^4}{K^4} \right ] \leq c_0 \rho^8
	\end{align}
\end{proof}

\subsection{Bound on Principal Gram Submatrices - Intrinsic Dimension }

Often the data we work with resides in a low dimensional subspace and the effective or the intrinsic dimensionality of the data is much less than the ambient dimension $d$. To get tighter bounds then it then becomes necessary to incorporate some notion of intrinsic dimension. 
We work with the following defintion of intrinsic dimension (from \cite{Tropp:12tail}) for a positive semi-definite matrix $\mathbf{A}$.
\begin{align}
intdim(\mathbf{A}) = \frac{tr(\mathbf{A})}{\norm{\mathbf{A}}}
\end{align}

The proof applies the intrinsic dimension version of Matrix Bernstein inequality of Tropp~\cite{Tropp:12tail}.

\begin{theorem}\label{theorem:specnormIntdim}
Let $\mc{P}$ be a distribution on $\mathbb{R}^d$ with second moment matrix $\mbs{\Sigma} = \E_{\Y \sim \mc{P}}[ \Y \Y^{\trans}]$ such that $\alpha \le \norm{\Y_k} \le 1$ almost surely for some $\alpha>\rho$.  Let $\rho^2 = \spec(\mbs{\Sigma})$.  Let $\Y_1, \Y_2, \ldots, \Y_K$ be an i.i.d. sample from $\mc{P}$ and let
	\[
	\Q_K = \sum_{k=1}^{K} \Y_k \Y_k^{\trans}
	\]
be the empirical second moment matrix of the data.  Then for $K > \frac{4}{\rho^2} \log\frac{11}{\rho^2(\alpha^2-\rho^2)}$,
	\begin{align}
	\E\left[ \frac{ \spec( \Q_K ) }{K} \right] 
	&\le 5 \rho^2 
	\end{align}
\end{theorem}

\begin{proof}
Let $\Y$ be the $d \times K$ matrix whose columns are $\{\Y_k\}$.  Define $\X_k = \Y_k \Y_k^{\trans} - \mbs{\Sigma}$.  Then $\E[ \X_k ] = \mbf{0}$ and
	\begin{align*}
	\lmax(\X_k) &= \lmax\left( \Y_k \Y_k^{\trans} - \mbs{\Sigma} \right) \\
	&\le \norm{ \Y_k }^2 \\
	&\le 1,
	\end{align*}
because $\mbs{\Sigma}$ is positive semidefinite and $\norm{ x_i } \le 1$ for all $i$. 

Let us define $\Z = \sum_{k} \X_k$ then we have
\begin{align}
\norm{\E[\Z^2]} = K\norm{\E[\X_k^2]}
\end{align}

using $\norm{\mathbf{A}} - \norm{\mathbf{B}} < \norm{\mathbf{A}-\mathbf{B}}$ we get a lower bound for $\norm{\E[\Z^2]}$ as follows
\begin{align}
K\norm{\E[\X_k^2]} &\geq K \left ( \norm{\E \left [ \norm{Y_k}^2\Y_k \Y_k^{\trans} \right ]} -  \norm{\E \left [ \Y_k \Y_k^{\trans}\right ]}^2\right ) \notag \\
&= K (\alpha^2\rho^2 - \rho^4)
\end{align}

This gives us
\begin{align}
\label{eq:intdimBnd}
intdim(\E[\Z^2]) &= \frac{tr(\E[\Z^2])}{\norm{\E[\Z^2]}} \leq K \frac{tr \left(  \E \left [ \Y_k \Y_k^{\trans} \right ] - \E \left [ \Y_k \Y_k^{\trans} \right ]^2 \right) } {K (\alpha\rho^2 - \rho^4)} \notag \\
&\leq \frac{1-tr\left (\E \left [ \Y_k \Y_k^{\trans} \right ]^2 \right)}{\alpha^2\rho^2 - \rho^4} \notag \\
&\leq \frac{1}{\alpha^2\rho^2 - \rho^4}
\end{align}

Let us define $d=d(\E[\Z^2]) = intdim(\E[\Z^2])$ then we have
	
Applying the intrinsic dimension Matrix Bernstein inequality of Tropp~\cite[Theorem 7.3.1]{Tropp:12tail}
	\begin{align}
	\Prob\left( \spec\left( \sum_{k=1}^{K} \mbf{X}_k \right) \ge r \right) 
	&\le \left\{ 
		\begin{array}{ll}
		4d \exp\left( - 3 r^2/( 16 K \rho^2) \right) & r/K \le 2 \rho^2 \\
		4d \exp\left( - 3 r/8 \right) & r/K \ge 2 \rho^2
		\end{array}
		\right.
	\label{eq:bernstein1}
	\end{align}
Now, note that
	\begin{align*}
	\spec\left( \sum_{k=1}^{K} \X_k  \right) = \spec\left( \sum_{k=1}^{K} \Y_k \Y_k^{\trans} - \mbs{\Sigma}  \right),
	\end{align*}
so $\spec\left( \sum_{k=1}^{K} \mbf{X}_k \right) \ge r$ is implied by 
	\[
	\left| \frac{1}{t} \spec\left( \sum_{k=1}^{K} \Y_k \Y_k^{\trans} \right) - \spec\left( \mbs{\Sigma}  \right) \right| \ge r.  
	\]
Therefore
	\begin{align}
	\Prob\left( \left| \frac{ \spec(\Q_K) }{ K } - \rho^2 \right| \ge r' \right)  \
	&\le \left\{ 
		\begin{array}{ll}
		4d \exp\left( - 3 K r'^2/( 16 \rho^2) \right) & r' \le 2 \rho^2 \\
		4d \exp\left( - 3 K r'/8 \right) & r' \ge 2 \rho^2
		\end{array}
		\right.
	\label{eq:bernstein2}
	\end{align}
Integrating \eqref{eq:bernstein2} yields
	\begin{align*}
	\E\left[ \frac{ \spec( \Q_K ) }{K} \right]
	&= \int_{0}^{\infty} \Prob\left( \frac{ \spec(\Q_K) }{ K } \ge x \right) dx \\
	&\le 3 \rho^2 + \int_{3 \rho^2}^{\infty} \Prob\left( \frac{ \spec(\Q_t) }{ K } - \rho^2 \ge x - \rho^2 \right) dx \\
	&\le 3 \rho^2 + \int_{2 \rho^2}^{\infty} \Prob\left( \frac{ \spec(\Q_t) }{ K } - \rho^2 \ge r' \right) dr' \\
	&\le 3 \rho^2 + \int_{2 \rho^2}^{\infty} 4d \exp\left( - \frac{3}{8} K r' \right) dr' \\
	&= 3 \rho^2 + \frac{32}{3} \cdot \frac{d}{K} \exp\left( - \frac{3}{4} \rho^2 K \right) \notag \\
	&=  3 \rho^2 + 11 \cdot \frac{d}{K} \exp\left( - \frac{3}{4} \rho^2 K \right)
	\end{align*}
For $K > \frac{4}{3 \rho^2} \log (11d)$,
	\begin{align*}
	\E\left[ \frac{ \spec( \Q_K ) }{K} \right] 
	&\le 3 \rho^2 + 2\rho^2 \\
	&\le 5 \rho^2
	\end{align*}	
	
Finally using bound \eqref{eq:intdimBnd} gives us the required result.
	
\end{proof}
 
\emph{Remark:} The intrinsic dimension Lemma no longer depends on the dimension of the ambient space ($d$) and gives us a weaker requirement on the number of points to be sampled for the upper bound to be true. 
 
\subsection{Sampling Without Replacment Bound}
\begin{table}[t!]
{
\centering
\footnotesize
\begin{tabular}{|c|r|r|r|r|l}
\hline
& \multicolumn{2}{c|}{\textbf{Astro-ph}}& \multicolumn{2}{c|}{\textbf{Covertype}} \\
 Batch Size($K$) & \multicolumn{1}{|c}{$\E[\norm{\Q_K}]$} &
 \multicolumn{1}{c|}{$\beta_K$} &
 \multicolumn{1}{|c}{$\E[\norm{\Q_K}]$} &
 \multicolumn{1}{c|}{$\beta_K$} 
 \\ \hline \hline
 $32$ & 1.48  & 1.46 & 7.97 & 7.49
 \\
 $256$ & 4.776 & 4.75 & 54.57 & 54.36
 \\
 $1024$& 16.01 & 16.05  & 214.97 & 215.1
 \\
 $4096$ & 61.16 & 61.24 & 858.28 & 857.9 \\
 $8192$ & 121.47 & 121.50 & 1715.45 & 1715.13 \\
 \hline
\end{tabular}
\label{tab:dataspec}
\small \caption{Values of $\E[\norm{\Q_K}]$ and $\beta_K$ for two datasets of considerably different sparsity. It can be observed that the two values are very close and hence from a computational perspective best to use the one easier to compute (i.e. $\E[\norm{\Q_K}]$ as discussed before). }
}
\end{table}

In this section we obtain high probability bounds for sampling without replacement. These results provide us an approximation to the spectral norm of a large matrix. For example a (size $K$) block coordinate descent method described in \cite{TakacBRS:13icml} uses the step size $\beta_K \approx 1 + (K-1)\hat{\rho}^2$. For large $N$ the spectral norm is prohibitive to compute and an approximate estimate via the spectral norm of a submatrix of size $K$ is then useful. Additionally empirical values in table 2.1 indicate that these values are indeed close.

Since the block coordinate descent strategy described in \cite{TakacBRS:13icml} samples each block using with replacement strategy  we look at sampling from a fixed Gram matrix $\Q$ of $N$ points with replacement. 

We employ the Matrix Azuma inequality (~\cite[Theorem 7.1]{Tropp:12tail}) which states that
\begin{theorem}
For a finite sequence of $\{\X_k\}$ of self-adjoint matrices in dimension $d$ and a fixed sequence of $\{\mathbf{A}_k\}$ of self-adjoint matrices that satisfy
$\E[\X_k|\X_1,...,\X_{k-1}] = \mathbf{0}$ and $\mathbf{A}_k^2 \succeq \X_k^2$ for a fixed self-adjoint sequence $\mathbf{A}_k$ . Then for all $t\geq 0$ we have
\begin{align*}
P\left ( \lambda_{\max} \left(\Y=\sum_k \X_k \right) \geq t   \right ) \leq d \exp(-t^2/8\sigma^2)
\end{align*}
where $\sigma^2 = \|\sum_k \mathbf{A}_k^2\|$.
\end{theorem}

Applying this theorem to our setting ields the following result
\begin{theorem}
For a random principal submatrix $\Q_K$ of size $K$ sampled from a gram matrix $Q=\X\X^{\trans}$, without replacment, of size $N$ we have that with probability at least $1-\delta$
\begin{align}
\spec(\Q_K) \leq K\rho^2  + \sqrt{8CK\rho^2\log \frac{d}{\delta}}
\end{align} 
where $C$ is a universal constant and $\rho^2 = \spec(\Q)/N$.
\end{theorem}
\begin{proof}
Suppose $\mathcal{A} \subset [N]$ sampled without replacment such that $|\mathcal{A}|=K$.

let $\Y_k = \z_k \z_k^{T}$ and $\mathcal{B}_k =\{\mathcal{A}[1],...,\mathcal{A}[k]\}$  then let us compute in the sampling without replacement setting
\begin{align}
\E[\Y_k|\Y_1,...,\Y_{k-1}] = \E[\z_k \z_k^{T} | \Y_1,...,\Y_{k-1}] = \sum_{i \in [N]-\mathcal{B}_k}\frac{\z_i \z_i^{T}}{N+1-k}
\end{align}

Now define a new sequence $\X_k = \Y_k -  \sum_{i \in [N]-\mathcal{B}_k}\frac{\z_i \z_i^{T}}{N+1-k}$ This sequence is clearly a conditionally centered sequence satisfy first assumption in the theorem.

Now for $t = \sqrt{8\sigma^2 \log \frac{d}{\delta}}$ (by setting $t$ equal to the right hand side of the probability inequality), using the fact that $\norm{\mathbf{A}-\mathbf{B}} \geq \norm{\mathbf{A}} - \norm{\mathbf{B}}$ and the triangle inequality for the norm and the fact that the spectral norm of a matrix is always greater than or equal to its submatrix norm, we have that with probability at least $1-\delta$
\begin{align}
t \geq \spec(\Y) &= \spec\left( \sum_k \left ( \z_k \z_k^{T} - \sum_{i \in [N]-\mathcal{B}_k}\frac{\z_i \z_i^{T}}{N+1-k} \right )\right) \notag  \\
 &\geq \spec(\Q_K) - \sum_{k=1}^K \spec\left( \sum_{i \in [N]-\mathcal{B}_k}\frac{\z_i \z_i^{T}}{N+1-k} \right ) \notag \\
 &= \spec(\Q_K) - \sum_{k=1}^K  \frac{\spec\left(\Q_{[N]-\mc{B}_k}\right)}{N+1-K} \notag \\ 
 &\geq \spec(\Q_K) - \sum_{k=1}^K  \frac{\spec\left(\Q\right)}{N} \notag \\ 
 &= \spec(\Q_K) -  K\rho^2 \notag 
\end{align}

Now we need $\mathbf{A}_k$. We have for some universal constant $C$
\begin{align}
\sigma^2=\|\sum_k \mathbf{A}_k^2 \| \leq Cb\rho^2
\end{align}

Combining all the above we get that with probability at least $1-\delta$
\begin{align}
\spec(\Q_K) \leq K\rho^2  + \sqrt{8CK\rho^2\log \frac{d}{\delta}}
\end{align}

\end{proof}

This concludes the statements and the proofs of the Lemmas we will use in the subsequent chapters. 

\section{Summary}

We proposed bounds on spectral norm of randomly sampled gram submatrices. In the ensuing chapters these bounds will lead us to the data-dependence convergence of distributed SGD algorithms.



\chapter{Consensus SGD and Convergence Rates} 
\label{Chapter3}
\lhead{Chapter 3. \emph{Consensus SGD and Convergence Rates}} 

In this chapter we characterize how the spectral norm $\specnorm^2 = \spec( \E_{\hat{\mc{P}}}[ \x \x^{\trans} ])$ of the sample covariance of the data affects the rate of convergence of   stochastic consensus schemes under different communication requirements. Elucidating this dependence can help guide empirical practice by providing insight into when these methods will work well. We prove an upper bound on the suboptimality gap for distributed primal averaging that depends on $\specnorm^2$ as well as the mixing time of the weight matrix associated to the algorithm. Our result shows that networks of size $m<\frac{1}{\specnorm^2}$ gain from parallelization. Moreover in an asymptotic regime with infinite data at every node we show that for twice-differentiable loss functions this network effect disappears and that we gain from additional parallelization.

\textbf{Related Work.} Several authors have proposed distributed algorithms  involving nodes computing local gradient steps and averaging iterates, gradients, or other functions of their neighbors~\citep{nedicDistributedOptimization,dualAveraging,distrStochSubgrOpt}.  By alternating local updates and consensus with neighbors, estimates at the nodes converge to the optimizer of $J(\cdot)$. 
In these works no assumption is made on the local objective functions and they can be arbitrary. Consequently the convergence guarantees do not reflect the setting when the data is homogenous (for e.g. when data has the same distribution), specifically error increases as we add more machines. This is counterintuitive, especially in the large scale regime, since this suggests that despite homogeneity the methods perform worse than the centralized setting (all data on one node). 

We provide a first analysis of a consensus based stochastic gradient method in the homogenous setting  and demonstrate that there exist regimes %
where we benefit from having more machines in any network. We also show that for twice-differentiable losses, having more machines always helps (via a variance reduction) in the infinite data regime, using results of Bianchi et al.~\cite{BianchiFortHachem:13IEEETrans}. 

In contrast to our stochastic gradient based results, data dependence via the Hessian of the objective has also been demonstrated in parallel coordinate descent based approaches of Liu et al.~\cite{LiuWrightReBittSri} and the Shotgun algorithm of Bradley et al.~\cite{BradleyKyrola}. The assumptions differ from us in that the objective function is assumed to be smooth~\cite{LiuWrightReBittSri} or $\mathcal{L}_1$ regularized~\cite{BradleyKyrola}. Most importantly, our results hold for arbitrary networks of compute nodes, while the coordinate descent based results hold only for networks where all nodes communicate with a central aggregator (sometimes referred to as a master-slave architecture, or a star network), which can be used to model shared-memory systems.

We take an approach similar to that of Tak\'{a}\v{c} et al.~\citep{TakacBRS:13icml}, who developed a spectral-norm based analysis of mini-batching for non-smooth functions. We decompose the iterate in terms of the data points encountered in the sample path~\citep{CotterSSS:11nips}. This differs from analysis based on smoothness considerations alone~\citep{CotterSSS:11nips,AgarwalD:11nips,DekelGSX:12mini,ShamirSrebroZhang:14icml} and gives practical insight into how communication (full or intermittent) impacts the performance of these algorithms. Note that our work is fundamentally different in that these other works either assume a centralized setting \citep{CotterSSS:11nips,DekelGSX:12mini,ShamirSrebroZhang:14icml} or implicitly assume a specific network topology (e.g. \citep{ZhangDW:12} uses a star topology). For the main results we only assume strong convexity while the existing guarantees for the cited methods depend on a variety of regularity and smoothness conditions.

 \textbf{Limitation.} In the stochastic convex optimization (see for e.g. \cite{StochConOpt}) setting the quantity of interest is the population objective corresponding to problem \ref{eq:optForm}. When minimizing this population objective our results suggest that adding more machines worsens convergence (See Theorem \ref{theorem:mainThrm}). For finite data our convergence results satisfy the intuition that adding more nodes in an arbitrary network will hurt convergence. The finite homogenous setting is most relevant in settings such as data centers, where the processors hold data which essentially looks the same. 
In the infinite or large scale data setting, common in machine learning applications, this is  counterintuitive since when each node has infinite data, any distributed scheme including one on arbitrary networks shouldn't perform worse than the centralized scheme (all data on one node). Thus our analysis is limited in that it doesn't unify the stochastic optimization and the consensus setting in a completely satisfactory manner. To partially remedy this we explore distributed primal averaging for smooth strongly convex objectives in the asymptotic regime and show that one can gain from adding more machines in any network.

In this chapter we focus on a simple and well-studied protocol~\cite{nedicDistributedOptimization}. However, our analysis approach and insights may yield data-dependent bounds for other more complex algorithms such as distributed dual averaging~\cite{dualAveraging}. 
More sophisticated gradient averaging schemes such as that of Mokhtari and Ribeiro~\cite{MokhtariR:16jmlr} can exploit dependence across iterations~\cite{ShiLWY:15extra,SchmidLF:2015} to improve the convergence rate; analyzing the impact of the data distribution is considerably more complex in these algorithms.

These results provide a first step towards understanding data-dependent bounds for distributed stochastic optimization in settings common to machine learning. Our analysis coincides with phenomenon seen in practice: for data sets with small $\specnorm$, distributing the computation across many machines is beneficial, but for data with larger $\specnorm$ more machines is not necessarily better. We provide upper bounds on the gap between the iterates and the optimal solution: these bounds do not immediately yield parameters for practical use, but our work does suggest that taking into account the data dependence can improve the empirical performance of these methods.

\section{Problem Structure and Model}\label{section:samplingModel}

The data and network setup is as described in Section (\ref{section:prelim}) of Chapter \eqref{Chapter1}.

\subsubsection{Sampling Model}

We assume the $N$ data points are divided evenly among the $m$ nodes, and define $n \eqdef N/m$ to be the number of points at each node.   Let $S_j$ be the subset of $n$ points at node $j$.  The local stochastic gradient procedure consists of each node $j \in [m]$ sampling from $S_j$ with replacement.  This is an approximation to the local objective function
\begin{align}
J_j(\w) = \sum_{i \in S_j} \frac{\ell_i(\w^{\trans} \x_i)}{n} + \frac{\reg}{2} \norm{\w}^2.
\end{align}

\subsection{Algorithm}

\begin{algorithm}[!htp]
   \caption{Consensus Strongly Convex Optimization}
   \label{alg:DiSCO}
\begin{algorithmic}
   \STATE {\bfseries Input:} $\{\x_i\}_{i=1}^N$, $\reg> 0$, $T \geq 1$ 
   \STATE
   \STATE \COMMENT{Each $i \in [m]$ executes}
   \STATE {\bfseries Initialize:} set $\w_i(1) = {\bf 0} \in \R^d$.
   \FOR{$t=1$ {\bfseries to} $T$}
   \STATE Sample $\x_{i_t}$ uniformly with replacement from $S_i$.
   \STATE Compute $\g_i(t) \in  \partial\ell(\w_i(t)^{\trans}\x_{i_t})\x_{i_t} + \reg \w_i(t)$
   \STATE $\w_i(t+1) = \sum_{j=1}^m \w_j(t)\bP_{ij}(t) - \eta_t \g_i(t)\label{eq:updRule}$
   \ENDFOR
    \STATE {\bfseries Output:} $\bar{\w}_T(i)=\frac{1}{T}\sum_{t=1}^{T} \w_i(t)$ for any $i \in [m]$.
\end{algorithmic}
\end{algorithm}

Algorithm (\ref{alg:DiSCO}) describes the distributed strategy we analyze in the subsequent sections. Every node $i$ samples a point uniformly with replacement from a local pool of $n$ points and then updates its iterate by computing a weighted sum with its neighbors followed by a local subgradient step. The time dependence of the Markov matrix $\bP(t)$ and the step-size indicates that Algorithm \eqref{alg:DiSCO} is a generalized strategy and specific choices of this matrix with the step-size will pave the way for analysis of different strategies for communication. 


Algorithms like \eqref{alg:DiSCO}, also referred to as primal averaging, have been analyzed previously~\cite{nedicDistributedOptimization,distrStochSubgrOpt,DistStronglyConvex}. In these works it is shown that the convergence properties depend on the structure of the underlying network via the second largest eigenvalue of $\bP$.  We consider in this section the case when $\bP(t)=\bP$ for all $t$ where $\bP$ is a fixed Markov matrix.  This corresponds to a synchronous setting where communication occurs at every iteration.

We analyze the use of the step-size $\eta_t=1/(\mu t)$ in Algorithm \ref{alg:DiSCO} and show that the convergence depends on the spectral norm of the covariance matrix of the data distribution $\mc{P}$. Specifically, we prove the following result.
\begin{theorem}
Fix a Markov matrix $\bP$ and let $\specnorm^2 =\sigma_1(\mbf{\hat{\Sigma}})$ denote the spectral norm of the covariance matrix of the data distribution. Consider Algorithm \eqref{alg:DiSCO} when the objective $J(\w)$ is strongly convex, $\bP(t) = \bP$ for all $t$, and $\step_t=1/(\reg t)$.  Let $\lambda_2(\bP)$ denote the second largest eigenvalue of $\bP$.  Then if the number of samples on each machine $n$ satisfies 
	\begin{align}
	n > \frac{4}{3 \rho^2} \log \left( d\right)
	\end{align}
and the number of iterations $T$ satisfies
	\begin{align}
	T &> 2 e  \log(1/\sqrt{ \lambda_2(\bP)}) \\
	\frac{T}{\log(T)} &> \max\left(\frac{4}{3 \rho^2} \log \left( d\right), \frac{ \left( \frac{8}{5} \right)^{\frac{1}{4}} \sqrt{ m/\specnorm }}{\log(1/\lambda_2(\bP))}\right), \label{eq:thm_Tbound}
	\end{align}
then the expected error for each node $i$ satisfies
\begin{align}
\E\left[ J(\hat{\w}_i(T)) -J(\w^{*})  \right] \ \le \ \left(\frac{1}{m} + \frac{100\sqrt{m\specnorm^2}\cdot \log T}{1-\sqrt{\lambda_2(\bP)}} \right)\cdot \frac{L^2}{\reg} \cdot \frac{\log T}{T}.
\label{eq:morecomm:totalerr}
\end{align}
\label{theorem:mainThrm}
\end{theorem}


\noindent \textit{Remark 1:} Theorem \ref{theorem:mainThrm} indicates that the number of machines should be chosen as a function of $\specnorm$. We can identify three sub-cases of interest:

\textbf{Case (a):} \textit{$m\le \frac{1}{\specnorm^{2/3}}$}: In this regime since $1/m > \sqrt{m\specnorm^2}$ (ignoring the constants and the $\log T$ term) we always benefit from adding more machines.

\textbf{Case (b):} \textit{$\frac{1}{\specnorm^{2/3}} < m\le \frac{1}{\specnorm^2}$}: 
The result tells us that there is no degradation in the error and the bound improves by a factor $\sqrt{m}\specnorm$. Sparse data sets generally have a smaller value of $\specnorm^2$ (as seen in Tak\'{a}\v{c} et al.~\cite{TakacBRS:13icml}); Theorem \ref{theorem:mainThrm} suggests that for such data sets we can use a larger number of machines without losing performance. However the requirements on the number of iterations also increases. This provides additional perspective on the observation by Tak\'{a}\v{c} et al~\cite{TakacBRS:13icml} that sparse datasets are more amenable to parallelization via mini-batching. The same holds for our type of parallelization as well.

\textbf{Case (c):} \textit{$m > \frac{1}{\specnorm^2}$}: 
In this case we pay a pay a penalty $\sqrt{m\specnorm^2} \ge 1$ suggesting that for datasets with large $\specnorm$ we should expect to lose performance even with relatively fewer machines.

Note that $m>1$ is implicit in the condition $T > 2 e  \log(1/\sqrt{ \lambda_2)}) $ since $\lambda_2 =0$ for $m=1$. This excludes the single node Pegasos~\citep{TakacBRS:13icml} case. Additionally in the case of general strongly convex losses (not necessarily dependent on $\w^{\trans}\x$) we can obtain a convergence rate of $\mc{O}(\log^2(T)/T)$. 

\noindent \textit{Remark 2:} The lower bound on the number of iterations \ref{eq:thm_Tbound} can be considerably improved by instead looking at the intrinsic dimension of the data since for several real datasets the intrinsic dimension can be much smaller than the dimension of the ambient space. However this requires us to assume a lower bound on the norm of the data samples, which is a less natural assumption.

\section{Proof of Data Dependent Convergence}
Let  $\mc{F}_t$ be the sigma algebra generated by data and random selections of the algorithm up to time $t$, so that the iterates $\{\w_i(t) : i \in [m]\}$ are measurable with respect to $\mc{F}_t$.  Theorem \ref{theorem:mainThrm} provides a bound on the suboptimality gap for the output $\hat{\w}_i(T)$ of Algorithm \eqref{alg:DiSCO} at node $i$, which is the average of that node's iterates. In the analysis we relate this local average to the average iterate across nodes at time $t$:
	\begin{align}
	\bar{\w}(t)  = \sum_{i=1}^m \frac{\w_i(t)}{m}.
	\label{eq:node_avg_iterate}
	\end{align}
We will also consider the average of $\bar{\w}(t)$ over time. 

The proof consists of three main steps. 
\begin{itemize}
\item We establish the following inequality for the objective error:
	\begin{align}
	\E \left[ J(\bar{\w}(t))- J(\w^{*}) \right] &\le \notag \\
   	&\hspace{-1in}
	\frac{(\step_t^{-1}-\reg)}{2} \E \left[ \norm{\bar{\w}(t)-\w^{*}}^2 \right] \notag \\
   	&\hspace{-1in}
	- \frac{\step_t^{-1}}{2} \E\left[ \norm{\bar{\w}(t+1)-\w^{*}}^2 \right] 		
		\notag \\
	&\hspace{-1in} 
	+ \frac{\step_t}{2} \E\left[ \norm{\sum_{i=1}^{m}\frac{\g_i(t)}{m}}^2 
			\right] \notag \\ 
	&\hspace{-1in} 
	+\sum_{i=1}^{m} 
			\sqrt{\E\left[\norm{ \bar{\w}(t)-\w_i(t) }^2 \right]} \notag \\
	&\hspace{-1in}
	\cdot \sqrt{\E\left[ \left(\norm{\nabla J_i(\w_i(t))} 
				+ \norm{\nabla J_i(\bar{\w}(t))} \right)^2 \right] }/m,
	 \label{eq:mainBnd_before}
	 \end{align}
where $\bar{\w}(t)$ is the average of the iterates at all nodes and the expectation is with respect to $\calF_t$ while conditioned on the sample split across nodes. All expectations, except when explicitly stated, will be conditioned on this split.
 \item We bound $\E\left[\norm{\nabla J(\w_i(t))}^2\right] $ and $\frac{\step_t}{2}\E \left[ \norm{\sum_{i=1}^{m}\frac{\g_i(t)}{m}}^2 \right]$ in terms of the spectral norm of the covariance matrix of the distribution $\mc{P}$ by additionally taking expectation with respect to the sample $S$.
 \item We bound the network error $\E\left[\norm{ \bar{\w}(t)-\w_i(t) }^2 \right]$ in term of the network size $m$ and a spectral property of the matrix $\bP$.
\end{itemize}
Combining the bounds using inequality \eqref{eq:mainBnd_before} and applying the definition of subgradients yields the result of Theorem \ref{theorem:mainThrm}.

\subsection{Spectral Norm of Random Submatrices}
In this section we restate Theorem \ref{theorem:specnorm} proved in Chapter \ref{Chapter2} pertaining to the spectral norm of submatrices that is central to our results. 
\begin{lemma}\label{lem:specnormIntdim}
Let $\mc{P}$ be a distribution on $\mathbb{R}^d$ with second moment matrix $\mbs{\Sigma} = \E_{\Y \sim \mc{P}}[ \Y \Y^{\trans}]$ such that $\norm{\Y_k} \le 1$ almost surely.  Let $\dists^2 = \spec(\mbs{\Sigma})$.  Let $\Y_1, \Y_2, \ldots, \Y_K$ be an i.i.d. sample from $\mc{P}$ and let
	\[
	\Q_K = \sum_{k=1}^{K} \Y_k \Y_k^{\trans}
	\]
be the empirical second moment matrix of the data.  Then for $K > \frac{4}{3\dists^2} \log(d)$,
	\begin{align}
	\E\left[ \frac{ \spec( \Q_K ) }{K} \right] 
	&\le 5 \dists^2. 
	\end{align}
\end{lemma}

Thus when $\mc{P}$ is the empirical distribution we get that $	\E\left[ \frac{ \spec( \Q_K ) }{K} \right] \le 5 \dists^2$.

\subsection{Decomposing the expected suboptimality gap}

The proof in part follows \citep{nedicDistributedOptimization}. It is easy to verify that because $\bP$ is doubly stochastic the average of the iterates across the nodes at time $t$, the average of the iterates across the nodes in \eqref{eq:node_avg_iterate} satisfies the following update rule:
\begin{align}
\bar{\w}(t+1) = \bar{\w}(t) - \step_t\sum_{i=1}^{m}\frac{\g_i(t)}{m}.
\label{eq:avgUp}
\end{align}
We emphasize that in Algorithm \eqref{alg:DiSCO} we do not perform a final averaging across nodes at the end as in \eqref{eq:node_avg_iterate}. Rather, we analyze the average at a single node across its iterates (sometimes called Polyak averaging).  Analyzing \eqref{eq:node_avg_iterate} provides us with a way to understand how the objective $J(\w_i(t))$ evaluated at any node $i$'s iterate $\w_i(t)$ compares to the minimum value $J(\w^*)$.  The details can be found in Section \ref{sec:combo_bound}.

To simplify notation, we treat all expectations as conditioned on the sample $S$. Then \eqref{eq:avgUp},
\begin{align}
\E\left[ \norm{\bar{\w}(t+1)-\w^{*}}^2 \Big| \calF_t \right] & \notag \\
&\hspace{-1.2in}
=\E \left[ \norm{\bar{\w}(t)-\w^{*}}^2 | \calF_t \right] \notag \\
	&\hspace{-0.8in} 
	+ \step_t^2\E \left[ \norm{\sum_{i=1}^{m}\frac{\g_i(t)}{m}}^2 \Big| \calF_t\right] 
	\notag \\
	&\hspace{-0.8in}
	- 2\step_t(\bar{\w}(t)-\w^{*})^{\trans}\sum_{i=1}^{m}\frac{\E \left[ \g_i(t)|\calF_t \right]}{m} 
	\notag \\
&\hspace{-1.2in}
=  \E \left[ \norm{\bar{\w}(t)-\w^{*}}^2 |\calF_t \right] \notag \\
	&\hspace{-0.8in} 
	+ \step_t^2\E \left[ \norm{\sum_{i=1}^{m}\frac{\g_i(t)}{m}}^2 \Big| \calF_t \right] \notag \\
	&\hspace{-0.8in} 
	- 2\step_t\sum_{i=1}^{m}(\bar{\w}(t)-\w^{*})^{\trans}\frac{\E \left[ \g_i(t)|\calF_t \right]}{m}. 
\label{eq:comm_recursion}
\end{align}

Note that $\nabla J_i(\w_i(t)) = \E \left[ \g_i(t)|\calF_t \right]$, so for the last term, for each $i$ we have
\begin{align}
\nabla J_i(\w_i(t))^{\trans}(\bar{\w}(t)-\w^{*}) \notag \\
&\hspace{-1.2in}
= 
\nabla J_i(\w_i(t))^{\trans}\left(\bar{\w}(t)-\w_i(t)\right) 
	\notag \\
	&\hspace{-0.8in}
	+ \nabla J_i(\w_i(t))^{\trans}\left(\w_i(t) - \w^{*}\right) 
	\notag \\
&\hspace{-1.2in} 
\ge 
-\norm{\nabla J_i(\w_i(t))} \norm{\bar{\w}(t)-\w_i(t)} 	
	\notag \\
	&\hspace{-0.8in}
	+ \nabla J_i(\w_i(t))^{\trans}\left(\w_i(t) - \w^{*}\right) 
	\notag \\
&\hspace{-1.2in}
\ge
-\norm{\nabla J_i(\w_i(t))} \norm{\bar{\w}(t)-\w_i(t)} 
	\notag \\
	&\hspace{-0.8in}
	+ J_i(\w_i(t)) - J_i(\w^*)  
	+ \frac{\reg}{2}\norm{\w_i(t)-\w^{*}}^2
	\notag \\
&\hspace{-1.2in}
= 
-\norm{\nabla J_i(\w_i(t))} \norm{\bar{\w}(t)-\w_i(t)} 
	\notag \\
	&\hspace{-0.8in}
	+ J_i(\w_i(t)) - J_i(\bar{\w}(t)) 
	\notag \\
  	&\hspace{-0.8in} 
	+ \frac{\reg}{2}\norm{\w_i(t)-\w^{*}}^2 
	+  J_i(\bar{\w}(t)) - J_i(\w^*) 
	\notag \\
&\hspace{-1.2in}
\ge 
	-\norm{\nabla J_i(\w_i(t))} \norm{\bar{\w}(t)-\w_i(t)} 
	\notag \\
	&\hspace{-0.8in}
	+\nabla J_i(\bar{\w}(t))^{\trans}\left(\w_i(t) - \bar{\w}(t)\right) 
	\notag \\
  	&\hspace{-0.8in}
	+ \frac{\reg}{2}\norm{\w_i(t)-\w^{*}}^2 
	+  J_i(\bar{\w}(t)) - J_i(\w^*) 
	\notag \\
&\hspace{-1.2in}
\ge 
	-\left(\norm{\nabla J_i(\w_i(t))}
		+\norm{\nabla J_i(\bar{\w}(t))}\right) \norm{\bar{\w}(t)-\w_i(t)} 
	\notag \\
 	&\hspace{-0.8in}
	+ \frac{\reg}{2}\norm{\w_i(t)-\w^{*}}^2 
	+ J_i(\bar{\w}(t)) - J_i(\w^*),
\end{align}
where the second and third lines comes from applying the Cauchy-Shwartz inequality and strong convexity, the fifth line comes from the definition of subgradient, and the last line is another application of the Cauchy-Shwartz inequality.

Averaging over all the nodes, using convexity of $\norm{\cdot}^2$, the definition of $J(\cdot)$, and Jensen's inequality yields the following inequality:
\begin{align}
-2 \step_t \sum_{i=1}^{m} (\bar{\w}(t)-\w^{*})^{\trans} \frac{\E[\g_i(t)|\calF_t]}{m} 	
	\notag \\
&\hspace{-2in}
\le 
	2\step_t  
	\sum_{i=1}^{m}
		\frac{
			\norm{ \bar{\w}(t)-\w_i(t) } 
			\left (\norm{\nabla J_i(\w_i(t))} 
				+ \norm{\nabla J_i(\bar{\w}(t))} \right )
			}{m}
	\notag \\
	&\hspace{-1.7in}
	-2\step_t
	\left( \sum_{i=1}^m \frac{J_i(\bar{\w}(t))- J_i(\w^{*})}{m} \right) 
	\notag \\
	&\hspace{-1.7in}
	-\reg\step_t \sum_{i=1}^{m}\frac{\norm{\w_i(t)-\w^{*}}^2}{m} 
	\notag \\
&\hspace{-2in}
\le
	2\step_t 
	\sum_{i=1}^{m}\frac{\norm{ \bar{\w}(t)-\w_i(t) } \left (\norm{\nabla J_i(\w_i(t))} + \norm{\nabla J_i(\bar{\w}(t))} \right )}{m}
	\notag \\
	&\hspace{-1.7in}
	-2\step_t\left( J(\bar{\w}(t))- J(\w^{*}) \right) - \reg\step_t \norm{\bar{\w}(t)-\w^{*}}^2
\label{eq:node_average}
\end{align}

Substituting inequality \eqref{eq:node_average} in recursion \eqref{eq:comm_recursion}, 
	\begin{align}
	\E\left[ \norm{\bar{\w}(t+1)-\w^{*}}^2 \big| \calF_t \right] 
	\notag \\
	&\hspace{-1.6in}
	\le
		\E \left[ \norm{\bar{\w}(t)-\w^{*}}^2 |\calF_t \right] 
		\notag \\
  		&\hspace{-1.5in}
		+ \step_t^2 \E\left[ \norm{\sum_{i=1}^{m}\frac{\g_i(t)}{m}}^2 
			~\Big|~ \calF_t\right] 
		\notag \\
		&\hspace{-1.5in} 
		+ 2 \step_t
		\sum_{i=1}^{m}
			\frac{
				\norm{ \bar{\w}(t)-\w_i(t) } 
				\left(\norm{\nabla J_i(\w_i(t))} 
					+ \norm{\nabla J_i(\bar{\w}(t))}
				\right)
			}{m} 
		\notag \\
		&\hspace{-1.5in} 
 		-2 \step_t \left( J(\bar{\w}(t))- J(\w^{*}) \right) 
		- \reg \step_t \norm{\bar{\w}(t)-\w^{*}}^2.
	\label{eq:recurse1}
	\end{align}
Taking expectations with respect to the entire history $\mc{F}_t$,
\begin{align}
\E\left[ \norm{\bar{\w}(t+1)-\w^{*}}^2 \right]  
	\notag \\
&\hspace{-1.3in}
\le
	\E\left[ \norm{\bar{\w}(t)-\w^{*}}^2 \right] 
	 + \step_t^2\E \left[ \norm{\sum_{i=1}^{m}\frac{\g_i(t)}{m}}^2 \right] 	
	\notag \\ 
	&\hspace{-1.2in}
	+ 2\step_t \cdot 
	\notag \\
	&\hspace{-1.1in}
		\sum_{i=1}^{m} \frac{
			\E\left[\norm{ \bar{\w}(t)-\w_i(t) } 
				\left(\norm{\nabla J_i(\w_i(t))} 
				+ \norm{\nabla J_i(\bar{\w}(t))} \right)
				\right]
			}{m} 
	\notag \\ 
	&\hspace{-1.2in}
	-2\step_t \left( \E\left[ J(\bar{\w}(t))- J(\w^{*}) \right] \right)
			- \reg\step_t \E \left[ \norm{\bar{\w}(t)-\w^{*}}^2 \right] 
	\notag \\ 
&\hspace{-1.3in}
\le 
	-2 \step_t \left( \E \left[ J(\bar{\w}(t))- J(\w^{*}) \right] \right)
	\notag \\ 
	&\hspace{-1.2in}
	+ (1 - \reg \step_t) \E\left[ \norm{\bar{\w}(t)-\w^{*}}^2 \right] 
	+ \step_t^2\E \left[ \norm{\sum_{i=1}^{m}\frac{\g_i(t)}{m}}^2 \right] 
	\notag \\ 
	&\hspace{-1.2in}
	+ \frac{2\step_t}{m} 
			 \sum_{i=1}^{m} 
			\sqrt{\E\left[\norm{ \bar{\w}(t)-\w_i(t) }^2 \right]} 
			\notag \\ 
			&\hspace{-0.6in}
			\cdot \sqrt{\E\left[ \left(\norm{\nabla J_i(\w_i(t))} 
				+ \norm{\nabla J_i(\bar{\w}(t))} \right)^2 \right] }
\end{align}

This lets us bound the expected suboptimality gap $\E \left[ J(\bar{\w}(t))- J(\w^{*}) \right]$ via three terms:
	\begin{align}
	\text{T1}
	&= \frac{(\step_t^{-1}-\reg)}{2}
			\E \left[ \norm{\bar{\w}(t)-\w^{*}}^2 \right]
		\notag \\
		&\hspace{0.5in} 
		- \frac{\step_t^{-1}}{2}
			\E\left[ \norm{\bar{\w}(t+1)-\w^{*}}^2 \right]
	\label{eq:T1}
	\\
	\text{T2}
	&=
		\frac{\step_t}{2}\E \left[ \norm{\sum_{i=1}^{m}\frac{\g_i(t)}{m}}^2 \right]
	\label{eq:T2}
	\\
	\text{T3}
	&= \frac{1}{m} \sum_{i=1}^{m} 
			\sqrt{\E\left[\norm{ \bar{\w}(t)-\w_i(t) }^2 \right]} 
		\notag \\
		&\hspace{0.5in} 
		\cdot 
		\sqrt{\E\left[ \left (\norm{\nabla J_i(\w_i(t))} 
				+ \norm{\nabla J_i(\bar{\w}(t))} \right)^2 \right] },
	\label{eq:T3}
	\end{align}
where
\begin{align}
\E \left[ J(\bar{\w}(t))- J(\w^{*}) \right]  
&\le \text{T1} + \text{T2} + \text{T3}.
 \label{eq:mainBnd}
 \end{align}
The remainder of the proof is to bound these three terms separately.

\subsection{Network Error Bound} 
We need to prove an intermediate bound first to handle term \text{T3}.

\begin{lemma}\label{lemma:avgdevBnd}
Fix a Markov matrix $\bP$ and consider Algorithm \ref{alg:DiSCO} when the objective $J(\w)$ is strongly convex \removed{and the number of iterations \$T$ satisfies
\begin{align}
	T &> 2 e  \log(1/\sqrt{ \lambda_2(\bP)}) 
	\end{align} }
we have the following inequality for the expected squared error between the iterate $\w_i(t)$ at node $i$ at time $t$ and the average $\bar{\w}(t)$ defined in Algorithm \ref{alg:DiSCO}:
	\begin{align}
	\sqrt{\E\left[\norm{ \bar{\w}(t)-\w_i(t) }^2 \right]} \le \frac{2L}{\reg}\cdot \frac{\sqrt{m}}{b }\cdot\frac{\log(2bet^2)}{t},
	\end{align}
where $b= (1/2)\log(1/\lambda_2(\bP))$.
\end{lemma}

\begin{proof}
We follow a similar analysis as others \citep[Prop. 3]{nedicDistributedOptimization}~\citep[IV.A]{dualAveraging}~\citep{DistStronglyConvex}.  Let $\mbf{W}(t)$ be the $m \times d$ matrix whose $i$-th row is $\w_i(t)$ and $\mbf{G}(t)$ be the $m \times d$ matrix whose $i$-th row is $\g_i(t)$ .  Then the iteration can be compactly written as
	\[
	\mbf{W}(t+1) = \bP(t) \mbf{W}(t) - \step_t \mbf{G}(t)
	\]
and the network average matrix $\bar{\mbf{W}}(t) = \frac{1}{m} \mbf{1} \mbf{1}^{\top} \mbf{W}(t)$.  Then we can write the difference using the fact that $\bP(t) = \bP$ for all $t$:
	\begin{align}
	&\bar{\mbf{W}}(t+1) - \mbf{W}(t+1) =\notag \\
		&\left(  \frac{1}{m} \mbf{1} \mbf{1}^{\top} - I \right) \left( \bP \mbf{W}(t) - \step_t \mbf{G}(t) \right) \notag \\
			&\hspace{-1in} \notag \\
		&= \left(  \frac{1}{m} \mbf{1} \mbf{1}^{\top} - \bP \right) \mbf{W}(t)
			- \step_t \left(  \frac{1}{m} \mbf{1} \mbf{1}^{\top} - I \right) \mbf{G}(t) \notag \\
		&\hspace{-1in} \notag \\
		&= \left(  \frac{1}{m} \mbf{1} \mbf{1}^{\top} - \bP \right) \left( \bP \mbf{W}(t-1) - \step_{t-1} \mbf{G}(t-1) \right)
			\notag \\
			&\quad - \step_t \left(  \frac{1}{m} \mbf{1} \mbf{1}^{\top} - I \right) \mbf{G}(t) \notag \\
				&\hspace{-1in} \notag \\
		&= \left(  \frac{1}{m} \mbf{1} \mbf{1}^{\top} - \bP^2 \right) \mbf{W}(t-1) \notag \\
			&\quad- \step_{t-1} \left(  \frac{1}{m} \mbf{1} \mbf{1}^{\top} - \bP \right) \mbf{G}(t-1)
			\notag \\
			&\quad- \step_t \left(  \frac{1}{m} \mbf{1} \mbf{1}^{\top} - I \right) \mbf{G}(t) \notag \\
				&\hspace{-1in} \notag \\
		&= \left(  \frac{1}{m} \mbf{1} \mbf{1}^{\top} - \bP^2 \right) \mbf{W}(t-1)\notag \\
		&\quad- \sum_{s=t-1}^{t} \step_s \left(  \frac{1}{m} \mbf{1} \mbf{1}^{\top} - \bP^{t-s} \right) \mbf{G}(s).
	\end{align}
Continuing the expansion and using the fact that $\mbf{W}(1) = \mbf{0}$,
	\begin{align}
   &	\bar{\mbf{W}}(t+1) - \mbf{W}(t+1) = \notag \\
	&\left( \frac{1}{m} \mbf{1} \mbf{1}^{\top} - \bP^t \right)  \mbf{W}(1)
		- \sum_{s=1}^{t} \step_s \left(  \frac{1}{m} \mbf{1} \mbf{1}^{\top} - \bP^{t-s} \right) \mbf{G}(s) \notag \\
		&\quad= - \sum_{s=1}^{t} \step_s \left(  \frac{1}{m} \mbf{1} \mbf{1}^{\top} - \bP^{t-s} \right) \mbf{G}(s) \notag \\
		&\quad= - \sum_{s=1}^{t-1} \step_s \left(  \frac{1}{m} \mbf{1} \mbf{1}^{\top} - \bP^{t-s} \right) \mbf{G}(s) \notag \\
		&\quad- \step_t \left(  \frac{1}{m} \mbf{1} \mbf{1}^{\top} - I \right) \mbf{G}(t).
			\label{eq:more:neterr_matrix}
	\end{align}
Now looking at the norm of the $i$-th row of \eqref{eq:more:neterr_matrix} and using the bound on the gradient norm:
	\begin{align}
	&\norm{ \bar{\w}(t)-\w_i(t) } \notag \\
		&\quad \le \Bigg\| 
			\sum_{s=1}^{t-1} \step_s \sum_{j=1}^{m} 
				\left( \frac{1}{m} - (\bP^{t-s})_{ij} \right) \g_j(s) 
			\notag \\
			&\hspace{1in}
			+ \step_t \left( \sum_{j=1}^{m} \frac{1}{m} \g_j(t) - \g_i(t) \right) 
				\Bigg\| \\
		&\le \sum_{s=1}^{t-1} \frac{L}{\reg s} \cdot \norm{ \frac{1}{m} - (\bP^{t-s})_{i} }_1 + \frac{2L}{\reg t}.
		\label{eq:devBnd}
	\end{align}
	
We handle the term $\norm{ \frac{1}{m} - (\bP^{t-s})_{i} }_1$ using a bound on the mixing rate of Markov chains (c.f. (74) in Tsianos and Rabbat~\cite{DistStronglyConvex}):
	\begin{align}
	\sum_{s=1}^{t-1} \frac{L}{\reg s} \cdot \norm{ \frac{1}{m} - (\bP^{t-s})_{i} }_1 
		&\le \frac{L \sqrt{m}}{\reg}  \sum_{s=1}^{t-1} \left( \sqrt{ \lambda_2(\bP)} \right)^{t - s} \frac{1}{s}.
	\label{eq:mixingrate}
	\end{align}

Define $a = \sqrt{\lambda_2(\bP)} \le 1$ and $b = - \log(a) > 0$.  Then we have the following identities:
	\begin{align}
	\sum_{\tau=1}^t \frac{a^{t-\tau+1}}{\tau} 
	= \sum_{\tau=1}^t \frac{a^\tau}{t-\tau+1}
	=  \sum_{\tau=1}^t \frac{\exp(-b\tau)}{t-\tau+1}.
	\end{align}
Now using the fact that when $x > -1$ we have $\exp(-x)< 1/(1+x)$ and using the integral upper bound we get
	\begin{align}
   &	\sum_{\tau=1}^t \frac{a^{t-\tau+1}}{\tau} \notag \\
   	&\le
	\sum_{\tau=1}^t \frac{1}{(1+b\tau)(t-\tau+1)} \notag \\ 
	&\le  
	\frac{1}{(1+b)t} + \int_{1}^{t} \frac{d\tau}{(1+b\tau)(t-\tau+1)} \notag \\
 	&=  
	\frac{1}{(1+b)t} + \left[ \frac{\log(b\tau+1)
		-\log(t-\tau+1)}{bt+b+1} \right]_{\tau=1}^{t} 
		\notag \\
 	&=  
	\frac{1}{(1+b)t} + \frac{\log(bt+1)-\log(b+1)+\log(t)}{bt+b+1} \notag \\
 	&\le   
	\frac{\log(et(bt+1))}{bt} \notag \\
 	&\le   \frac{\log(2bet^2)}{bt}.
	\label{eq:seriesBnd}
	\end{align}
Using \eqref{eq:mixingrate} and \eqref{eq:seriesBnd} in \eqref{eq:devBnd} we get
	\begin{align}
	\norm{\bar{\w}(t)-\w_i(t)} 
	&\le \frac{L\sqrt{m}}{\reg}\frac{\log(2bet^2)}{bt} + \frac{2L}{ \reg t} 
	\notag \\
	&\le \frac{2L\sqrt{m}}{\reg}\frac{\log(2bet^2)}{bt}.
	\label{eq:networkDevBnd}
	\end{align}
Therefore we have 
	\begin{align}
	\sqrt{\E\left[\norm{ \bar{\w}(t)-\w_i(t) }^2 \right]}
		&\le \frac{2L\sqrt{m}}{\reg}\frac{\log(2bet^2)}{bt}.
	\label{eq:avgtoind:normbnd1}
\end{align}

\end{proof}

\emph{Remark:} The network lemma ignores the fact that the data is i.i.d and this leads to a loose dependence on $m$. This is the main weakness of our analysis and in the future we intend to prove a stronger result.

\subsection{Bounds for expected gradient norms \label{subsec:gradBnds}}

\subsubsection{Bounding Gradient at Averaged Iterate}

Let $\sgi{j}{t} \in \partial \ell(\bar{\w}(t)^{\trans}\x_{i,j})$ denote a subgradient for the $j$-th point at node $i$ and $\SGI{t} = (\sgi{1}{t}, \sgi{2}{t}, \ldots, \sgi{n}{t})^{\trans}$ be the vector of subgradients at time $t$.  Let $\Q_{S_i}$ be the $n \times n$ Gram matrix of the data set $S_i$.  From the definition of $\norm{\nabla J_i(\bar{\w}(t))}$ and using the Lipschitz property of the loss functions, we have the following bound:
	\begin{align}
   	&\norm{\nabla J_i(\bar{\w}(t))}^2 
	\notag \\
	&\le \norm{\sum_{j \in S_i }\frac{\sgi{j}{t}\x_{i,j}}{n} + \reg \bar{\w}(t)}^2 
		\notag \\
	&\le
	2 \norm{\sum_{j\in S_i } \frac{\sgi{j}{t} \x_{i,j}}{n}}^2 
		+ 2\reg^2 \norm{\bar{\w}(t)}^2 
		\notag \\
	&= \frac{ 2 \sum_{j \in S_i }\sum_{j^{\prime}\in S_i } 
		\sgi{j}{t} \sgi{j^{\prime}}{t} \x_{i,j}^{\trans} \x_{i,j}^{\prime}  
		}{n^2} 
		+ 2\reg^2 \norm{\bar{\w}(t)}^2 \notag \\
&= \frac{ 2 }{n^2} \SGI{t}^{\trans} \Q_{S_i} \SGI{t}  + 2\reg^2 \norm{\bar{\w}(t)}^2 \notag \\
&\le \frac{2}{n^2} \norm{ \SGI{t} }^2 \spec(\Q_{S_i}) + 2\reg^2 \norm{\bar{\w}(t)}^2 \notag \\
&\le  2L^2 \frac{\spec(\Q_{S_i})}{n} + 2\reg^2 \norm{\bar{\w}(t)}^2.
\label{eq:basicGradBnd}
\end{align}

We rewrite the update \eqref{eq:avgUp} in terms of $\{ \x_{i,t} \}$, the points sampled at the nodes at time $t$:
	\begin{align}
	\bar{\w}(t+1) 
	= \bar{\w}(t)(1-\reg \step_t) 
		- \step_t \sum_{i=1}^{m} 
			\frac{\partial \ell(\w_i(t)^{\trans}\x_{i,t})\x_{i,t}}{m}.
	\label{eq:avgUprecurse}
	\end{align}
Now from equation \eqref{eq:avgUprecurse}, after unrolling the recursion as in Shalev-Shwarz et al.~\cite{SSSC11:pegasos} we see
	\begin{align}
	\bar{\w}(t) 
	= \frac{1}{\reg(t-1)} \sum_{\tau=1}^{t-1} 
		\frac{
		\sum_{i=1}^{m}\partial \ell(\w_i(\tau)^{\trans}\x_{i,\tau})\x_{i,\tau}
		}{m}.
\end{align}
Let $\gamma^i_{\tau} \in \partial \ell(\w_i(\tau)^{\trans}\x_{i,\tau})$ the subgradient set for the $i$th node computed at time $\tau$, then we have
	\begin{align}
	\norm{\bar{\w}(t)} 
	\le 
	\frac{1}{\reg(t-1)} \cdot \frac{1}{m} \sum_{i=1}^{m}  
		\norm{\sum_{\tau=1}^{t-1}\gamma^i_{\tau} \x_{i,\tau}}.
\label{eq:avg:normbnd1}
\end{align}

Let us in turn bound for each node $i$ the term $\norm{\sum_{\tau=1}^{t-1}\gamma^i_{\tau} \x_{i,\tau}}$. Let $\sgt{\tau}^{i} \in \partial \ell(\w_i(\tau)^{\trans}\x_{i,{\tau}})$ denote a subgradient for the point sampled at time $\tau$ at node $i$ and $\SGT^{i} = (\sgt{1}^{i}, \sgt{2}^{i}, \ldots, \sgt{t-1}^{i})^{\trans}$ be the vector of subgradients up to time $t-1$.  We have
	\begin{align}
	\norm{\sum_{\tau=1}^{t-1}\gamma^i_{\tau} \x_{i,\tau}}^2 
	&= \sum_{\tau,\tau^{\prime}} 
		\gamma^i_{\tau} \gamma^i_{\tau^{\prime}} 
		\x_{i,\tau}^{\trans}\x_{i,\tau^{\prime}} 
		\notag \\
	&= ( \mbs{\gamma}^{i} )^{\trans} \Q_{i,t-1} \mbs{\gamma}^{i} 
		\notag \\
	&\le \norm{\mbs{\gamma}^{i}}^2 \spec(\Q_{i,t-1}) 
		\notag \\
	&\le (t-1) L^2\spec(\Q_{i,t-1}),
\end{align}
where $\Q_{i,t-1}$ is the $(t-1) \times (t-1)$ Gram submatrix  corresponding to the points sampled at the $i$-th node until time $t-1$.

Further bounding \eqref{eq:avg:normbnd1}:
	\begin{align*}
	\norm{\bar{\w}(t)}^2 
	&\le \left(\frac{1}{\reg(t-1)} 
		\frac{\sum_{i=1}^{m} \sqrt{(t - 1) L^2\spec(\Q_{i,t-1})}}{m}\right)^2  
		\notag \\
	&\le \frac{L^2}{\reg^2} \left( \frac{1}{m} 
		\sum_{i=1}^{m} \sqrt{ \frac{\spec(\Q_{i,t-1})}{t-1} }
		\right)^2.
	\end{align*}
Since as stated before everything is conditioned on the sample split we take expectations w.r.t the history and the random split and using the Cauchy-Schwarz inequality again, and the fact that the points are sampled i.i.d. from the same distribution,
	\begin{align}
	&\E\left[ \norm{\bar{\w}(t)}^2 \right] 
	\notag \\
	&\le 
	\frac{L^2}{\reg^2} \frac{1}{m^2} \sum_{i=1}^{m} \sum_{j=1}^{m}
		\E\left[  \frac{ \sqrt{\spec(\Q_{i,t-1})\spec(\Q_{j,t-1}) } }{t-1}
			\right] \notag \\
	&\le \frac{L^2}{\reg^2} \frac{1}{m^2}
		\sum_{i=1}^{m} \sum_{j=1}^{m}
		\sqrt{ \E\left[ \frac{ \spec(\Q_{i,t-1})  }{t-1} \right]
			\E\left[ \frac{ \spec(\Q_{j,t-1}) }{t-1} \right] } \notag \\
	&= \frac{L^2}{\reg^2} \E\left[ \frac{ \spec(\Q_{i,t-1})  }{t-1} \right] .
	\label{eq:avg:normbnd2}
	\end{align}
The last line follows from the expectation over the sampling model: the data at node $i$ and node $j$ have the same expected covariance since they are sampled uniformly at random from the total data.

Taking the expectation in \eqref{eq:basicGradBnd} and substituting \eqref{eq:avg:normbnd2} we have
	\begin{align}
   	\E\left[\norm{\nabla J_i(\bar{\w}(t))}^2\right] 
	&\le 
		2L^2 \E\left[\frac{\spec(\Q_{S_i})}{n}\right] 
		\notag \\
		&\hspace{0.5in}
		+ 2L^2 \E\left[\frac{\spec(\Q_{i,t-1})}{t-1} \right].
	\end{align}
Since $S_i$ is a uniform random draw from $S$ and by assuming both $t$ and $n$ to be greater than $4/(3\specnorm^2)\log (d)$, applying Lemma \ref{lem:specnormIntdim} gives us
	\begin{align}
	\E\left[\norm{\nabla J_i(\bar{\w}(t))}^2\right] \le 20L^2\specnorm^2.
	\label{eq:itravgnrmBnd}
	\end{align}

\subsubsection{Bounding Gradient at any Node}

We have just as in the previous subsection
\begin{align*}
\norm{\nabla J_i(\w_i(t))}^2 &\le  2L^2 \frac{\spec(\Q_{S_i})}{n} + 2\reg^2 \norm{\w_i(t)}^2.
\end{align*}
Using the triangle inequality, the fact that $(a_1+a_2)^2 \le 2a_1^2 + 2a_2^2$, the bounds \eqref{eq:networkDevBnd} and \eqref{eq:avg:normbnd2}, and Lemma \ref{lem:specnormIntdim}:
\begin{align}
\E\left[\norm{\w_i(t)}^2\right] &\le 2\E\left[\norm{\w_i(t) - \bar{\w}(t)}^2\right]  + 2\E\left[\norm{\bar{\w}(t)}^2\right] \notag \\
&\le  \frac{8L^2m}{\reg^2}\frac{\log^2(2bet^2)}{b^2 (t-1)^2} + \frac{5L^2\specnorm^2}{\reg^2}.
\label{eq:netBnd}
\end{align}
From \eqref{eq:netBnd} we can infer that for the second term to dominate the first we require
\begin{align*}
\frac{t}{\log(t)} > \sqrt{ \frac{8}{5} } \frac{\sqrt{m}}{\specnorm b}.
\end{align*}
This gives us
\begin{align}
\E\left[\norm{\w_i(t)}^2\right] &\le \frac{10L^2\specnorm^2}{\reg^2},
	\label{eq:localiter:norm}
\end{align}
and therefore
\begin{align}
\E\left[ \norm{\nabla J_i(\w_i(t))}^2 \right] \le 30L^2\specnorm^2.
\label{eq:itrnrmBnd}
\end{align}

\subsection{Intermediate Bound - $1$}

Because the gradients are bounded,
\begin{align*}
&\E \left[ \norm{\sum_{i=1}^{m}\frac{\g_i(t)}{m}}^2 \right] 
\notag \\
&=\E \left[\sum_{i,j}\frac{\g_i(t)^{\trans}\g_i(t)}{m^2} \right] \notag \\
&=\sum_{i=1}^m \frac{\E \left[ \norm{\g_i(t)}^2 \right]}{m^2} + \sum_{i\ne j}\frac{\E \left[\g_i(t)^{\trans}\g_j(t)\right]}{m^2}  \\
&\le \frac{L^2}{m} +  \sum_{i\ne j}\frac{\E \left[\g_i(t)^{\trans}\g_j(t)\right]}{m^2}  \\
&= \frac{L^2}{m} +\frac{\sum_{i \ne j} \E_{\calF_{t-1}} \left[ \E\left[\g_i(t)^{\trans}\g_j(t) |  \calF_{t-1} \right] \right]}{m^2}.
\end{align*}
Now using the fact that the gradients $\g_i(t)$ are unbiased estimates of $\nabla J_i(\w_t)$ and that $\g_i(t)$ and $\g_j(t)$ are independent given past history and inequality \eqref{eq:itrnrmBnd} for node $i$ and $j$ we get
\begin{align}
&\frac{\sum_{i \ne j} \E_{\calF_{t-1}} \left[ \E\left[\g_i(t)^{\trans}\g_j(t) |  \calF_{t-1} \right] \right]}{m^2} 
\notag \\ 
&=\sum_{i \ne j} \frac{\E_{\calF_{t-1}} \left[\nabla J_i(\w_i(t))^{\trans}\nabla J_j(\w_j(t))  \right]}{m^2} \notag \\
&\le \sum_{i \ne j} \frac{\sqrt{\E_{\calF_{t-1}} \left[\norm{\nabla J_i(\w_i(t))}^2\right]}\sqrt{\E_{\calF_{t-1}} \left[\norm{\nabla J_j(\w_j(t))}^2\right]}}{m^2} \notag \\
&
= \frac{(m-1)}{m} \cdot 30L^2\specnorm^2\notag \\
&
\le 30L^2\specnorm^2.
\end{align}
Therefore to bound the term $\mathrm{T2}$ in \eqref{eq:mainBnd} we can use
\begin{align}
\E \left[ \norm{\sum_{i=1}^{m}\frac{\g_i(t)}{m}}^2 \right] 
	\le \frac{L^2}{m} + 30L^2\specnorm^2.
\label{eq:T2Bnd}
\end{align}

\subsection{Intermediate Bound - $2$ }

Applying \eqref{eq:avgtoind:normbnd1}, \eqref{eq:itravgnrmBnd}, and \eqref{eq:itrnrmBnd} to $\mathrm{T3}$ in \eqref{eq:mainBnd}, as well as Lemma \ref{lemma:avgdevBnd} and the fact that $(a_1+a_2)^2\le 2a_1^2 + 2a_2^2$ we obtain the following bound:
\begin{align}
T3 &\le \frac{1}{m} \sum_{i=1}^{m} 
	\sqrt{\E\left[\norm{ \bar{\w}(t)-\w_i(t) }^2 \right]} 
	\notag \\
	&\hspace{0.5in}
	\cdot \sqrt{\E\left[ \left(\norm{\nabla J_i(\w_i(t))} 
				+ \norm{\nabla J_i(\bar{\w}(t))} \right)^2 \right] }
	\notag \\
 &\le \frac{1}{m} \sum_{i=1}^{m} \frac{2L\sqrt{m}}{\mu}\frac{\log(2bet^2)}{bt} \cdot 10L \specnorm \notag \\
 &\le   \frac{20L^2}{\mu} \cdot \frac{\sqrt{m}}{b} \cdot \frac{\log(T)}{t} \cdot \specnorm.
\label{eq:T3Bnd}
\end{align}

\subsection{Combining the Bounds \label{sec:combo_bound}}

Finally combining \eqref{eq:T2Bnd} and \eqref{eq:T3Bnd} in \eqref{eq:mainBnd} and applying the step size assumption $\step_t = 1/(\reg t)$:
\begin{align}
&\E \left[ J(\bar{\w}(t))- J(\w^{*}) \right]  
\notag \\
&\le 
	\frac{(\step_t^{-1}-\reg)}{2}\E \left[ \norm{\bar{\w}(t)-\w^{*}}^2 \right] 
	\notag \\
	&\qquad 
	- \frac{\step_t^{-1}}{2}\E\left[ \norm{\bar{\w}(t+1)-\w^{*}}^2 \right]
	\notag \\
	&\qquad 
	+  \left(\frac{30L^2\specnorm^2}{\reg} + \frac{L^2}{\reg m}\right)\cdot \frac{1}	{t}\notag \\
	&\qquad 
	+  \frac{20L^2}{\reg} \cdot \frac{\sqrt{m}}{b} \cdot \frac{\log(2bet^2)}{t} \cdot \specnorm 
\notag \\
&\le 
	\frac{\reg(t-1)}{2}\E \left[ \norm{\bar{\w}(t)-\w^{*}}^2 \right]  
	\notag \\
	&\qquad
	- \frac{\reg t}{2}\E\left[ \norm{\bar{\w}(t+1)-\w^{*}}^2 \right] 
	+ K_0 \cdot \frac{L^2}{\reg t},
 \label{eq:mainBnd2}
 \end{align}
where $K_0 =  \left(30 \specnorm^2 + 1/m + \left(60 \cdot \sqrt{m\specnorm^2} \cdot \log(T)\right)/b \right)$, using $t\le T$ and assuming $T>2be$.

Let us now define two new sequences, the average of the average of iterates over nodes from $t=1$ to $T$ and the average for any node $i\in [m]$ 
\begin{align}
\hat{\w}(T) &= \frac{1}{T}  \sum_{t=1}^T \bar{\w}(t) \\
\hat{\w}_i(T) &= \frac{1}{T} \sum_{t=1}^T \w_i(t).
\end{align} 
Then summing \eqref{eq:mainBnd2} from $t=1$ to $T$, using the convexity of $J$ and collapsing the telescoping sum in the first two terms of \eqref{eq:mainBnd2},
	\begin{align}
&\E \left[ J(\hat{\w}(T))- J(\w^{*}) \right]  \notag \\
&\le \frac{1}{T} \sum_{t=1}^T \E \left[ J(\bar{\w}(t))- J(\w^{*}) \right] 
	 \notag \\
&\le - \frac{\reg T}{2} \E\left[ \norm{\bar{\w}(T+1)-\w^{*}}^2 \right] + K_0 \cdot \frac{L^2}{\reg} \cdot \frac{\sum_{t=1}^{T} 1/t}{T}  	\notag \\
&\le  K_0 \cdot \frac{L^2}{\reg} \cdot \frac{\log(T)}{T}.
\label{eq:avgofavgBnd} 
\end{align}

Now using the definition of subgradient, Cauchy-Schwarz, and Jensen's inequality we have 
\begin{align}
	&J(\hat{\w}_i(T)) -J(\w^{*}) \notag \\
   &\le	J(\hat{\w}(T))-J(\w^{*}) + \nabla J(\hat{\w}_i(T))^{\trans}(\hat{\w}_i(t) - \hat{\w}(T)) 
	\notag \\
	&\le
	J(\hat{\w}(T))-J(\w^{*}) + \norm{\nabla J(\hat{\w}_i(T)}\norm{\hat{\w}_i(t) - \hat{\w}(T)} 
	\notag \\
	&\le
	J(\hat{\w}(T))-J(\w^{*})  
	\notag \\
	&\hspace{0.5in}
	+ \norm{\nabla J(\hat{\w}_i(T))} \cdot \sum_{t=1}^T \frac{\norm{\w_i(t)-\bar{\w}(t)}}{T}.
	\label{eq:combine:bound1}
\end{align}
To proceed we must bound $\E\left[ \norm{\nabla J(\hat{\w}_i(T))}^2 \right]$ in a similar way as the bound \eqref{eq:itravgnrmBnd}.  First, let $\alpha_{i} = \partial \ell( \hat{\w}_i(T)^{\trans} \x_i)$ denote the subgradient for the $i$-th loss function of $J(\cdot)$ in \eqref{eq:optForm}, evaluated at $\hat{\w}_i(T)$, and $\mbs{\alpha}_T = (\alpha_1, \alpha_2, \ldots, \alpha_N)^{\trans}$ be the vector of subgradients.  As before,
	\begin{align*}
	\norm{ \nabla J(\hat{\w}_i(T)) }^2
	&= \norm{ \frac{1}{N} \sum_{i=1}^{N} \alpha_i \x_i + \reg \hat{\w}_i(T) }^2 \\
	&\le \frac{2}{N^2} \mbs{\alpha}^{\trans} \Q \mbs{\alpha} + 2 \reg^2 \norm{ \hat{\w}_i(T) }^2 \\
	&\le 10 L^2 \specnorm^2 + 2 \reg^2 \norm{ \hat{\w}_i(T) }^2 \\
	&\le 10 L^2 \specnorm^2 + 2 \reg^2 \frac{1}{T} \sum_{t=1}^{T} \norm{ \w_i(t) }^2. 
	\end{align*}
Taking expectations of both sides and using \eqref{eq:localiter:norm} as before:
	\begin{align*}
	\E\left[ \norm{ \nabla J(\hat{\w}_i(T)) }^2 \right]
	&\le 30 L^2 \specnorm^2 .
	\end{align*}
Taking expectations of both sides of \eqref{eq:combine:bound1} and using the Cauchy-Schwarz inequality, \eqref{eq:avgofavgBnd}, the preceding gradient bound, Lemma \ref{lemma:avgdevBnd} and the definition of $K_0$ we get
	\begin{align}
	&\E\left[ J(\hat{\w}_i(T)) -J(\w^{*})  \right] \notag \\
	&\le
	K_0 \cdot \frac{ L^2}{\reg} \cdot \frac{\log (T)}{T}
		+ \frac{2 \sqrt{30} L^2}{\reg} \cdot \frac{\sqrt{m}}{b} \cdot \specnorm \cdot \frac{\log(T)}{T} \cdot \sum_{t=1}^T\frac{1}{t}	\notag \\
	&\le \left(K_0 + \frac{2\sqrt{30}\cdot \sqrt{m\specnorm^2} \cdot \log T}{b}  \right)\cdot \frac{\log T}{T} \notag \\
	&\le \left(30\specnorm^2 + \frac{1}{m} + \frac{70\sqrt{m\specnorm^2}\cdot \log T}{b} \right) 
	\cdot \frac{L^2}{\reg} \cdot \frac{\log T}{T}.
	\end{align}
Recalling that $b=\log(1/\lambda_2(\bP)) \ge 1-\lambda_2(P)$, assuming $T> 2be$ and subsuming the first term in the third and taking expectations with respect to the sample split the above bound can be written as
\begin{align}
\E\left[ J(\hat{\w}_i(T)) -J(\w^{*})  \right] 
&\le \left(\frac{1}{m} + \frac{100\sqrt{m\specnorm^2}\cdot \log T}{1-\lambda_2(P)} \right)
\notag \\
&\hspace{0.5in}
\cdot \frac{L^2}{\reg} \cdot \frac{\log T}{T}.
\end{align}

\section{General Convergence Result}

For Algorithm \eqref{alg:DiSCO} we also establish a general convergence guarantee for arbitrary strongly convex functions (not necessarily dependent on inner products $\w^{\trans}\x$). Specifically we show that

\begin{theorem}
Fix a Markov matrix $\bP$ and let $\ds=\sigma_1(\mbf{\Sigma})$ denote the spectral norm of the covariance matrix of the data distribution. Consider Algorithm \ref{alg:DiSCO} when the objective $J(\w)$ is strongly convex, $\bP(t) = \bP$ for all $t$, and $\eta_t=1/(\mu t)$.  Let $\lambda_2(\bP)$ denote the second largest eigenvalue of $\bP$.  
Then the expected error for each node $i$ satisfies
\begin{align}
\E\left [ J(\bar{\w}_i(T)) -J(\w^{*})  \right ] \leq 12 \frac{L^2}{\mu} \cdot \frac{\sqrt{m}}{1-\lambda_2(\bP)} \cdot  \frac{\log(T) + \log^2(T)}{T}.
\end{align}
\label{theorem:mainThrmDataIndep}
\end{theorem}
\begin{proof}
Following along the proof of Theorem \ref{theorem:mainThrm} we observe that aside from the data-dependence bounds all steps in the proof remain the same and lead to the above result.
\end{proof}

\section{Asymptotic Analysis}

In this section we explore the sub-optimality of distributed primal averaging when $T \rightarrow \infty$ for the case of smooth strongly convex objectives. As discussed before the results of previous sections do not explain the behaviour of Algorithm \eqref{alg:DiSCO} when each node has infinite data (and therefore each node processes a fresh sample at every iteration). In this case we expect to gain from adding more machines in any network. Towards that goal we investigate the behaviour of Consensus SGD in the asymptotic regime and show that the network effect disappears and we can gain from more machines in any network.

Our analysis depends on the asymptotic normality of a variation of Algorithm \eqref{alg:DiSCO} (See Thm. 5 of \cite{BianchiFortHachem:13IEEETrans}). The main differences being that in this case we average the iterates after making the local update. We make the following assumptions for the analysis in this section: (1) The loss function differentials $\{ \partial \left(\ell(\cdot) \right) \}$ are differentiable and $G$-Lipschitz for some $G>0$, (2) the stochastic gradients are of the form $\g_i(t)=\nabla J(\w_i(t)) + \mbs{\xi}_t$ where $\E[\mbs{\xi}_t] = \mbf{0}$ and $\E[\mbs{\xi}_t \mbs{\xi}^{\trans}_t] = \mbf{C}$, and (3) there exists $p>0$ such that $\E\left[ \norm{\mbs{\xi}_t}^{2+p} \right] < \infty$. Apart from smoothness the rest of the assumptions subsume the case of sampling with replacement in Algorithm \eqref{alg:DiSCO} and hold for arbitrary sampled gradient estimates with a covariance $\mbf{C}$.

Our results hold for all smooth strongly convex objectives not necessarily dependent on $\w^{\trans}\x$.

\begin{lemma}   \label{lemma:mainThrmLarge}
Fix a Markov matrix $\bP$. Consider Algorithm \eqref{alg:DiSCO} when the objective $J(\w)$ is strongly convex and twice differentiable, $\bP(t) = \bP$ for all $t$, and $\step_t=1/(\lambda t)$.  
then the expected error for each node $i$ satisfies for a arbitrary split of $N$ samples into $m$ nodes
\begin{align}
\limsup\limits_{T\rightarrow \infty} T \cdot \E\left[J\left(\sum_{j=1}^m \bP_{ij} \w_j(T)\right) - J(\w^{*})\right] \le
\sum_{j \in \mc{N}(i)}(\bP_{ij})^2 \cdot \Tr\left(\mbf{H}\right) \cdot \frac{G}{\reg}
\end{align}
where $\mbf{H}$ is the solution to the equation
\begin{align}
\nabla J^2(\w^{*}) \mbf{H} + \mbf{H} \nabla J^2(\w^{*})^{T} = \mbf{C}.
\end{align}
\end{lemma}

\noindent \textit{Remark:}  This result shows that asymptotically the network effect from Theorem \eqref{theorem:mainThrmStoch} disappears and that additional nodes can speed convergence.

An application of Lemma \eqref{lemma:mainThrmLarge} to the problem \eqref{eq:optForm} gives us the following result for the specialized case of a complete graph with constant weight matrix $\bP$.

\begin{theorem}   \label{theorem:mainThrmLargeApplication}
Consider Algorithm \ref{alg:DiSCO} when the objective $J(\w)$ has the form \ref{eq:optForm}
, $\bP(t) = \bP$ and corresponds to a complete graph with uniform weights for all $t$, and $\step_t=1/(\lambda t)$.  
then the expected error for each node $i$ satisfies 
\begin{align}
\limsup\limits_{T\rightarrow \infty} T \cdot \E\left[J\left(\sum_{j=1}^m \bP_{ij} \w_j(T)\right) - J(\w^{*})\right] \le \frac{25\rho L^2}{m} \cdot \Tr\left(\nabla^2 J(\w^{*})^{-1} \right) \cdot \frac{G}{\reg} 
 \end{align}
where the expectation is with respect to the history of the sampled gradients as well as the uniform random splits of $N$ data points across $m$ machines.
\end{theorem}

\noindent \textit{Remark:}  For objective \eqref{eq:optForm} we obtain a $1/m$ variance reduction and the network effect disappears.

\subsection{Proof of General Asymptotic Lemma}

\begin{proof}

In the proof we will first show that the iterate of Algorithm \eqref{alg:DiSCO} is asymptotically normal by showing it is close to the iterate of the consensus Algorithm of \citet{BianchiFortHachem:13IEEETrans} and then use the corresponding multivariate normality result of Bianchi et al.~\cite[Theorem 5]{BianchiFortHachem:13IEEETrans}. Finally using smoothness and strong convexity we shall get Lemma \ref{lemma:mainThrmLarge}.  

We need to verify that Algorithm \eqref{alg:DiSCO} satisfies all the assumptions necessary (Assumptions \textbf{1}, \textbf{4}, \textbf{6}, \textbf{7}, \textbf{8a}, and \textbf{8b} in Bianchi et al.~\cite{BianchiFortHachem:13IEEETrans}) for the result to hold.
\begin{itemize}
\item Assumption \textbf{1} requires the weight matrix $\bP(t)$ to be row stochastic almost surely, identically distributed over time, and that $\E[\bP(t)]$ is column stochastic. Our Markov matrix is constant over time and doubly stochastic. Assumption \textbf{1b} follows because $\bP$ is constant and independent of the stochastic gradients, which are sampled uniformly with replacement.
\item Assumption \textbf{4} requires square integrability of the gradients as well as a regularity condition. In our setting, this follows since the sampled gradients are bounded almost everywhere.
\item Assumption \textbf{6} imposes some analytic conditions at the optimum value.  These hold since the gradient is assumed to be differentiable and the Hessian matrix at $\w^{*}$ is positive definite with its smallest eigenvalue is at least $\reg > 0$ (this follows from strong convexity). 
\item Assumption \textbf{7} of Bianchi et al.~\cite{BianchiFortHachem:13IEEETrans} follows from our existing assumptions.
\item Assumptions \textbf{8a} and \textbf{8b} are standard stochastic approximation assumptions on the step size that are easily satisfied by $\step_t = \frac{1}{\reg t}$.
\end{itemize} 

It is easy to show that the average over the nodes of the iterates $\tilde{\w}_i(t)$, $\w_i(t)$ for Algorithm \eqref{alg:DiSCO} and \eqref{alg:DiSCO_large} respectively are the same and satisfy
\begin{align}
\bar{\tilde{\w}}(t+1) = \bar{\tilde{\w}}(t) - \eta_t \frac{\sum_{i=1}^m \g_i(t)}{m}  \notag \\ 
\bar{\w}_i(t+1) = \bar{\w}_i(t+1) - \eta_t \frac{\sum_{i=1}^m \g_i(t)}{m}  
\label{eq:avgIterates}
\end{align}

Now note that 
\begin{align}
\w_i(t) - \w^{*} = \underbrace{\w_i(t) - \bar{\w}_i(t)}_{\textbf{T1=Network Error}} + \underbrace{\bar{\w}_i(t) - \w^{*}}_{\textbf{T2=Asymptotically Normal}}
\label{eq:NetworkErrAsymNormal}
\end{align}

From Lemma \ref{lemma:avgdevBnd} we know that the network error (\textbf{T1}) decays and from update equation \eqref{eq:avgIterates} we know that the averaged iterate for Algorithm \eqref{alg:DiSCO} and consensus Algorithm of \citet{BianchiFortHachem:13IEEETrans} are the same. Then the proof of Theorem $5$ of Bianchi et al.~\cite{BianchiFortHachem:13IEEETrans} shows that the term \textbf{T2}, under the above assumptions when appropriately normalized converges to a centered Gaussian distribution. Equation \eqref{eq:NetworkErrAsymNormal} then implies
\begin{align}
\sqrt{\reg t}\left(\w_i(t) - \w^{*}\right) \sim \mc{N}\left(\mbf{0} , \mbf{H}\right )
\label{eq:multinormDist}
\end{align}
where $\mbf{H}$ is the solution to the equation
\begin{align}
\nabla J^2(\w^{*}) \mbf{H} + \mbf{H} \nabla J^2(\w^{*})^{T} = \mbf{C}.
\end{align}
Let $\mbf{Y} \sim \mc{N}(\mbf{0},\mc{\mbf{I}})$, so we can always write for any $\mbf{X} \sim \mc{N}(\mbf{0},\mbf{\mbf{H}})$
\begin{align}
\mbf{X} = \mbf{Y} \mbf{H}^{1/2},
\end{align}
and thus
\begin{align}
\norm{\mbf{X}}^2 = \mbf{Y}^{\trans}\mbf{H}\mbf{Y}.
\end{align}
Then it is well known that $\norm{\mbf{X}}^2 \sim \chi^2(\Tr(\mbf{H}))$ and so $\E\left[ \norm{\mbf{X}}^2 \right] = \Tr(\mbf{H})$

Let us now consider the suboptimality at the iterate $\sum_{j=1}^m \bP_{ij} \w_j(t)$. It is easy to see that for a differentiable and strongly convex function
\begin{align}
J\left(\sum_{j=1}^m \bP_{ij} \w_j(t)\right) - J(\w^{*}) &\leq \frac{G}{2} \norm{\sum_{j=1}^m \bP_{ij} \w_j(t) - \w^{*}}^2.
\label{eq:suboptAsym}
\end{align}

Now it is easy to see from \eqref{eq:multinormDist} that for a node $j \in \mc{N}(i)$
\begin{align}
\bP_{ij}\sqrt{\reg t}\left(\w_j(t) - \w^{*}\right) \sim \mc{N}\left(\mbf{0} , (\bP_{ij})^2 \mbf{H}\right ).
\end{align}
This implies that
\begin{align}
\sum_{j \in \mc{N}(i)}\bP_{ij}\sqrt{\reg t}\left(\w_j(t) - \w^{*}\right) \sim \mc{N}\left(\mbf{0} ,  \left(\sum_{j \in \mc{N}(i)}(\bP_{ij})^2\right) \mbf{H} \right ).
\label{eq:asmypDist}
\end{align}
Then taking expectation w.r.t to the distribution \eqref{eq:asmypDist} and using standard properties of norms of multivariate normal variables,
\begin{align}
&\E\left[\norm{\sum_{j \in \mc{N}(i)}\bP_{ij}\sqrt{\reg t}\left(\w_j(t) - \w^{*}\right)}^2\right]  \notag \\
&=\left(\sum_{j \in \mc{N}(i)}(\bP_{ij})^2\right) \Tr\left(\mbf{H}\right).
\end{align}

Then substituting in bound \eqref{eq:suboptAsym} and taking the limit we finally get
\begin{align}
&\limsup\limits_{T\rightarrow \infty} T\cdot \E\left[J\left(\sum_{j=1}^m \bP_{ij} \w_j(T)\right) - J(\w^{*})\right] \notag \\
&\le \sum_{j \in \mc{N}(i)}(\bP_{ij})^2 \cdot \Tr\left(\mbf{H}\right) \cdot \frac{G}{\reg}.
\end{align}

\end{proof}

\subsection{Proof of Asymptotic Result for $\ell_2$-regularized objectives}
\begin{proof}
The the covariance of the gradient noise under the sampling with replacement model is
\begin{align}
\mathbf{C} &= \E\left[\g_i(t)\g_i(t)^{\trans}\right] - \nabla J(\w_i(t)) \nabla J(\w_i(t))^{\trans} \notag \\
  &= \frac{\sum_{i=1}^N \beta_{i,t}\x_i\x_i^{T}}{N} + \frac{\mu}{N}\sum_{i=1}^N \beta_{i,t} \left(\x_i\w_i(t)^{\trans} +  \w_i(t)\x_i^{\trans}\right) \notag \\
  &+ \mu^2 \w_i(t)\w_i(t)^{\trans} - \nabla J(\w_i(t)) \nabla J(\w_i(t))^{\trans} \notag \\
\end{align}

Thus we can bound the spectral norm of $\mathbf{C}$ as
\begin{align}
\spec(\mathbf{C}) &\leq L^2 \rho^2 + 2\mu L \E\left[\norm{\w_i(t)}\right] + \mu^2 \E\left[\norm{\w_i(t)}^2\right] \notag \\
                 &+ \E\left[\norm{\nabla J(\w_i(t))}^2\right]
\end{align}

Now from bound \eqref{eq:localiter:norm} since $T \rightarrow \infty$ we have 
\begin{align*}
&\E\left[\norm{\w_i(t)}^2\right] \le \frac{10L^2\specnorm^2}{\reg^2} \notag \\ 
&\E\left[ \norm{\nabla J_i(\w_i(t))}^2 \right] \le 30L^2\specnorm^2
\end{align*}

Putting everything together we get
\begin{align}
\spec(\mathbf{C}) \leq 50\rho L^2.
\label{eq:noiseCov}
\end{align}

Next note that $H = C\left(\nabla^2 J(\w^{*})\right)^{-1}/2$.  From the completeness and uniform weight assumptions on the graph, we have
\begin{align}
\sum_{j \in \mc{N}(i)}(\bP_{ij})^2 = \frac{1}{m}
\end{align}. 

Thus substituting in Lemma \eqref{lemma:mainThrmLarge}, using \eqref{eq:noiseCov} gives us 
\begin{align}
&\limsup\limits_{t\rightarrow \infty} t \cdot \E\left[J\left(\sum_{j=1}^m \bP_{ij} \w_j(t)\right) - J(\w^{*})\right] 
\notag \\
&\le \frac{1}{m} \cdot \frac{\Tr\left(\left(\mbf{C} \nabla^2 J(\w^{*})\right)^{-1} \right)}{2} \cdot \frac{G}{\reg} \notag \\
&\le \frac{25\rho L^2}{m} \cdot \Tr\left(\nabla^2 J(\w^{*})^{-1} \right) \cdot \frac{G}{\reg} \notag \notag \\
\label{eq:finalbndlarge}
\end{align}
\end{proof}

\begin{figure*}
\centering
\includegraphics[width=3.2in]{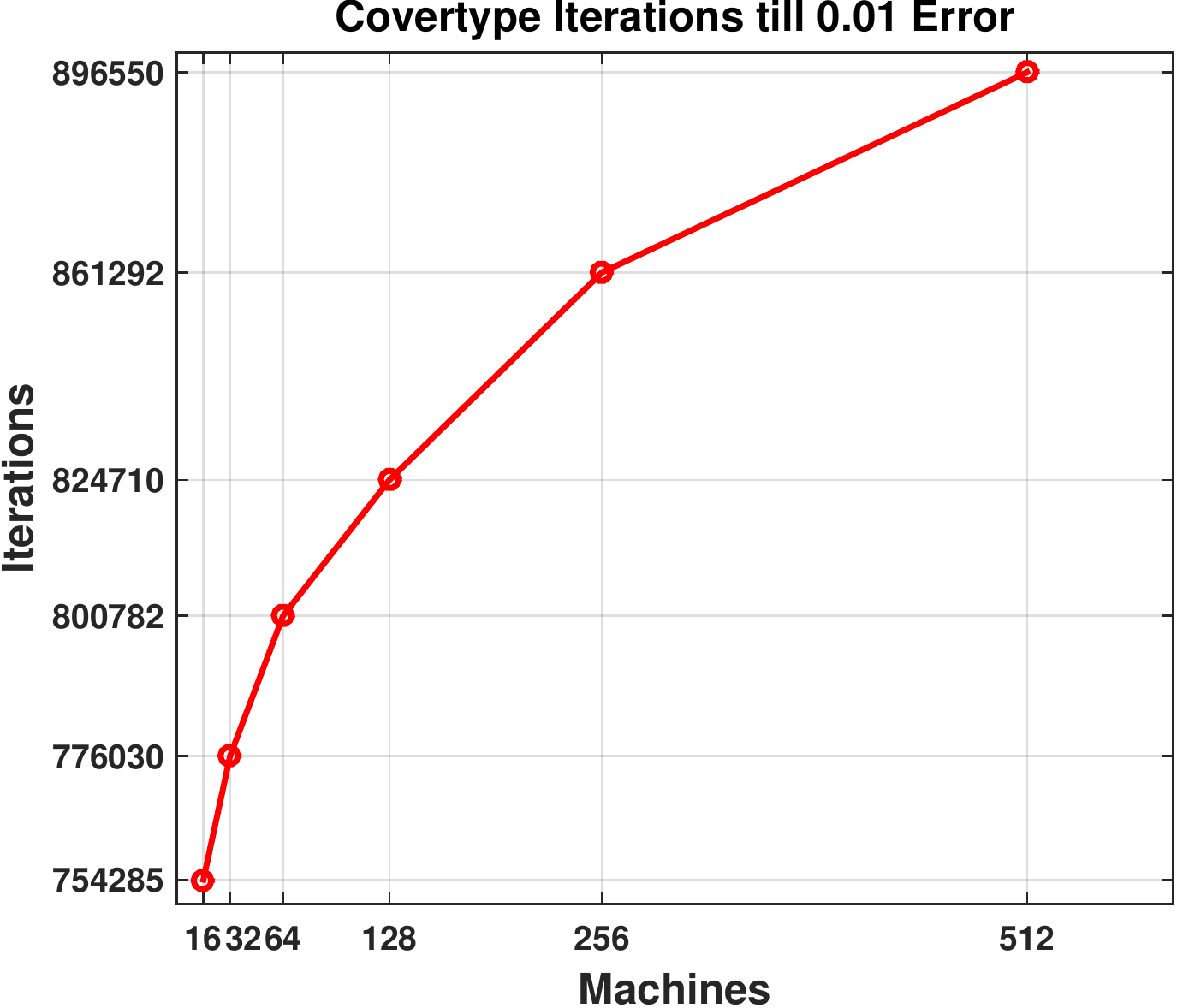}
\includegraphics[width=3in]{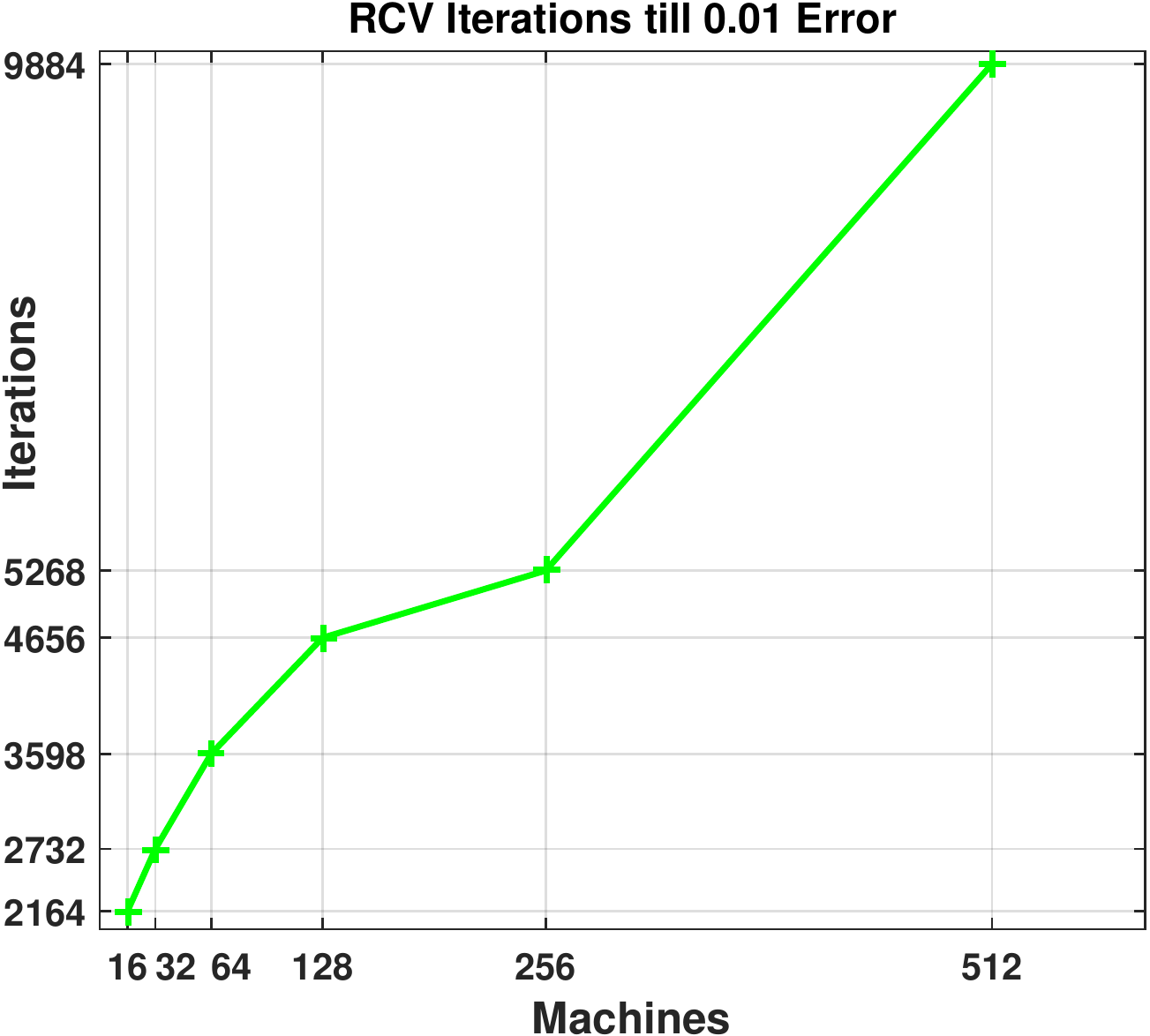}
\caption{ Iterations of Algorithm \eqref{alg:DiSCO} till $\epsilon=0.01$ error on datasets with very different $\specnorm^2$. The performance decay for increasing $m$ is worse for larger $\specnorm^2$. (\ctype~with $\specnorm^2=0.21$ and \rcv~with $\specnorm^2=0.013$)}
\label{fig:rcvcovIters}
\end{figure*}

\removed{
\begin{figure*}[t]
\centering
\includegraphics[width=3.4in]{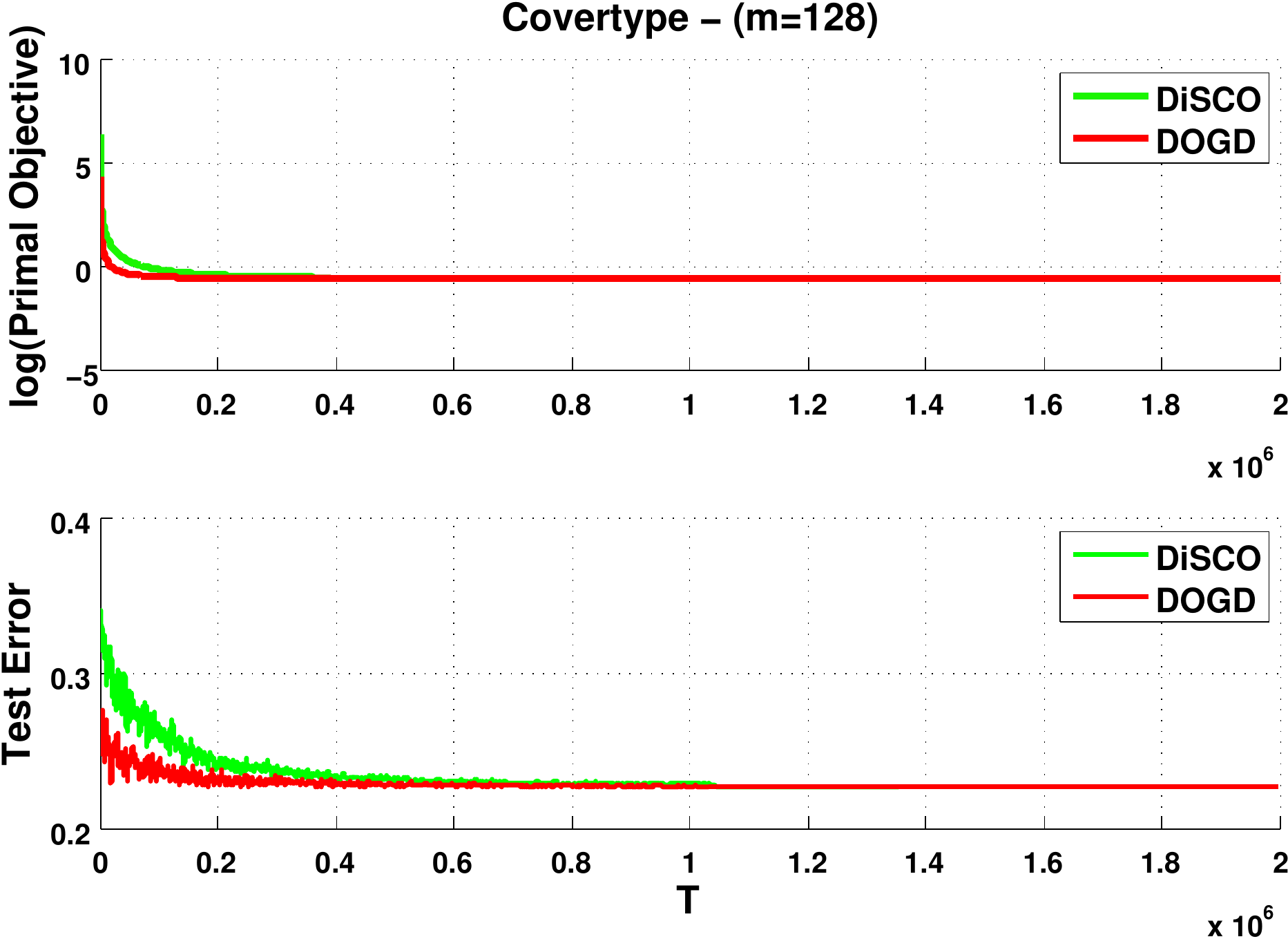} 
\label{fig:dist:cov}\includegraphics[width=3.4in]{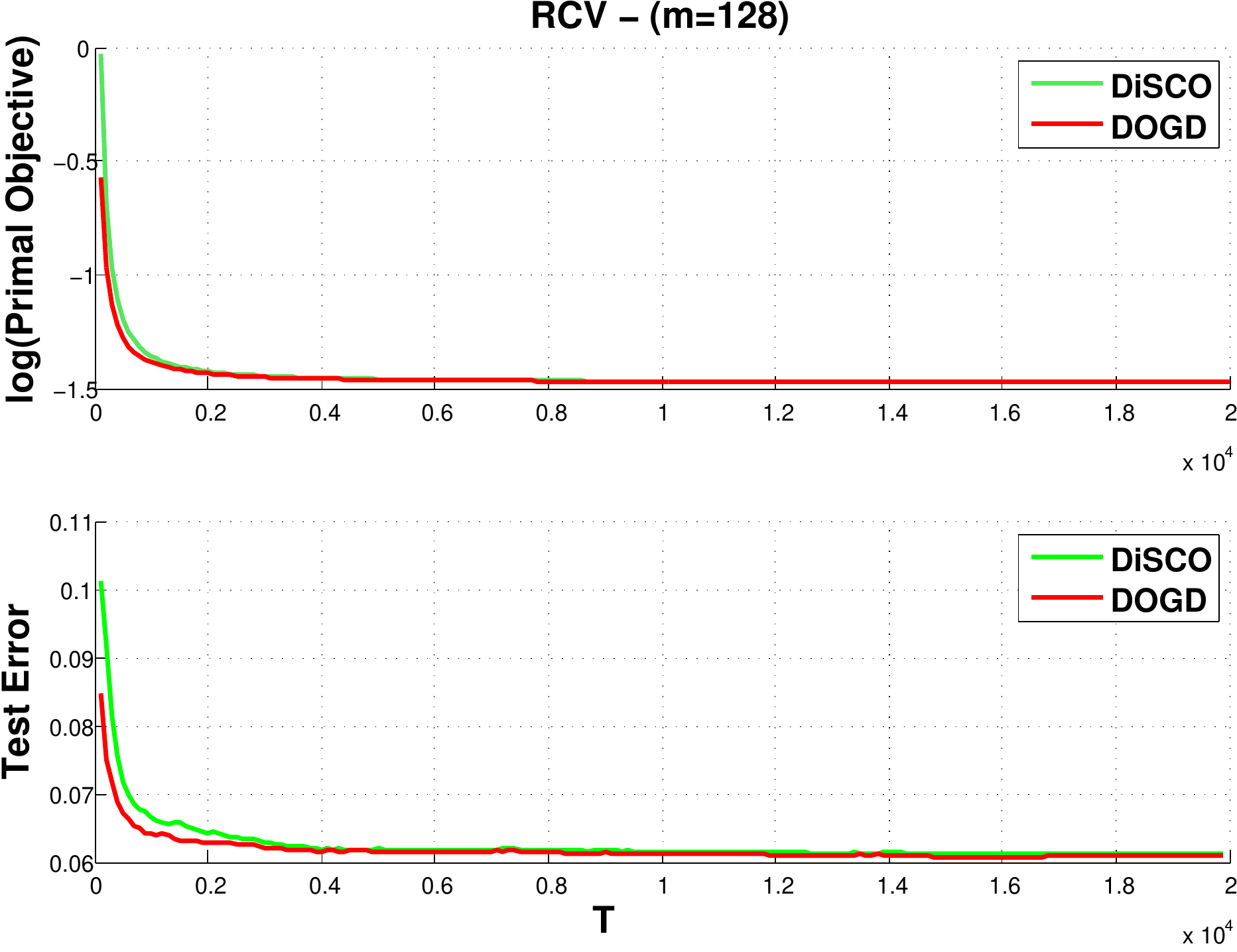}
\label{fig:dist:rcv}
\caption{ Performance of Consensus SGD against \textsc{DoGD} on datasets with very different $\specnorm^2$. The algorithm performs competitively especially for smaller $\specnorm^2$. (\ctype~with $\specnorm^2=0.21$ and \rcv~with $\specnorm^2=0.013$)}
\label{fig:DiSCOvsDODG}
\end{figure*}
\newpage
}
\section{Empirical Results}

\subsection{Performance of as a function of $\specnorm^2$}

In the first set of experiments we look at the number of iterations till $\epsilon=0.01$ error as we vary the number of machines. We looked at the objective after a fixed number of iterations and computed the number of iterations for different $m$s to reach this point. The network considered is a bounded degree expander with Metropolis-Hastings weights. We can see in Figure \ref{fig:rcvcovIters} that the iterations for dataset with a smaller $\specnorm^2$ grow slower as we spread the data across more machines.

\begin{figure}[t]
\centering
\scalebox{0.95}{
\includegraphics[width=4.5in]{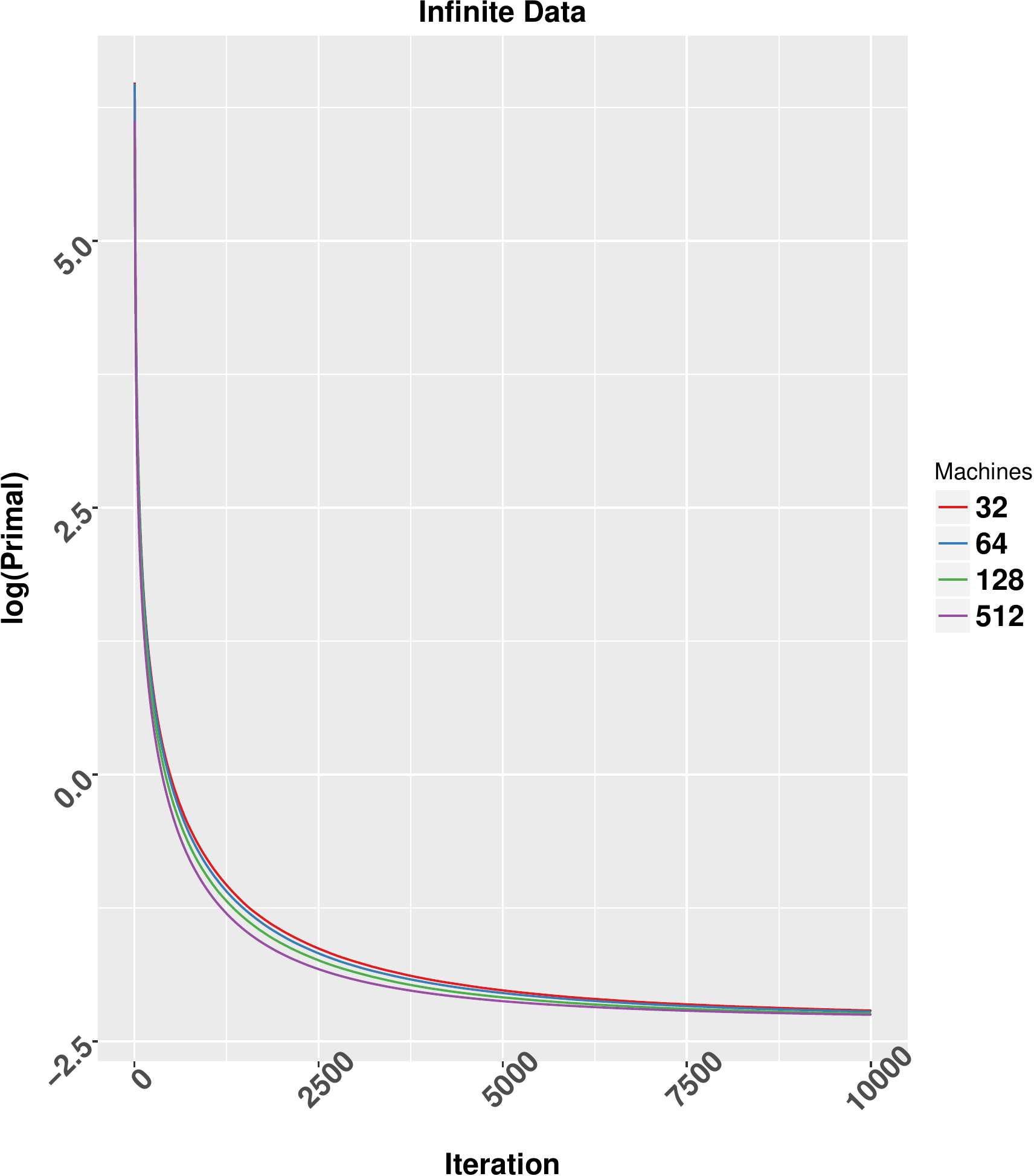}}
\label{fig:inifnite}
\caption{No network effect in the case of infinite data.}
\label{fig:Infinite}
\end{figure}
\removed{
\subsection{Comparison Against \textsc{DoGD}} 
 The next question we answer pertains to the competitiveness of Algorithm \ref{alg:DiSCO} in a real setting. We show that the algorithm is competitive with a recently proposed method and the corresponding convergence guarantees give us a theoretical understanding of when it is expected to perform better.

To do so we compare Algorithm \ref{alg:DiSCO} to the distributed method presented in \citep{DistStronglyConvex}(referred to here as \textsc{DoGD}- Distributed Online Gradient Descent). They are quite similar and differ essentially in how the step-sizes decay (constant over epoch and decaying vs strictly decaying). As shown in the theoretical results we expect Algorithm  \ref{alg:DiSCO} to perform better on datasets with a smaller spectral norm (\astro~and \rcv) and we can see in Figure \ref{fig:DiSCOvsDODG} that is the case, while on \ctype~\textsc{DoGD} performs slightly better (especially on test error). In general, we can see that tuning Algorithm \ref{alg:DiSCO} appropriately gives competitive performance.
}
\subsection{Infinite Data} 

To provide some empirical evidence of the fact that for infinite data from the same distribution we shouldn't see any degradation in the error(for growing $m$), we generate a very large ($N=10^7$) synthetic dataset from a multivariate Normal distribution and created a simple binary classification task using a random hyperplane. As we can see in Figure \ref{fig:Infinite} for the SVM problem and a $k$-regular network we continue to gain as we add more machines and then eventually we stabilize but never lose from more machines. 

\section{Summary}

We analyzed a primal averaging based SGD algorithm and showed that the convergence rate depends on the spectral norm of the sample covariance. This analysis reveals a tradeoff of the data distribution with network parameters that is missing from previous analyses of consensus based optimization algorithms. Additionally we looked at the asymptotic regime when the loss function is smooth and strongly convex. The analysis shows that the network effect disappears in this regime and we can gain from adding more machines in a arbitrary network. In the next chapter we analyze the effect of intermittent communication and show how data-dependence can help us.



\chapter{General Protocols and Sparse Communication} 
\label{Chapter4}
\lhead{Chapter 4. \emph{Communication Protocols}} 

To circumvent the high communication frequency requirements in the scheme of Chapter \eqref{Chapter3} we now establish a general convergence result that can allow for varied communication protocols such as asynchronous and intermittent communication. The ensuing convergence analysis will reveal how certain data distributions allow for relaxed communication regimes. Particularly we generalize our analysis in Chapter  \eqref{Chapter3} to include time-varying and stochastic communication matrices $\bP(t)$.   We study the case where the matrices are chosen i.i.d.~over time.  A special case is the strategy when the network alternates between some fixed $\bP$ and the identity matrix, a situation we refer to as \textit{intermittent connectivity}. Intermittent communication can result in slower convergence because the gap between the local node estimates and their average is larger.  We call this the network error.  Our goal is to show how knowing $\rho^2$ can help us balance the network error and optimality gap.

\removed{A subcase of intermittent communication is the case when the frequency of communication decreases with iterations. Underlying this approach is the intuition that when the local iterates are all close to the optimum they shouldn't need to communicate often. The convergence result for this diminishing communication regime doesn't follow from the analysis of general stochastic communication matrices since the identically distributed assumption is violated and requires a different approach, as we shall see later. }

\section{Stochastic Communication}

\subsection{General Protocols}
We adapt the proof of the the general communication protocol convergence from the distributed dual averaging approach of ~\citep{dualAveraging}. \removed{However, as mentioned before this analysis doesn't apply to the diminishing communication protocol. }

In the remainder of the chapter we first establish the network error lemma for general communication schemes and then go on to apply the proof to the case of intermittent regimes. We follow up the theoretical analysis by showing the performance of the various schemes on different datasets defined in Chapter \eqref{Chapter1}.
\begin{lemma}\label{lemma:avgdevBndStoch}
Fix a Markov matrix $\bP$ and consider Algorithm \eqref{alg:DiSCO} when the objective $J(\w)$ is strongly convex \removed{and the number of iterations \$T$ satisfies
\begin{align}
	T &> 2 e  \log(1/\sqrt{ \lambda_2(\bP)}) 
	\end{align} }
we have the following inequality for the expected squared error between the iterate $\w_i(t)$ at node $i$ at time $t$ and the average $\bar{\w}(t)$ defined in Algorithm \eqref{alg:DiSCO}:
	\begin{align}
	\sqrt{\E\left[\norm{ \bar{\w}(t)-\w_i(t) }^2 \right ]} \leq \frac{2L}{\mu}\cdot \frac{\sqrt{m}}{b }\cdot\frac{\log(2bet^2)}{t},
	\end{align}
where $b= (1/2)\log(1/\lambda_2(\bP))$.
\end{lemma}

\begin{proof}
We follow a similar analysis as others \citep[Prop. 3]{nedicDistributedOptimization}~\citep[IV.A]{dualAveraging}~\citep{DistStronglyConvex}.  Let $\mbf{W}(t)$ be the $m \times d$ matrix whose $i$-th row is $\w_i(t)$ and $\mbf{G}(t)$ be the $m \times d$ matrix whose $i$-th row is $\g_i(t)$ .  Then the iteration can be compactly written as
	\[
	\mbf{W}(t+1) = \bP(t) \mbf{W}(t) - \eta_t \mbf{G}(t)
	\]
and the network average matrix $\bar{\mbf{W}}(t) = \frac{1}{m} \mbf{1} \mbf{1}^{\top} \mbf{W}(t)$.  Then we can write the difference using the fact that $\bP(t) = \bP$ for all $t$:
	\begin{align}
	\bar{\mbf{W}}(t+1) - \mbf{W}(t+1) 
		&= \left(  \frac{1}{m} \mbf{1} \mbf{1}^{\top} - I \right) \left( \bP \mbf{W}(t) - \eta_t \mbf{G}(t) \right) \notag \\
		&= \left(  \frac{1}{m} \mbf{1} \mbf{1}^{\top} - \bP \right) \mbf{W}(t)
			- \eta_t \left(  \frac{1}{m} \mbf{1} \mbf{1}^{\top} - I \right) \mbf{G}(t) \notag \\
		&= \left(  \frac{1}{m} \mbf{1} \mbf{1}^{\top} - \bP \right) \left( \bP \mbf{W}(t-1) - \eta_{t-1} \mbf{G}(t-1) \right)
			\notag \\
			&\hspace{1in} - \eta_t \left(  \frac{1}{m} \mbf{1} \mbf{1}^{\top} - I \right) \mbf{G}(t) \notag \\
		&= \left(  \frac{1}{m} \mbf{1} \mbf{1}^{\top} - \bP^2 \right) \mbf{W}(t-1)
			- \eta_{t-1} \left(  \frac{1}{m} \mbf{1} \mbf{1}^{\top} - \bP \right) \mbf{G}(t-1)
			\notag \\
			&\hspace{1in} - \eta_t \left(  \frac{1}{m} \mbf{1} \mbf{1}^{\top} - I \right) \mbf{G}(t) \notag \\
		&= \left(  \frac{1}{m} \mbf{1} \mbf{1}^{\top} - \bP^2 \right) \mbf{W}(t-1)
			- \sum_{s=t-1}^{t} \eta_s \left(  \frac{1}{m} \mbf{1} \mbf{1}^{\top} - \bP^{t-s} \right) \mbf{G}(s).
	\end{align}
Continuing the expansion and using the fact that $\mbf{W}(1) = \mbf{0}$,
	\begin{align}
	\bar{\mbf{W}}(t+1) - \mbf{W}(t+1) &= \left( \frac{1}{m} \mbf{1} \mbf{1}^{\top} - \bP^t \right)  \mbf{W}(1)
		- \sum_{s=1}^{t} \eta_s \left(  \frac{1}{m} \mbf{1} \mbf{1}^{\top} - \bP^{t-s} \right) \mbf{G}(s) \notag \\
		&= - \sum_{s=1}^{t} \eta_s \left(  \frac{1}{m} \mbf{1} \mbf{1}^{\top} - \bP^{t-s} \right) \mbf{G}(s) \notag \\
		&= - \sum_{s=1}^{t-1} \eta_s \left(  \frac{1}{m} \mbf{1} \mbf{1}^{\top} - \bP^{t-s} \right) \mbf{G}(s) - \eta_t \left(  \frac{1}{m} \mbf{1} \mbf{1}^{\top} - I \right) \mbf{G}(t).
			\label{eq:more:neterr_matrix}
	\end{align}
Now looking at the norm of the $i$-th row of \eqref{eq:more:neterr_matrix} and using the bound on the gradient norm:
	\begin{align}
	\norm{ \bar{\w}(t)-\w_i(t) } 
		&\le \norm{ \sum_{s=1}^{t-1} \eta_s \sum_{j=1}^{m} \left( \frac{1}{m} - (\bP^{t-s})_{ij} \right) \g_j(s) 
			+ \eta_t \left( \sum_{j=1}^{m} \frac{1}{m} \g_j(t) - \g_i(t) \right) } \\
		&\le \sum_{s=1}^{t-1} \frac{L}{\mu s} \cdot \norm{ \frac{1}{m} - (\bP^{t-s})_{i} }_1 + \frac{2L}{\mu t}.
		\label{eq:devBnd}
	\end{align}

We handle the term $\norm{ \frac{1}{m} - (\bP^{t-s})_{i} }_1$ using a bound on the mixing rate of Markov chains (c.f. (74) in \citet{DistStronglyConvex}):
	\begin{align}
	\sum_{s=1}^{t-1} \frac{L}{\mu s} \cdot \norm{ \frac{1}{m} - (\bP^{t-s})_{i} }_1 
		&\le \frac{L}{\mu} \sqrt{m} \sum_{s=1}^{t-1} \left( \sqrt{ \lambda_2(\bP)} \right)^{t - s} \frac{1}{s}.
	\end{align}

Define $a = \sqrt{\lambda_2(\bP)} \le 1$ and $b = - \log(a) > 0$.  Then we have the following identities:
\begin{align}
\sum_{\tau=1}^t \frac{a^{t-\tau+1}}{\tau} 
= \sum_{\tau=1}^t \frac{a^\tau}{t-\tau+1}
=  \sum_{\tau=1}^t \frac{\exp(-b\tau)}{t-\tau+1}.
\end{align}
Now using the fact that when $x > -1$ we have $\exp(-x)< 1/(1+x)$ and using the integral upper bound we get
\begin{align}
\sum_{\tau=1}^t \frac{a^{t-\tau+1}}{\tau}   \leq &  \sum_{\tau=1}^t \frac{1}{(1+b\tau)(t-\tau+1)} \notag \\
 \leq &  \frac{1}{(1+b)t} + \int_{1}^{t} \frac{d\tau}{(1+b\tau)(t-\tau+1)} \notag \\
 = &  \frac{1}{(1+b)t} + \left |\frac{\log(b\tau+1)-\log(t-\tau+1)}{bt+b+1} \right |_1^{t} \notag \\
 = & \frac{1}{(1+b)t} + \frac{\log(bt+1)-\log(b+1)+\log(t)}{bt+b+1} \notag \\
 \leq &  \frac{\log(et(bt+1))}{bt} \notag \\
 \leq &  \frac{\log(2bet^2)}{bt}.
 \label{eq:seriesBnd}
\end{align}
Using \eqref{eq:seriesBnd} in \eqref{eq:devBnd} we get
\begin{align}
\norm{\bar{\w}(t)-\w_i(t)} &\le \frac{L\sqrt{m}}{\mu}\frac{\log(2bet^2)}{bt} + \frac{2L}{ \mu t} \nonumber \\
&\le \frac{2L\sqrt{m}}{\mu}\frac{\log(2bet^2)}{bt}.
\label{eq:networkDevBnd}
\end{align}
Therefore we have 
	\begin{align}
	\sqrt{\E\left [\norm{ \bar{\w}(t)-\w_i(t) }^2 \right ]}
		&\le \frac{2L\sqrt{m}}{\mu}\frac{\log(2bet^2)}{bt}.
	\label{eq:avgtoind:normbnd1}
\end{align}

\end{proof}

Armed with Lemma \eqref{lemma:avgdevBndStoch} we prove the following theorem for Algorithm \eqref{alg:DiSCO} in the case of stochastic communication
\begin{theorem}
Let $\{\bP(t)\}$ be an i.i.d sequence of doubly stochastic matrices and $\rho^2=\sigma_1(\mbf{\hat{\Sigma}})$ denote the spectral norm of the covariance matrix of the data distribution. Consider Algorithm \eqref{alg:DiSCO} when the objective $J(\w)$ is strongly convex, and $\eta_t=1/(\mu t)$. Then if the number of samples on each machine $n$ satisfies 
	\begin{align}
	n > (4/3) \left(\rho^2 \log \left ( d\right ) \right )
	\end{align}
and the number of iterations $T$ satisfies
	\begin{align}
	T &> 2 e  \log(1/\sqrt{ \lambda_2(\E\left[\bP^2(t)\right])}) \\
	\frac{T}{\log(T)} &> \max\left(\frac{4}{3\rho^2}\log\left( d \right), \sqrt{\frac{8}{5}} \cdot\sqrt{ \frac{m}{\hat{\rho}^2} } \cdot \frac{1}{\log(1/\lambda_2(\E\left[\bP^2(t)\right]))}\right), 
	\end{align}
then the expected error for each node $i$ satisfies
\begin{align}
\E\left [ J(\bar{\w}_i(T)) -J(\w^{*})  \right ]  
&\le 
\left (\frac{1}{m} + \frac{150\sqrt{m\hat{\rho}^2}\cdot \log T}{1-\sqrt{\lambda_2(\E\left[\bP^2(t)\right])}} \right )\cdot \frac{L^2}{\mu} \cdot \frac{\log T}{T} 
\end{align}
\label{theorem:mainThrmStoch}
\end{theorem}

\begin{proof}
Since \eqref{eq:mainBnd} still holds, we merely apply Lemma \eqref{lemma:avgdevBndStoch} in \eqref{eq:mainBnd} and continue in the same way as the proof of Theorem \eqref{theorem:mainThrm}.
\end{proof}

\section{Limiting Communication}

As an application of the stochastic communication scenario we now turn to an analysis of reducing the communication overhead of Algorithm \eqref{alg:DiSCO}.  This reduction can improve the overall running time of the algorithm because communication latency can hinder the convergence of many algorithms in practice.  A natural way of limiting communication is to communicate only a fraction $\commfrac$ of the $T$ total iterations; at other times nodes simply perform local gradient steps.

Let the sequence of random matrices $\{\bP(t)\}$ for Algorithm \eqref{alg:DiSCO} be i.i.d.~as follows
\begin{align}
\bP(t) =
\left\{
	\begin{array}{ll}
		\mathbf{I}  &  \text{ with probability } 1-\commfrac \\
		\bP &  \text{ with probability } \commfrac
	\end{array}
\right.
\label{eq:interComm}
\end{align}
where $\mathbf{I}$ is the identity matrix (implying no communication since $\bP_{ij}(t)=0$ for $i \ne j$) and, as in the previous section, $\bP$ is a fixed doubly stochastic matrix respecting the graph constraints. For this model the expected number of times communication takes place is simply $\nu T$.   Note that now we have an additional randomization due to the Bernoulli distribution over the doubly stochastic matrices.  

A straightforward application of Theorem \eqref{theorem:mainThrmStoch} reveals that the optimization error is proportional to $\frac{1}{\nu}$ and decays as $\mc{O}(\frac{1}{\nu} \cdot \frac{\log^2(T)}{T})$. However, this ignores the effect of the local communication-free iterations and suggests that only the communication rounds matter.

\subsection{Mini Batching Perspective on Intermittent Communication}
To account for local communication free iterations we modify the intermittent communication scheme to follow a deterministic schedule of communication every $1/\nu$ steps. However, instead of taking single gradient steps between communication rounds, each node gathers the (sub)gradients and then takes an aggregate gradient step. That is, after the $t$-th round of communication, the node samples a batch $\mc{I}_t$ of indices sampled with replacement from its local data set with $|\mc{I}_t| = 1/\nu$. We can think of this as the base algorithm with a better gradient estimate at each step. The update rule is now
\begin{align}
\w_i(t+1) = \sum_{j \in \mc{N}_i} \w_j(t) P_{ij}(t) - \step_t \nu \sum_{ i \in \mc{I}_i } \g_i(t).
\label{eq:mbatchComm}
\end{align} 
We define $\g^{1/\nu}_i(t) = \sum_{ i \in \mc{I}_i } \g_i(t)$. Now the iteration count is over the communication steps and $\g^{1/\nu}_i(t)$ is the aggregated mini-batch (sub)gradient of size $1/\nu$. Note that this is analogous to the random scheme above but the analysis is more tractable.

\begin{theorem}
Fix a Markov matrix $\bP$ and let $\specnorm^2=\sigma_1(\mbf{\hat{\Sigma}})$ denote the spectral norm of the covariance matrix of the data distribution. Consider Algorithm \eqref{alg:DiSCO} when the objective $J(\w)$ is strongly convex, $\bP(t) = \bP$ for all $t$, and $\step_t=1/(\reg t)$ for scheme \eqref{eq:mbatchComm}.  Let $\lambda_2(\bP)$ denote the second largest eigenvalue of $\bP$.  Then if the number of samples on each machine $n$ satisfies 
	\begin{align}
	n > \frac{4}{3 \rho^2} \log \left ( d\right ) 
	\end{align}
and 
	\begin{align}
	&T > \frac{2e}{\nu}  \log(1/\sqrt{ \lambda_2(\bP)}) \notag \\
	&\frac{T}{\log(\nu T)} > \max\left(\frac{4}{3\nu \specnorm^2}\log(d), \frac{\left( \frac{8}{5} \right)^{\frac{1}{4}} \sqrt{ m/\specnorm^2 }}{\log(1/\lambda_2)}\right) \notag \\
	&\frac{1}{\nu} > \frac{4}{3\rho^2} \cdot \log(d)
	 \label{eq:thm:Tbound}
	\end{align}
	and 
then the expected error for each node $i$ satisfies 
	\begin{align}
	\E\left [ J(\hat{\w}_i(T)) -J(\w^{*})  \right ]  \le \left( \frac{1}{m} 
		+ 200\sqrt{5} 
			\cdot \frac{\sqrt{m\specnorm^4}\cdot \log (\nu T)}{1-\sqrt{\lambda_2}} 
		\right) \cdot \frac{L^2}{\reg} \cdot \frac{\log (\nu T)}{T}.
\end{align}
where $\nu$ is the frequency of communication and where $\lambda_2=\lambda_2(\bP)$.
\label{theorem:mainThrm:mbatch}
\end{theorem}

\noindent \textit{Remark:} Theorem \ref{theorem:mainThrm:mbatch} suggests that if the inverse frequency of communication is large enough than we can obtain a sharper bound on the error by a factor of $\rho$. This is significantly better than a $\mc{O}(\sqrt{m\specnorm^2}\cdot\frac{\log \nu T}{\nu T})$ baseline guarantee from a direct application of Theorem \ref{theorem:mainThrm} when the number of iterations is $\nu T$.

Additionally the result suggests that if we communicate on a mini-batch(where batch size $b=1/\nu$) that is large enough we can improve Theorem \ref{theorem:mainThrm}, specifically now we get a $1/m$ improvement when $m \leq 1/\rho^{4/3}$.

\subsection{Proof of Convergence}

\begin{proof}

We will first establish the network lemma for scheme \eqref{eq:mbatchComm}.

\begin{lemma}\label{lemma:avgdevBndmbatch}
Fix a Markov matrix $\bP$ and consider Algorithm \eqref{alg:DiSCO} when the objective $J(\w)$ is strongly convex and the frequency of communication satisfies
    \begin{align}
	1/\nu > \frac{4}{3\specnorm^2} \log(d) 
	\end{align}
we have the following inequality for the expected squared error between the iterate $\w_i(t)$ at node $i$ at time $t$ and the average $\bar{\w}(t)$ defined in Algorithm \eqref{alg:DiSCO} for scheme \eqref{eq:mbatchComm}
	\begin{align}
	\sqrt{\E\left[\norm{ \bar{\w}(t)-\w_i(t) }^2 \right]} 
	\leq \frac{4L\sqrt{5m \specnorm^2}}{\reg}\cdot \frac{\log(2bet^2)}{bt}
	\end{align}
where $b = (1/2)\log(1/\lambda_2(\bP)) $.
\end{lemma}

\begin{proof}

It is easy to see that we can write the update equation in Algorithm \eqref{alg:DiSCO},
\begin{align}
\w_i(t+1) &= \sum_{j=1}^m\tilde{\bP}_{ij}(t)\w_j(t)- \step_t \g^{1/\nu}_i(t)
\end{align}
where
\begin{align}
\tilde{\bP}_{ij}(t) =
\left\{
	\begin{array}{ll}
		\bP_{ij}(t)  &  \text{ when } i \ne j \\
		\bP_{ii}(t)-\frac{1}{mt} &  \text{ when } i=j
	\end{array}
\right.
\end{align}
and $\g_i(t) = \g^{1/\nu}_i(t) + \reg\w_i(t)$.

\removed{
Next we have from \eqref{eq:devBnd} we have 
	\begin{align}
   &	\norm{ \bar{\w}(t)-\w_i(t) } \le \notag \\
	& \norm{ \sum_{s=1}^{t-1} \step_s \sum_{j=1}^m \left( \frac{1}{m} - (\tilde{\bP}^{t-s})_{ij} \right) \g^{1/\nu}_j(s) 
	+ \step_t \left( \sum_{j=1}^{m} \frac{1}{m} \g^{1/\nu}_j(t) - \g^{1/\nu}_i(t) \right) } 
	\end{align}}

We need first a bound on $\norm{\g^{1/\nu}_j(s)}$ using the definition of the minibatch (sub)gradient:
 \begin{align}
\norm{\g^{1/\nu}_i(s)}^2 &= \norm{ \frac{\sum_{i_{k_s} \in H^i_s}\partial \ell(\w_i(s)^{\trans}\x_{k_{i_s}})\x_{k_{i_s}} }{1/\nu}}^2  \notag \\
&\leq L^2\nu \norm{\Q_{1/\nu}} 
\end{align} 
From \eqref{eq:devBnd} and the minibatch (sub)gradient bound
	\begin{align*}
   	&
   	\norm{ \bar{\w}(t)-\w_i(t) } \notag \\
	&\le 
	\norm{ \sum_{s=1}^{t-1} \step_s \sum_{j=1}^m \left( \frac{1}{m} - (\tilde{\bP}^{t-s})_{ij} \right) \g^{1/\nu}_j(s)} 
	\notag \\ 
  		&\qquad
		+ \step_t \norm{\left( \sum_{j=1}^{m} \frac{\mbf{1}}{m} \g^{1/\nu}_j(t) - \g^{1/\nu}_i(t) \right) } \\
	&\le L\sqrt{\nu\norm{\Q_{1/\nu}}}\sum_{s=1}^{t-1} \frac{\norm{ \frac{\mbf{1}}{m} - (\tilde{\bP}^{t-s})_{i} }_1}{\reg s} + \frac{2L\sqrt{\nu\norm{\Q_{1/\nu}}}}{\reg t} \\
	&\le L\sqrt{\nu\norm{\Q_{1/\nu}}}
		\notag \\
		&\hspace{0.6in}
		\sum_{s=1}^{t-1} \frac{\norm{ \frac{\mbf{1}}{m} - (\bP^{t-s})_{i} }_1 + \norm{ (\bP^{t-s})_{i} - (\tilde{\bP}^{t-s})_{i} }_1}{\reg s} \notag \\
		&\qquad + \frac{2L\sqrt{\nu\norm{\Q_{1/\nu}}}}{\reg t} \\
	&\le 
	2L\sqrt{\nu\norm{\Q_{1/\nu}}}\sum_{s=1}^{t-1} \frac{\norm{ \frac{\mbf{1}}{m} - (\bP^{t-s})_{i} }_1}{\reg s} + \frac{2L\sqrt{\nu\norm{\Q_{1/\nu}}}}{\reg t} 
	\end{align*}
	
Continuing as in the proof of Lemma \ref{lemma:avgdevBnd}, taking expectations and using Lemma \ref{lem:specnormIntdim}, for 
$1/\nu > \frac{4}{3\specnorm^2} \log(d)$ we have
	\begin{align}
	\sqrt{\E\left [\norm{ \bar{\w}(t)-\w_i(t) }^2 \right ]}
		&\le \frac{4L\sqrt{m\nu\E\left[\norm{\Q^{1/\nu}}\right]}}{\reg}\frac{\log(2bet^2)}{bt} \notag \\
		&\le \frac{4L\sqrt{5m\specnorm^2}}{\reg}\frac{\log(2bet^2)}{bt}
		\label{eq:avgtoind:normbnd:minibatch}
\end{align}

\end{proof}

For the scheme \eqref{eq:mbatchComm} all the steps until bound \eqref{eq:mainBnd} from proof of Theorem \eqref{theorem:mainThrmStoch} remain the same. The difference in the rest of the proof arises primarily from the mini-batch gradient norm factor in Lemma \eqref{lemma:avgdevBndmbatch}. We have the same decomposition as \eqref{eq:mainBnd} with $\text{T1}$, $\text{T2}$, and $\text{T3}$ as in \eqref{eq:T1}, \eqref{eq:T2}, and \eqref{eq:T3}.
 The gradient norm bounds also don't change since the minibatch gradient is also an unbiased gradient of the true gradient $\nabla J(\cdot)$. Thus substituting Lemma \eqref{lemma:avgdevBndmbatch} in the above and following the same steps as in proof of Theorem \eqref{theorem:mainThrmStoch}, replacing $T$ by $\nu T$ where $T$ is now the total iterations including the communication as well as the minibatch gathering rounds, we get Theorem \eqref{theorem:mainThrm:mbatch}.
\end{proof}
\begin{figure*}[t]
\centering
\scalebox{0.95}{
\includegraphics[width=3.8in]{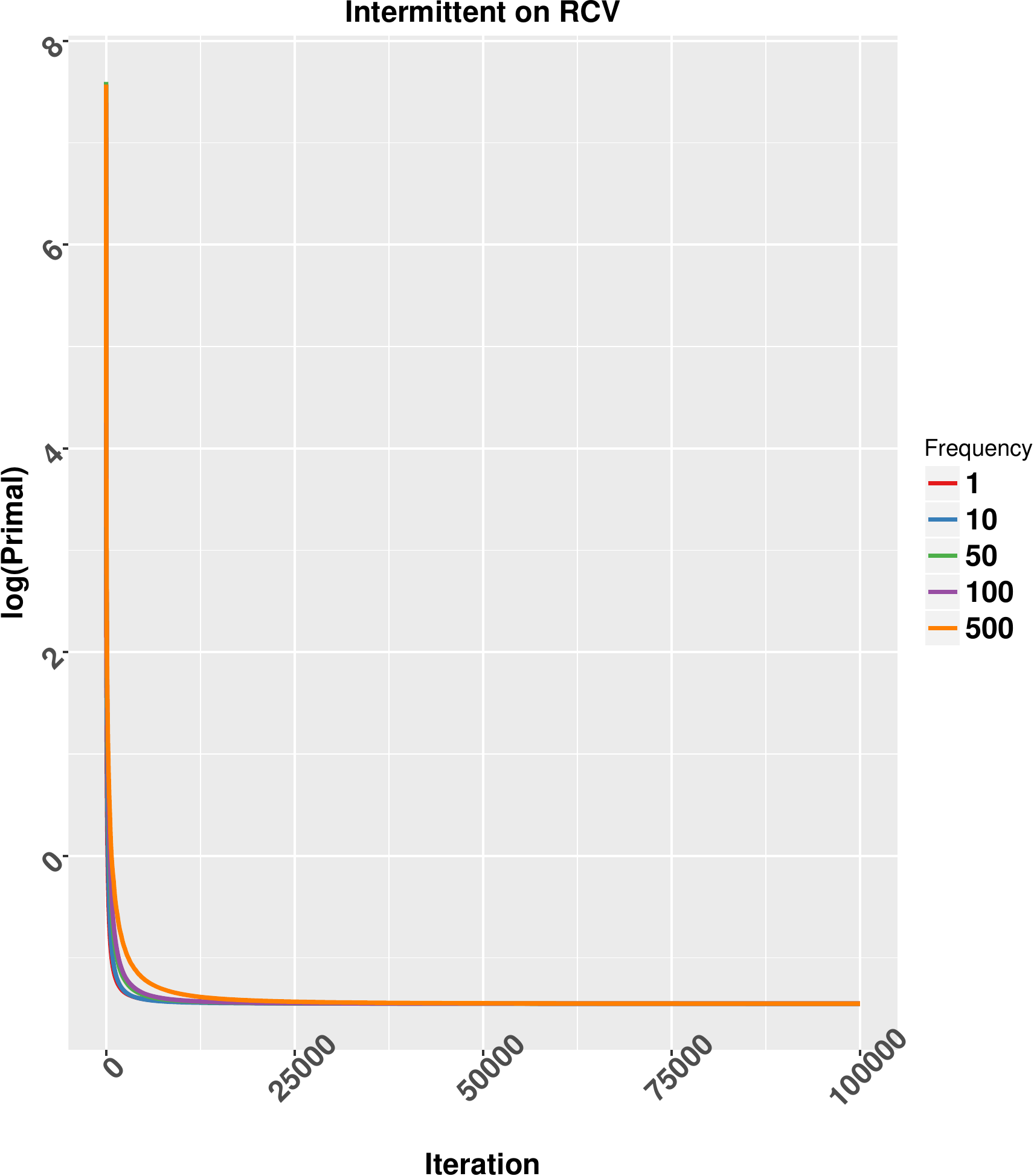}
}
\label{fig:dist:rcv}
\scalebox{0.95}{
\includegraphics[width=3.8in]{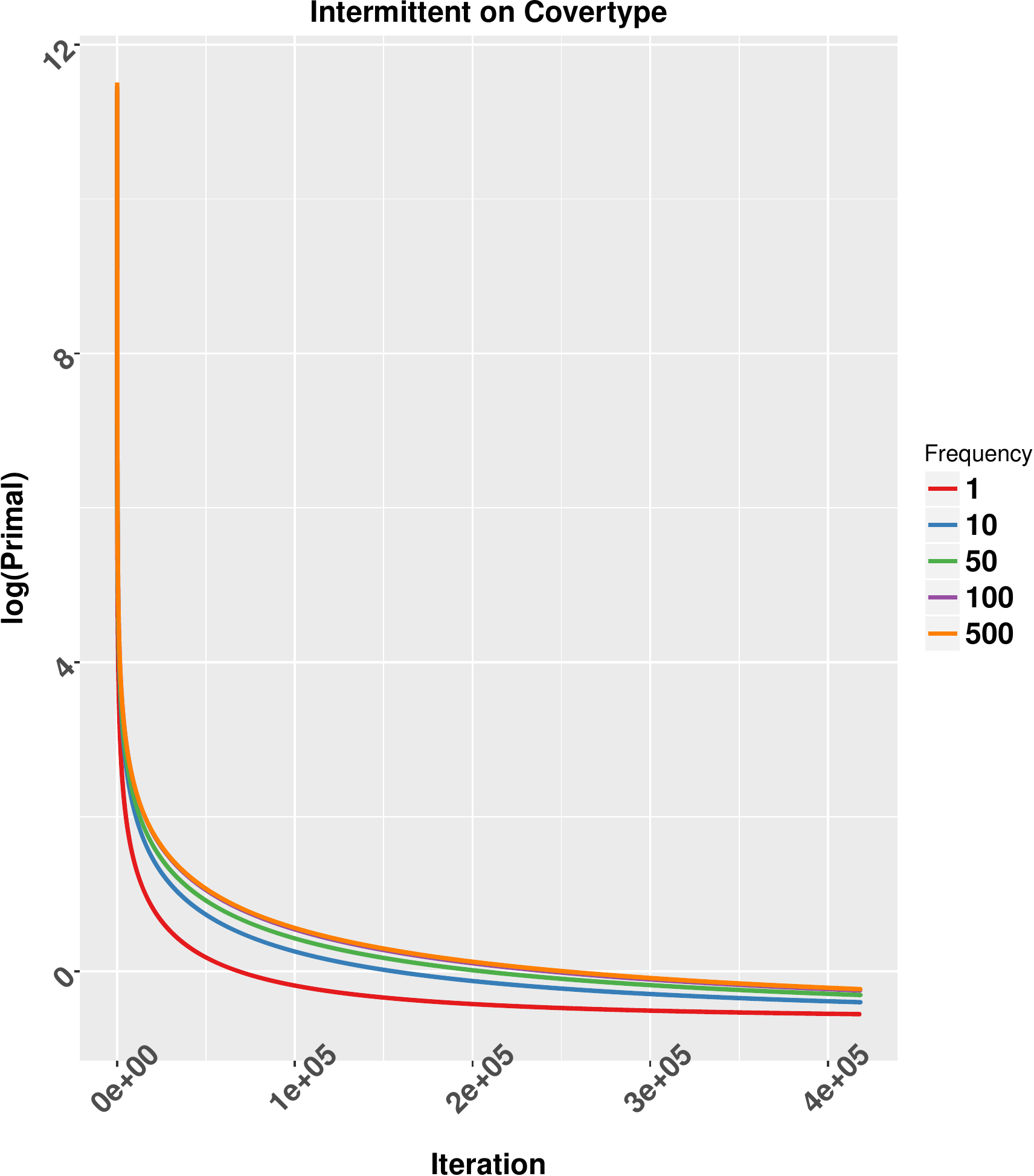}
}
\label{fig:dist:rcv}

\caption{ Performance of Algorithm \eqref{alg:DiSCO} with intermittent communication scheme on datasets with very different $\specnorm^2$. The algorithm works better for smaller $\specnorm^2$ and there is less decay in performance for \rcv~ as we decrease the number of communication rounds as opposed to \ctype (\ctype~with $\specnorm^2=0.21$ and \rcv~with $\specnorm^2=0.013$).}
\label{fig:DiSCOvsDODG}
\end{figure*} 
\begin{figure}[t]
\centering
\includegraphics[width=4.5in]{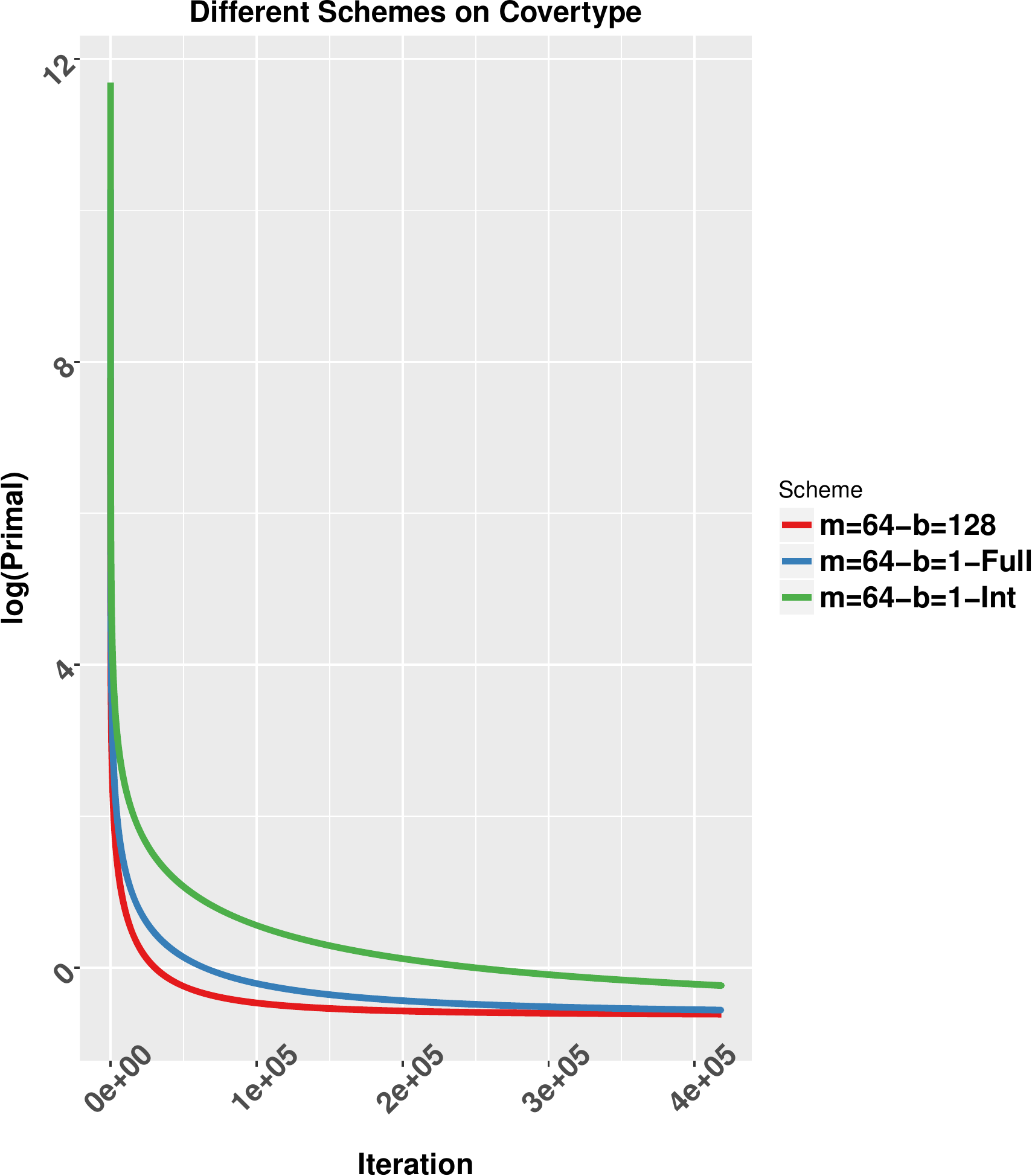}
\caption{Comparison of three different schemes a) Algorithm \eqref{alg:DiSCO} with Mini-Batching b) Standard c) Intermittent with $b=(1/\nu)=128$. As predicted the mini-batch scheme performs much better than the others.}
\label{fig:mbatch}
\end{figure} 

\section{Empirical Results}


\subsection{Intermittent Communication.} Theorem \eqref{theorem:mainThrmStoch} suggests that a data distribution with relatively smaller $\specnorm^2$ should perform favorably compared to a dataset with a higher $\specnorm^2$. 
Indeed, Figure \eqref{fig:DiSCOvsDODG} reveals that the \rcv~dataset with smaller $\specnorm^2$ doesn't suffer the effects of intermittent communication as the frequency of communication is varied. For \ctype~the performance degrades more with decreasing $\nu$.

\subsection{Comparison of Different Schemes}
We compare the three different schemes proposed in this paper. On a network of $m=64$ machines we plot the performance of the mini batch extension of Algorithm \eqref{alg:DiSCO} with batch size $128$ against the intermittent scheme that communicates after every $128$ iterations and also the standard version of the algorithm. In Figure \eqref{fig:mbatch} we see that as predicted in Theorem \eqref{theorem:mainThrm:mbatch} the mini batch scheme proposed in \eqref{eq:mbatchComm} does better than the vanilla and the intermittent scheme.

\section{Summary}

We analyzed different communication regimes and showed that distributions with smaller spectral norm are more tolerant of schemes involving less communication. Additionally we were able to improve the data-dependence factor by $\specnorm$ in Theorem \eqref{theorem:mainThrm:mbatch} over Theorem \eqref{theorem:mainThrm} in Chapter \eqref{Chapter3} after adding a mini-batch update in Algorithm \eqref{alg:DiSCO}. 

In the next chapter, assuming a fixed complete graph topology we will show that one can obtain speedups that are also data-dependence, thereby tying the analysis with previous chapters.



\chapter{Mini Batch Stochastic Gradient Descent - Complete Graph Topology} 
\label{Chapter5}
\lhead{Chapter 5. \emph{Mini Batch SGD}} 

In this chapter we explore the setting when the network is a complete graph of compute nodes. It is easy to see that this is equivalent to the mini-batch setting and therefore one would expect a speedup as we add more nodes to the network (See \cite{CotterSSS:11nips} and \cite{DekelGSX:12mini}). It would indeed be ideal to have a more general version of Theorem \eqref{theorem:mainThrm} covering the case of the complete graph, however the analysis relies on bounds from mixing rates of Markov chains which prove to be very loose in this extreme setting. As a consequence we do not recover any speedups as one would hope. This is true of all analysis (of consensus algorithms) that depend on mixing rates bounds.

This weakness of the analysis in the Chapter \eqref{Chapter3} thus provides a compelling reason to explore data-dependence in this setting using a simpler but different approach. 

In this Chapter we explore the data-dependence of Mini batching based SGD for problem \eqref{eq:optForm} and empirically in the context of support vector machines (SVM). As discussed in Chapter \eqref{Chapter1} this is different from a smoothness based analysis and therefore this data-dependence analysis essentially adds to the results presented in Chapter \eqref{Chapter3}.

\section{Mini Batches}

Problem \eqref{eq:optForm} is traditionally solved by a sequential SGD algorithm like \cite{pegasos}(for SVM) and despite the advantage of scalability these methods are inherently sequential and hence difficult to exploit in parallel computing settings. One potential way to parallelize these algorithms is through mini-batching. For SGD this corresponds to a update step performed using an average of the subgradients corresponding to $b(>1)$ samples. These computations can then be performed in parallel over $b$ processors in a distributed setting. One could then hope to achieve a perfectly linear speedup ($b$).

We develop a data-dependence understanding of mini-batches in the SGD settings and show that for these methods the minibatch speedup is related to how well \textit{conditioned} the data is. Specificly, we show that the key quantities in controlling the mini-batch speedup is $E\left[\norm{\Q_b}\right]$ (i.e. expected spectral norm) where $\Q_b$ corresponds to a prinicipal submatrix corresponding to the random subset $A \subset \{1,...,n\}$ such that $|A|=b$ and also $\rho^2$, the spectral norm of the sample gram matrix. Note that our results also apply to nonlinear SVM's.

\subsection{Mini-Batches in Stochastic Gradient Methods}

\begin{algorithm}[!htp]
   \caption{Mini Batch SGD}
   \label{alg:mbatchsgd}
\begin{algorithmic}
   \STATE {\bfseries Input:} $\{\x_i\}_{i=1}^N$, $\reg> 0$, $T \geq 1$ 
   \STATE
   \STATE {\bfseries Initialize:} set $\w(1) = {\bf 0} \in \R^d$.
   \FOR{$t=1$ {\bfseries to} $T$}
   \STATE Sample $A_t=\{i_t\}$ uniformly with replacement from $S$ such that $|A_t|=b$.
   \STATE Compute $\g_{A_t} \in  \sum_{j\in A_t}\partial\ell(\w(t)^{\trans}\x_{j})\x_{j}/b$
   \STATE $\w(t+1) = \w(t)(1-\reg\eta_t) - \eta_t \g_{A_t}\label{eq:updRule}$
   \ENDFOR
    \STATE {\bfseries Output:} $\bar{\w}(T) = \tfrac{2}{T}\sum_{t=\lfloor T/2 \rfloor +1}^T \w(t)$
\end{algorithmic}
\end{algorithm}

The sequential SGD algorithm (for e.g. Pegasos of \citet{SSSC11:pegasos}) analysis does not show any benefits to using mini-batches, same number of iterations are required even when large mini-batches are used. The main observation we make is that the $\norm{\g_{A_t}}$ can be bounded in terms of spectral properties of the data. The two quantities of interest here are $
\E[\norm{\Q_b}]$ and $1+\frac{(b-1)(N\rho^2-1)}{N-1}$.

Specifically we prove the following mini-batch result(s) for the above algorithm

\begin{theorem}\label{thm:pegtheorem}
For Algorithm \eqref{alg:mbatchsgd} on problem \eqref{eq:optForm} with a mini-batch of size $b$ we have that the expected error satisfies
\begin{align}
 E\big [ J(\bar{\w}(T)) \big ] - J(\w^*) \leq \frac{30L^2}{\reg} \cdot \frac{K_b}{b}.\frac{1}{T}
\end{align}
both for $K_b=\E[\norm{\Q_b}]$ and $K_b=1+\frac{(b-1)(N\rho^2-1)}{N-1}$.
\end{theorem}
The above result suggests that when $K_b = 1$ we get a exact linear speedup. This happens when all the points $\x_i$ are orthogonal to each other. At the other end of the spectrum when $K_b = b$ we get no speedup at all. This corresponds to all the samples being the same upto a scalar multiple. Thus in the intermediate regime one could expect a linear speedup upto a fixed mini-batch size.

This has important implications for machine learning in general since samples from a sparse data set are more likely to be near orthogonal, especially if the features are uniformly spread out, and hence because of the above discussion we could expect to get linear speedups with much larger mini-batch sizes. 

The quantity $1+\frac{(b-1)(N\rho^2-1)}{N-1}$ also appears in the general framework introduced in \cite{richtarik} applied to the SVM dual where the authors establish the dependence of the parallel speedups obtained on the quantity. Hence Theorem \eqref{thm:pegtheorem} in essence unifies mini-batch speedup results for both the primal and dual SVM formulations.  
\section{Proof of Convergence For Mini Batch SGD}

First we prove the intermediate Lemmas. 

\begin{lemma}\label{lemma:vQv}
  For any $\vv \in \R^n$, $\Q \in \R^{N \times N}$ a Gram matrix of data points and randomly sampled (without replacement) $A \subset \{1,...,N\}$ such that $|A| = b$,
  \begin{align*}
\E[ \vv^{\trans}_{[A]} \Q \vv_{[A]} ] &=
  \tfrac{b}{N}\left[ (1-\tfrac{b-1}{N-1})\sum_{i=1}^N \Q_{ii} \vv_i^2 + \tfrac{b-1}{N-1}
  \vv^{\trans} \Q \vv\right ].\\
\intertext{Moreover, if $\Q_{ii} \leq 1$ for all $i$ and $\tfrac{1}{N}\|\Q\| \leq \rho^2$, then}
\E[ \vv^{\trans}_{[A]} \Q \vv_{[A]} ] &\leq \tfrac{b}{N} K_b \norm{\vv}^2, \text{ where}
\end{align*}
\begin{equation}
  \label{eq:betab}
  K_b \eqdef 1+\tfrac{(b-1)(N \rho^2 -1)}{N-1}.
\end{equation}
\end{lemma}

\begin{proof}
\begin{align*}
\E[ \vv^{\trans}_{[A]} \Q \vv_{[A]} ] &= \E[\sum_{i \in A}\vv_i^2 \Q_{ii} + \sum_{i,j\in A, i\ne j} \vv_i\vv_j \Q_{ij}] \\
&= b \E_i [\vv_i^2\Q_{ii}] + b(b-1) \E_{i,j}[\vv_i\vv_j \Q_{ij}] \\
&= \tfrac{b}{N}\sum_i\Q_{ii}\vv_i^2 + \tfrac{b(b-1)}{n(n-1)}\vv^{\trans}(\Q-\text{diag}(\Q))\vv\\
&= \tfrac{b}{N}[(1-\tfrac{b-1}{N-1})\sum_i\Q_{ii}\vv_i^2 +
\tfrac{b-1}{N-1}\vv^{\trans}\Q \vv ],\\
&\leq \tfrac{b}{N}[(1-\tfrac{b-1}{N-1})\norm{\vv}^2 +
\tfrac{b-1}{N-1}N\sigma^2 \norm{\vv}^2] = \tfrac{b}{N}\beta_b
\norm{\vv}^2. \qedhere
\end{align*}
\end{proof}
where we used using $\Q_{ii} \leq 1$ and $\|\Q\|\leq N\rho^2$ and the expectations are over $i,j$ chosen uniformly at
random without replacement

We can now apply Lemma~\ref{lemma:vQv} to $\g_{A}$ defined as
\begin{equation}
\g_{A} = - \tfrac{1}{b} \sum_{i\in A} \partial\ell(\w(t)^{\trans}\x_{i}) \x_i
\end{equation} 

\begin{lemma}\label{lemma:nablaL}
  For any $\w\in \R^d$ and and randomly sampled (without replacement) $A \subset \{1,...,N\}$ we have
\begin{align*}
\E [ \norm{\g_{A}}^2 ] \leq L \cdot \frac{1+(b-1)(N \rho^2 -1)/(N-1)}{b}\\
\E [ \norm{\g_{A}}^2 ] \leq L \cdot \frac{\E\left[\norm{\Q_A}\right]}{b}
\end{align*}
\end{lemma}
\begin{proof}
If $\boldsymbol{\chi}\in\R^n$ is the vector with entries $\partial\ell(\w(t)^{\trans}\x_{i})$, then using Lemma \ref{lemma:vQv}
\begin{align*}
\E [ \norm{\g_{A}}^2  ] &= \E [\| \tfrac{1}{b} \sum_{i\in A} \partial\ell(\w(t)^{\trans}\x_{i}) \x_i\|^2 ]   \\  
                     &= \frac{L^2}{b^2} \E[\boldsymbol{\chi}_{[A]}^{\trans} \Q \boldsymbol{\chi}_{[A]}]\\
                     &\leq \frac{L^2}{b^2} \tfrac{b}{N}\beta_b\norm{\boldsymbol{\chi}}^2 \\
                     &\leq L^2 \frac{1+(b-1)(N \rho^2 -1)/(N-1)}{b}
\end{align*}

Additionally if we have $\boldsymbol{\chi}\in\R^b$ and $Q_A$ is the $b \times b$ submatrix of inner products corresponding to the sampled points in $A$
\begin{align*}
\E [ \norm{\g_A}^2 ] &= \E [\| \tfrac{1}{b} \sum_{i A} \partial\ell(\w(t)^{\trans}\x_{i}) \x_i\|^2 ]   \\  
                     &= \frac{L^2}{b^2} \E[\boldsymbol{\chi}^{\trans} \Q_A \boldsymbol{\chi}] \\
                     &\leq L^2 \cdot \frac{\norm{\boldsymbol{\chi}}^2 \E\left[\norm{\Q_A}\right]}{b^2} \\
                     &= L^2 \cdot \frac{\E\left[\norm{\Q_A}\right]}{b} \\
\end{align*}
\end{proof}

Now we have for the proof of Theorem \ref{thm:pegtheorem}

\begin{proof}
 Unrolling the iterate from Algorithm \ref{alg:mbatchsgd} with $\eta_t=1/(\reg t)$ yields
  \begin{equation}\label{eq:hs8js8s}
  \vc{\w}{t} = -\frac{1}{\reg (t-1)}\sum_{\tau=1}^{t-1} g^{(\tau)},
  \end{equation}
where $g^{(\tau)}\eqdef \g_{A_\tau}$. Using  the inequality
$\|\sum_{\tau=1}^{t-1}  g^{(\tau)}\|^2  \leq (t-1) \sum_{\tau=1}^{t-1} \|g^{(\tau)}\|^2$, we now get
\begin{align*}\label{eq:js8sjs800}
\E [ \|\vc{\w}{t}\|^2 ] &\leq \sum_{\tau=1}^{t-1}\frac{\E [\|g^{(\tau)}\|^2 ]}{\reg^2(t-1)} \\
                       &\leq \frac{L^2}{\reg^2} \cdot \frac{\beta_b}{b},
\end{align*}
Then using the above we finally have
\begin{align*}
E\left[\norm{\boldsymbol{\nabla}^{(t)}}^2\right] &\leq 2(\reg \norm{\w(t)} +L^2\frac{\beta_b}{b})\\
\E [\|\boldsymbol{\nabla}^{(t)}\|^2] &\leq 2(\reg^2 \E[ \|\w(t)\|^2  ] + L^2\frac{\beta_b}{b}) \\                                &\leq \frac{4L^2\beta_b}{b}
\end{align*}
where $\boldsymbol{\nabla}^{(t)} = \g_{A_t} + \reg\w(t)$
The performance guarantee is now given by the analysis of SGD with tail
averaging (Theorem~5 of \cite{RakhShamir:12}, with $\alpha=\tfrac{1}{2}$ and $G^2=4\tfrac{\beta_b}{b}$).
\end{proof}

As a constructive exercise we show the mini-batching result (without proof) for the case of convex constrained objectives applied to the SVM problem.

\subsection{Constrained Convex Objectives - SVM}

However in case one is interested in a constrained SVM formulation
\begin{equation}
  \label{eq:constSV}
  \min_{w\in\R^d,\norm{\w} \leq B}  F_B(\w) =\frac{1}{n} \sum_{i=1}^{n} \max \left[ 0,1 - y_i\langle \w,\x_i \rangle \right]  
\end{equation}
then for projected sub-gradient descent with iterates of the form
 \begin{equation}
  \label{eq:constsgd}
  \w(t+1) \leftarrow \Pi_B\left(\w(t) - \eta_t \nabla
  \hat{H}_{A_t}\w(t) \right)
\end{equation}
where $\Pi_B(w)$ is a projection onto $\norm{\w}\leq B$. We obtain a similar result for the averaged iterate $\bar{\w}(T)$ with a step-size $\eta_t = (B\sqrt{b/K_b}).(1/\sqrt{t})$

\begin{theorem}\label{thm:consttheorem}
After $T$ iterations of constrained SGD with Minibatch of size $b$ we have that for the averaged iterate $\bar{\w}(T) = \sum_{t=1}^T \w(t)/T $
\begin{equation}
  \label{eq:constsubopt}
  \E\left[ F_B(\bar{\w}(T)) \right] -\min_{\norm{\w}\leq B} F_B(\w) \leq \sqrt{ \frac{ (K_b/b) B^2}{T}}
\end{equation}
both for $K_b=\beta_a=\E[\norm{\Q_b}]$ and $K_b=\beta_b=1+\frac{(b-1)(\norm{\Q}-1)}{n-1}$.
\end{theorem}

In this setting one needs to compute $K_b$ to be used in the step-size. So it is fruitful to ask whether $\beta_a$ or $\beta_b$ is a better option. Estimating $\beta_a$ is simply a matter of sampling $m$ principal random submatrices of size $b$ and averaging their respective spectral norms. In contrast estimating $\beta_b$ requires the computation of the spectral norm of the entire data matrix. For extremely large data sets this can be a huge bottleneck.

Moreover from a practicioners persepctive to get a good idea of the speedups potentially obtained from a computational setting (no. of cores, threads)and a given data set it can be useful to obtain estimates of $b/\beta_a$ (b-ratio) prior to executing the algorithm. 
 \begin{figure*}[ht!]\label{table:speedups}
 \begin{center}
\begin{tabular}{rr}
\includegraphics[width=3.5in]{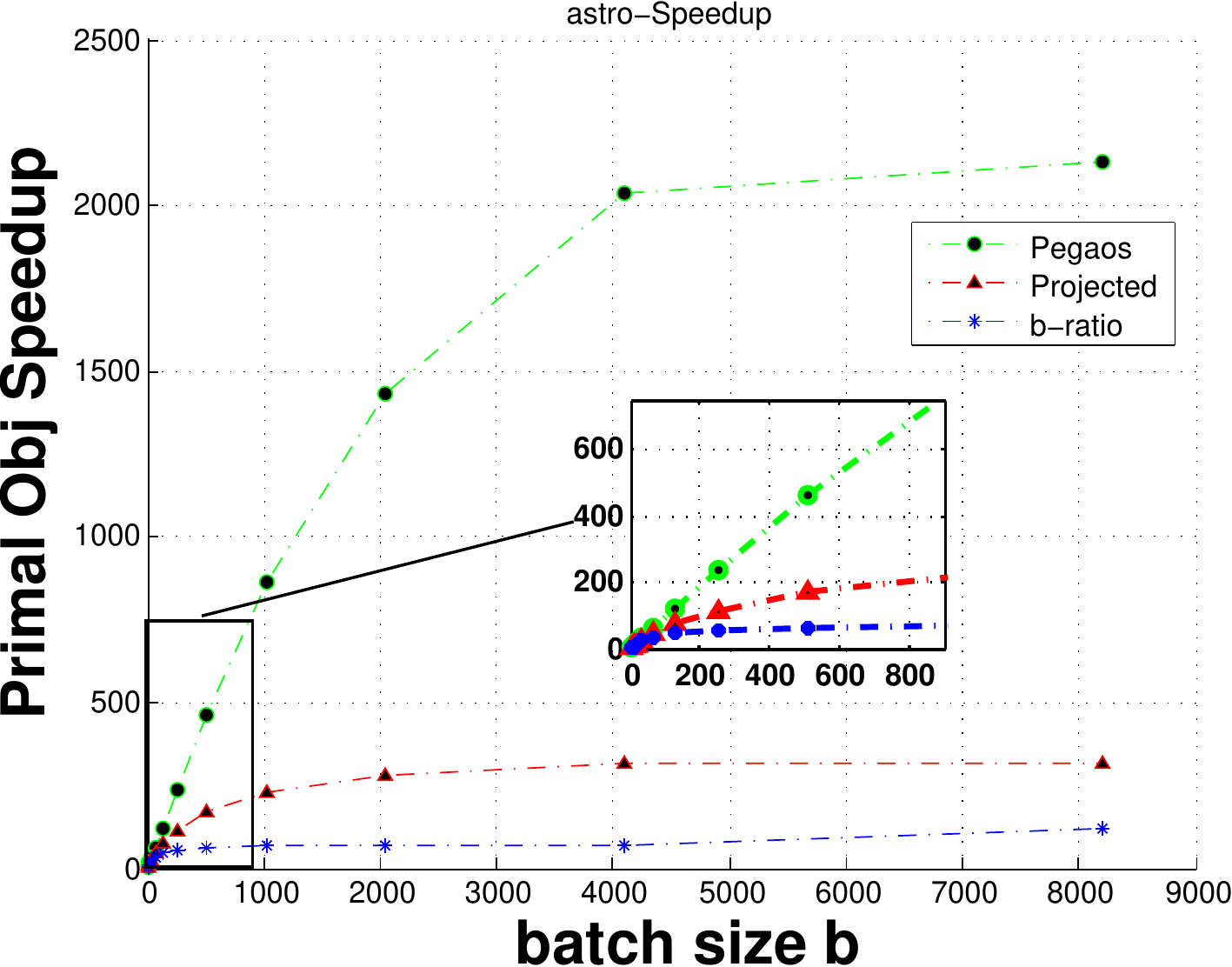}\\
\includegraphics[width=3.5in]{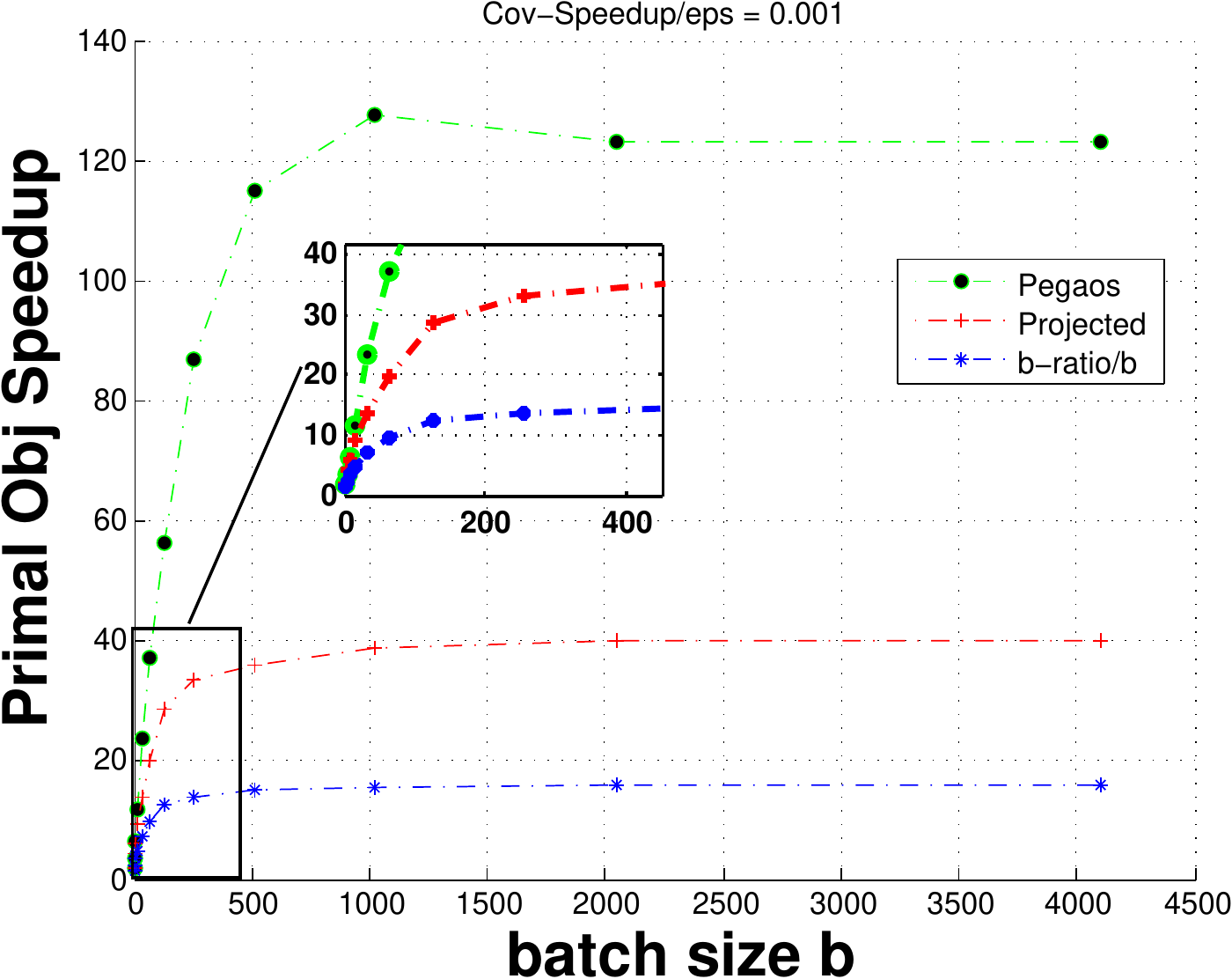}
\end{tabular}
\end{center}
 \caption{\small Speedup obtained (left vertical axis) to optimize a $0.001$-accurate \emph{primal} solution for different mini-batch sizes $b$ (horizontal axis). A) Astro-ph (astro) B) Covertype (cov)}
\end{figure*}
\newpage
\section{Empirical Validation}

Table 1 shows plots of speedups obtained for two datasets with varying sparsity. Astro-ph (Sparsity $0.08\%$) and Covertype (Sparsity $22.2\%$). The speedups were estimated by first computing a near optimal solution by running Pegasos and Projected SVM with a batch size $b=1$ for $2 \times 10^6$ iterations and then for each mini-batch size we computed the number of iterations to get within $\epsilon = 0.001$ of the computed approximate solution. The speedups are then simply the ratio of the number of iteration with $b=1$ divided by the number of required iterations for the different batch sizes. This was done for both pegasos with mini-batches (green) and projected SVM with mini-batches (red). The blue line corresponding to the $b$-ratio defined earlier represents a lower bound on the speedup obtained as per Theorem \ref{thm:pegtheorem} and \ref{thm:consttheorem}. Note that $K_b=\beta_a$ in these experiments and the results were averaged across $5$ runs.

We can see that in accordance with our theoretical prediction, since astro-ph is a lot sparser than covertype, we get near linear speedups (for pegasos) up to a batch size of $b=1024$. While for covertype we get near linear speedups only upto batch size $b=128$. 

\section{Summary}

We explored the data-dependence of a mini-batch approach to SGD and discovered that the speedups obtained depend on spectral properties of data covariance (mini-batch sample and overall sample). These results are of importance to the practitioner since we demonstrated empirically that sparse datasets are more amenable to speedups. 

In the next chapter we present an empirical exploration of data-dependence in an extreme form of distributed optimization where the nodes only communicate once.



\chapter{One Shot Averaging and Data Dependence - Empirical Evidence} 
\label{Chapter6}
\lhead{Chapter 6. \emph{One Shot Averaging}} 

As discussed in Chapter \eqref{Chapter1} some authors have proposed methods in which the nodes in the network process their local data and average their iterates only once~\citep{MannMMSW:09,McDonaldHM:2010,ZinkevichWSL:10,ZhangDW:12}. Zhang et al.~\cite{ZhangDW:12} recently provided an analysis of this procedure under smoothness and bounded moment assumptions and showed that the rate of convergence is $\mc{O}\left(\frac{1}{mn} + \frac{1}{n^{3/2}} \right)$.  

We conjecture that this rate of convergence for non smooth objectives depends on $\rho^2$. To add weight to this statement we provide  experimental evidence that seems to point towards the conjecture. We leave deriving a refined analysis of the average-at-the-end procedure incorporating this dependence for future work.

In all the experiments on real datasets described in Chapter \eqref{Chapter1} we looked at $\ell_2$-regularized hinge loss classification as a representative for problem \eqref{eq:optForm}.  

\section{Impact of $\rho^2$ on \textit{Average-at-the-end}}

We empirically corroborate our conjecture that the convergence rate of the mean squared error for the \textit{average-at-the-end} scheme also depends on $\rho^2$. The objective function is the $\ell_2$-regularized squared hinge loss. We measured the relative mean squared error (RMSE) $\E[\norm{\bar{\w}_n - \w^{*}}^2]/\norm{\w^{*}}^2$ and test errors.  We compared \textit{average-at-the-end} for partitions of the dataset of $N$ samples on $m$ machines (for varying $m$), a centralized SGD with $N$ samples on one machine, and a ``local'' single SGD with $n = M/n$ samples on one machine.  The expected relative mean squared error was estimated by computing an average over $50$ runs.

\begin{figure}[htp!]
\centering
\includegraphics[width=4in]{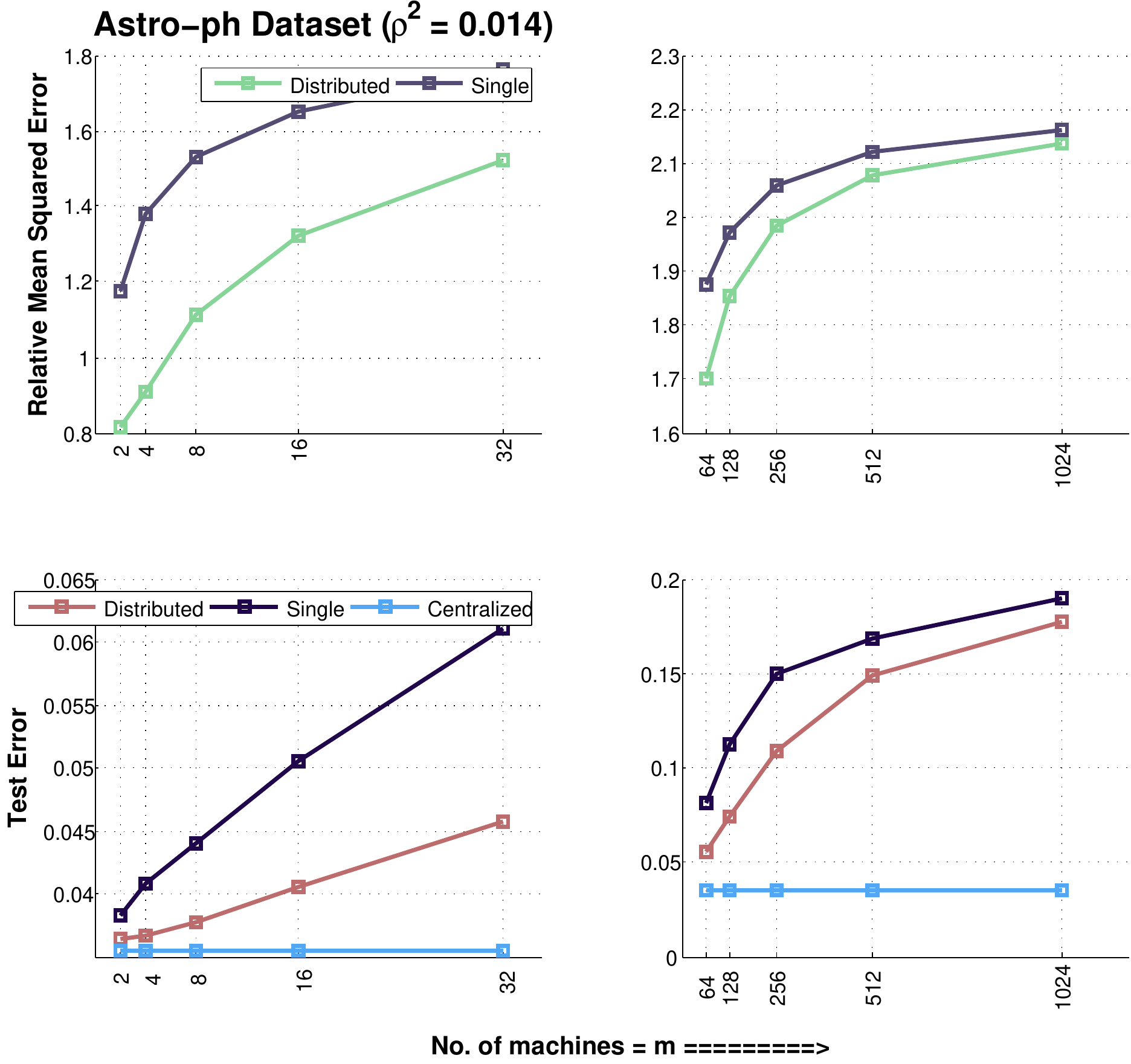}
\label{fig:dist-astro}\\
\includegraphics[width=4in]{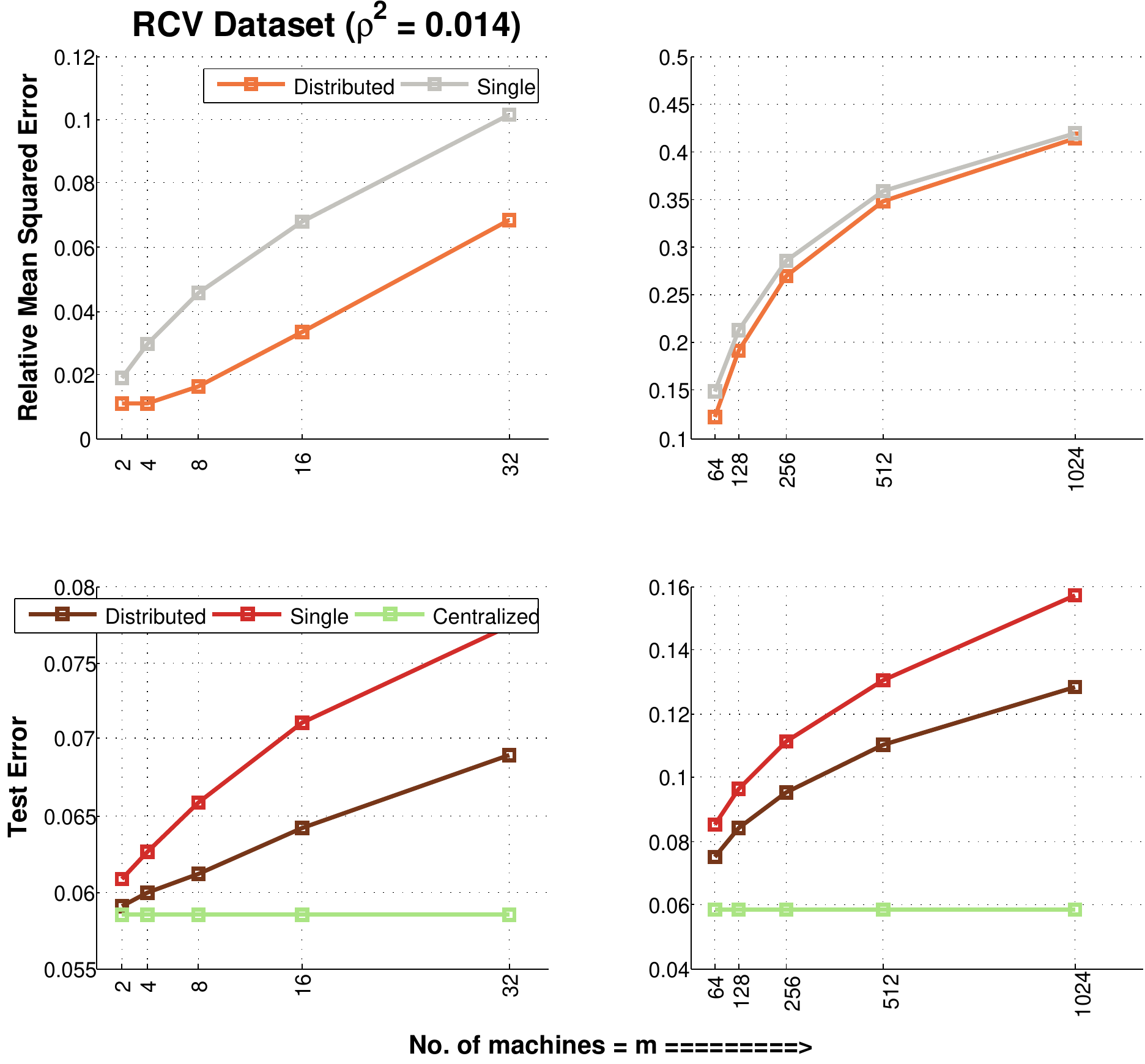}
\label{fig:dist-rcv}
\caption{Average-at-end SGD Performance on good datasets (\astro~with $\rho^2=0.014$ and \rcv~with $\rho^2=0.013$) as we increase the number of machines $m$. For datasets with smaller $\rho^2$ the performance of the average-at-end strategy is significantly better than a single machine output, but worse than the centralized scheme. }
\label{fig:good}
\end{figure} 

\begin{figure}[htp!]
\centering
\includegraphics[width=4in]{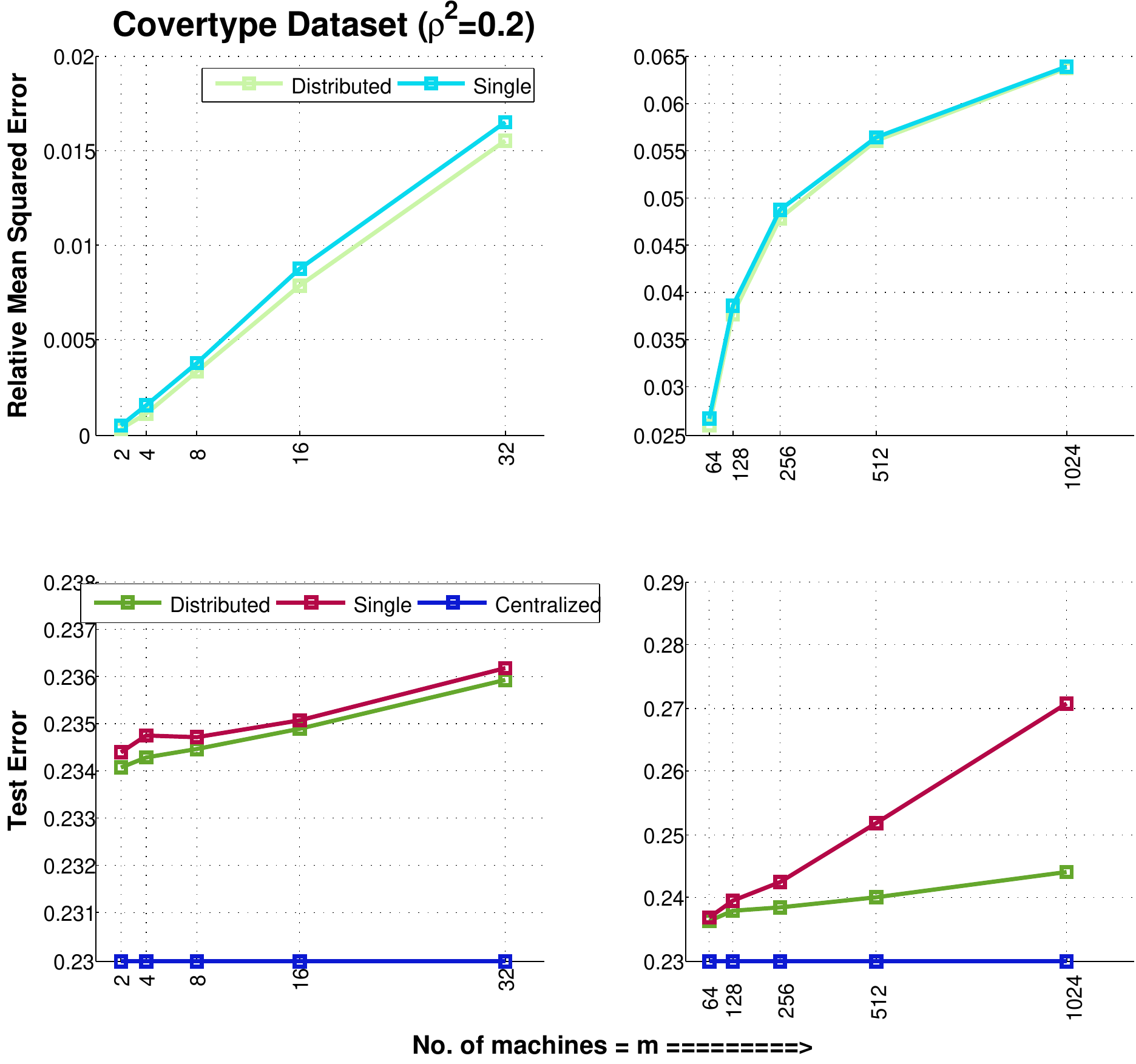}
\label{fig:dist-covertype} \\
\includegraphics[width=4in]{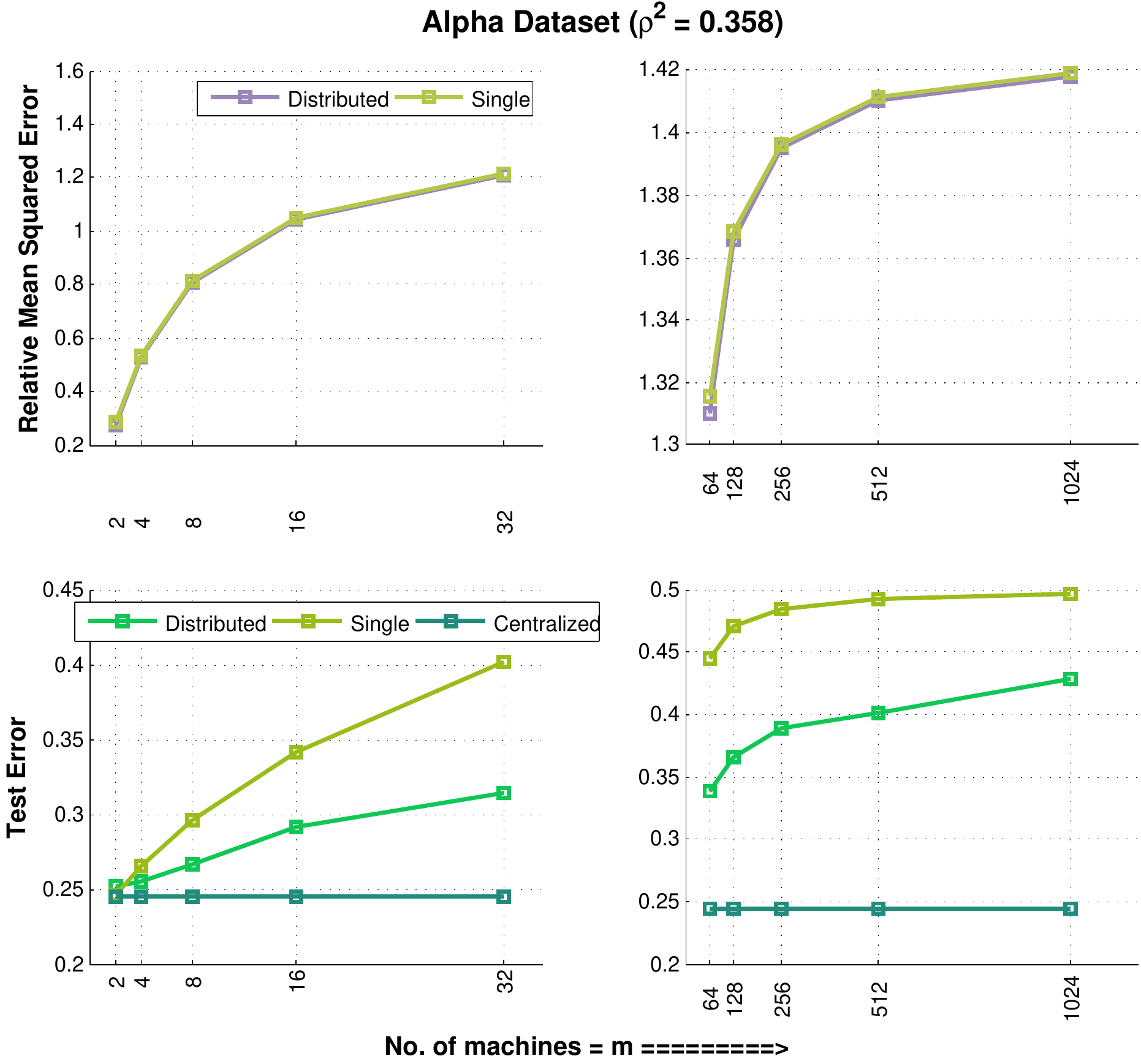}
\label{fig:dist-alpha}
\caption{a) Distributed SGD Performance on bad datasets as we increase the number of machines $m$. It can be seen that with these datasets with larger $\rho^2$ the performance of the average-at-end startegy is no better than a single machine output. a) Covertype ($\rho^2=0.2$) b) Alpha dataset ($\rho^2=0.35$).}
\label{fig:bad}
\end{figure}

On all the datasets we run SGD with step-size $\eta_t= 1/(\lambda (t+t_0))$ with $\lambda$ given in Table \ref{tab:data}. The algorithm is run for one pass over the data for both local and centralized schemes. Figure \ref{fig:good} shows the results for datasets with lower $\rho^2$ (\astro~and \rcv); these results indicate that for the \textit{average-at-end} scheme the RMSE is lower than for a single machine and this gap decreases (as conjectured) when we partition the data into more and more machines. Indeed, even for $m>256$ machines we get some benefit from \textit{average-at-the-end}. This is in stark contrast to Figure \ref{fig:bad} wherein we see that for the datasets with a relatively larger $\rho^2$ (\ctype~and \aset) there is no significant gap in the RMSE of \textit{average-at-the-end} and the single machine scheme. This suggests that the spectral norm of the covariance matrix also controls the error of the \textit{average-at-the-end} scheme. 

For the test error itself the \textit{average-at-end} performs better than a single machine but it is much less clear how this gain depends on the properties of the data distribution. Note that for this set of experiments the theoretical results obtained by Zhang et al ~\cite{ZhangDW:12} do not apply since the smoothed Hinge loss is not twice differentiable.

\section{Summary}

We conjectured that the mean squared error for the \textit{average-at-end} approach also depends on $\rho^2$ and provided empirical evidence. From a theoretical perspective we have some understanding of how the data-dependence terms should appear in the bounds but the analysis of the dependence on sample size is incomplete and we relegate it to future work.

In the final chapter we describe an efficient methodology to find a doubly stochastic matrix $\bP$ to facilitate faster convergence for a given graph topology for general graph based optimization problems. 



\chapter{Optimizing Doubly Stochastic Matrices} 
\label{Chapter6}
\lhead{\emph{Doubly Stochastic Graph Optimization}} 

In the previous Chapters the graph matrix $\bP$ was assumed to be given. However for consensus optimization algorithms we know that the convergence rate depends on the mixing rate of the Markov random walk described by the transitions matrix $\bP$. This leads us to the graph optimization problem.

Graph optimization is a class of problems that assigns edge weights or transition probabilities to a given graph which minimize a given criterion usually subject to some connectivity and other constraints. An example is the fastest mixing Markov chain problem \cite{DiaconisBoyd}, where the object is to assign transition probabilities that minimize the mixing rate of a Markov random walk on a given graph. Aside from consensus optimization the mixing rate problem has been shown to arise in a class of gossip algorithm problems \cite{BoydGhoshBalaji} where the object is to find an averaging algorithm or equivalently a transition matrix such that the averaging time over the graph is minimized. Other related problems that involve optimization over spectral functions of doubly stochastic matrices include
\begin{enumerate}
\item Minimizing Effective Resistance on a Graph \cite{GhoshBoyd}, where the idea is to choose a random walk on the graph that minimizes the average commute time between all nodes.
\item Finding the best doubly stochastic approximation to a given affinity matrix. This arises in the context of spectral clustering in machine learning.
\end{enumerate}

Here we consider the fastest mixing Markov chain(fmmc) problem and present an efficient approximate solution based on using a smaller subset of the space of large doubly stochastic transition matrices.  This involves a computationally expensive pre-processing step but needs to be executed only once for a given problem. In the next section we describe the notation and the underlying results used to reduce the dimensionality of the problem.

\section{Problem Formulation}

\subsection{Fastest Mixing Markov Chain}

Consider a symmetric Markov chain on a graph $G$ with a transition matrix $\bP \in \Re^{m\times m}$. The stationary distribution in this case is the uniform distribution $\mbf{\pi}=(1/n)\mbf{1}^{T}$ and the mixing rate measures how fast an initial distribution converges to the uniform. This rate of convergence is measured by the second largest eigenvalue modulus (SLEM) of $\bP$, $\mu(\bP)=\max(\lambda_2(\bP),-\lambda_n(\bP))$:  the smaller it is the faster the Markov chain converges. Here $\lambda_2(\bP)$ and $\lambda_n(\bP)$ are the second largest and the smallest eigenvalues of $\bP$. The problem of finding the fastest mixing Markov chain was described in \cite{DiaconisBoyd}. It can be written as
\begin{align*}
& \underset{\bP}{\text{min}}\ \ \mu(\bP) \\
& \text{subject to} \\
& \bP\mathbf{1} = \mathbf{1},\ \bP^{\trans} = \bP, \ \bP \geq \mbf{0} \\
& \bP_{ij} = 0 \ \textbf{if} \ \{i,j\} \notin E.
\end{align*}

The problem can be expressed as an SDP and solved using standard techniques and also for very large graphs with more $100,000$ or more edges a subgradient method is presented in \cite{DiaconisBoyd}.
\subsection{BN Decomposition}
Let $G=(V,E)$ be an undirected graph with $m$ vertices and $k$ edges . There is a Markov chain associated with the graph such that the weight of each edge is the transition probability $p_{ij}$ from the $i$ to the $j$th node. These transition probabilities are described by a doubly stochastic $m\times m$ matrix $\bP$. The entries are zero only if there is no corresponding edge in the given graph.

We can write the symmetric stochastic matrix $\bP$ as a convex combination of Permutation matrices using the following result due to Birkhoff \cite{Minc}\\

\begin{theorem}
Any $m\times m$ matrix $\bP$  is doubly stochastic if and only if there are $M$ $m\times m$ permutation matrices $\bP_1$,...,$\bP_M$ and positive scalars $\theta_1$,...,$\theta_M$ such that \\
\begin{equation}
\bP=\sum_{i=1}^{M}\theta_i \bP_i \ \ \ \text{and} \ \ \ \sum_{i=1}^{M}\theta_i=1
\end{equation}
\end{theorem}

\begin{table}
 Input- A $m\times m$ doubly stochastic matrix $\mbf{A}$
\begin{enumerate}
\item for $ l=1:m^2+m-2$
\item Using Bipartite Matching find a permutation $\pi_{l}$ of vertices ${1,...,m}$ such that each $\mbf{A}_{i,\pi_i}$ is positive
\item $\theta_l = \min_{i} \ (\mbf{A}_{i,\pi_i})$
\item $\mbf{A} = \mbf{A}- \theta_l \bP_{\pi_l}$, where $\bP_{\pi_l}$ is a permutation matrix corresponding to $\pi_l$.
\item Exit if all entries of $\mbf{A}$ are zero
\end{enumerate}
Output $(\theta_i,\bP_{\pi_i})$ such that $\mbf{A}=\sum_{i=1}^{M}\theta_i \bP_{\pi_i}$
\caption{BN Decomposition Algorithm for a Doubly Stochastic Matrix}
\end{table}

Some bounds on $M$ exist in the literature \cite{Minc}, but as we shall see it becomes irrelevant for our purposes. Now given a matrix $\bP$ an algorithm (table 1) based on a proof given by Dulmage and Halperin is described in \cite{MarshallOlkin}. It involves bipartite graph matching and to each such matching corresponds a permutation matrix.  These permutation matrices define a basis for a certain subset of the space of $m\times m$ doubly stochastic matrices.  

\subsubsection{Identifying Basis Subset}

If we have a reasonable choice for a Markov chain that mixes fast, for e.g. the Metropolis Hastings chain. We hope to use its BN decomposition to select a permutation basis.  Having then identified such a permutation basis we can hope to solve for the fastest mixing chain by optimizing over the $\theta$'s, keeping the basis matrices fixed. However the number of basis matrices (or equivalently the problem size could be very large). Can we then identify a smaller subset of basis matrices and search over those?

Consider the initial transition matrix $\bP^{m}$, an input to the decomposition algorithm to obtain a permutation basis for the space of DS matrices to search on. The BN procedure returns a $\theta=(\theta_1,...,\theta_M)$ vector and  $M$ permutation matrices $(\bP_1,...,\bP_M)$. Intuitively, since higher values of $\theta_i$ contribute more to probability weight on a particular edge, it makes sense to ignore altogether very small $\theta_i$'s and hence the corresponding $\bP_i$'s. This is reasonable if we make the assumption that since our initial heuristic choice $\bP^{m}$  is a good one any improvement that we hope to find over this one is structurally similar to $\bP^{m}$. Then we use the following procedure values to filter out insignificant $\bP_i$'s
\begin{algorithm}[!htp]
   \caption{Select Basis}
   \label{alg:DiSCO}
\begin{algorithmic}
   \STATE {\bfseries Input:} $\bP_1,...,\bP_M$ and $k << M$ 
   \STATE
      \FOR{$i=1$ {\bfseries to} $M$}
    \STATE Compute $r_i=\parallel \bP^{m}-\theta_i \bP_i \parallel_{F}$ for all $i$ 
   \ENDFOR
    \STATE {\bfseries Output:} $\{\bP_i\}$ corresponding to $k$ smallest $r_i$.
\end{algorithmic}
\end{algorithm}

Later in the experiments section we will show the results for different values of $k$ for a fixed $M$. Once we have the $\bP_i$'s we have in essence fixed a subset of Birkhoff polytope for our optimization procedure to search over.  The parameter space for the search is then defined by the $\theta=(\theta_1,...,\theta_k) \in \Re^k$ such that this new $\theta$ lies in the probability simplex and $\bP(\theta)=\sum_{i=1}^{k}\theta_i \bP_i$. Clearly the SLEM is also a nonlinear function of $\theta$ and will be written as $\mu(\theta)$. Note that since $\bP(\theta)$ is symmetric, $\bP(\theta)=(\bP(\theta)+\bP(\theta)^{T})/2=\sum_{i=1}^{k}\theta_i(\bP_i+\bP_i^{T})/2$. Hence our basis matrices are $(\bP_i+\bP_i^{T})/2$ for each $i$ to maintain the symmetry constraint.

In the ensuing sections we show experimental evidence and demonstrate that it is possible to truncate the space considerably and still obtain reasonable results. 
\subsection{Basis Subset Optimization for Fastest Mixing Chain} 

The most commonly used heuristic for fast mixing is the Metropolis-Hastings random walk. To obtain a Markov chain with the uniform stationary distribution the following transition matrix is constructed 
	\[
	P_{ij} = \left\{ \begin{array}{ll}
		\min\{ 1/d_i, 1/d_j \} & (i,j) \in \mc{E} \\
		\sum_{(i,k) \in \mc{E}} \max\{ 0, 1/d_i - 1/d_k\} & i = j \\
		0 & (i,j) \notin \mc{E}
		\end{array}
		\right.
	\]
Using the procedure to identify the basis subset in the previous section. We BN-decompose the $\bP^m$ matrix  to obtain the subset space to search over. Next we use a subgradient method to solve the fastest mixing problem on this smaller subset of DS matrices constrained by the fixed chosen permutation basis.

\subsubsection{Subgradient Method}
In this parameter space the optimization problem becomes
\begin{align}
&\min \ \mu(\theta)  \notag \\
& \text{subject to } \ \sum_{i=1}^{k} \theta_i =1 ,\theta_i \geq 0   
\end{align}  

The SLEM in general is a non-differentiable function of the entries of the matrix and therefore as a function of the $\theta$ parameter. We use subgradients to solve the optimization problem. It is easily shown that the subgradient $g$ corresponding to the equation below is dependent on the eigenvector corresponding to the second largest eigenvalue in magnitude (See \cite{DiaconisBoyd})
\begin{align}
\mu(\tilde{\theta}) &\geq \mu(\theta)+\mbf{v}(\theta)^{T}(\bP(\tilde{\theta})-\bP(\theta))\mbf{v} \notag \\
                    &=\mu(\theta)+\delta \theta^{T}g
\end{align}
and is given by $\g=(\mbf{v}(\theta)^{T}\bP_i\mbf{v}(\theta),...,\mbf{v}(\theta)^{T}\bP_k\mbf{v}(\theta))$. Where ($\bP_1,...,\bP_k$) are the permutation basis and $\mbf{v}(\theta)$ is the eigenvector  corresponding to the second largest eigenvalue in magnitude. Thus at each iteration an eigenvector computation is required. 

The projected subgradient method then proceeds as usual on this considerably smaller $k$-dimensional space and involves a projection step onto the probability simplex for which we use an efficient algorithm described in [8]. 

The overall algorithm can be given as
\begin{algorithm}
\begin{algorithmic}
\STATE {\bfseries Input:} $\bP^m$ and $k << M$ 
\STATE Compute the Metropolis Hastings matrix $\bP^m$ for a given Graph $G$.
\STATE Compute the BN decomposition of $\bP^m$ to obtain a permutation basis such that $\bP^m=\sum_{i=1}^{M}\theta_i^{m}\bP_i$.
\STATE Truncate the basis by ignoring $\bP_i$ corresponding to the $(M-k)$ $\theta_i^{m}$'s returned by the truncation procedure.
\STATE Solve the optimization problem (5) using the subgradient method.
  \STATE {\bfseries Output:} $\{\theta_i\}$ corresponding to $k$ smallest $r_i$.
\end{algorithmic}
\end{algorithm}






\begin{figure}
\centering
\includegraphics[width=4.0in]{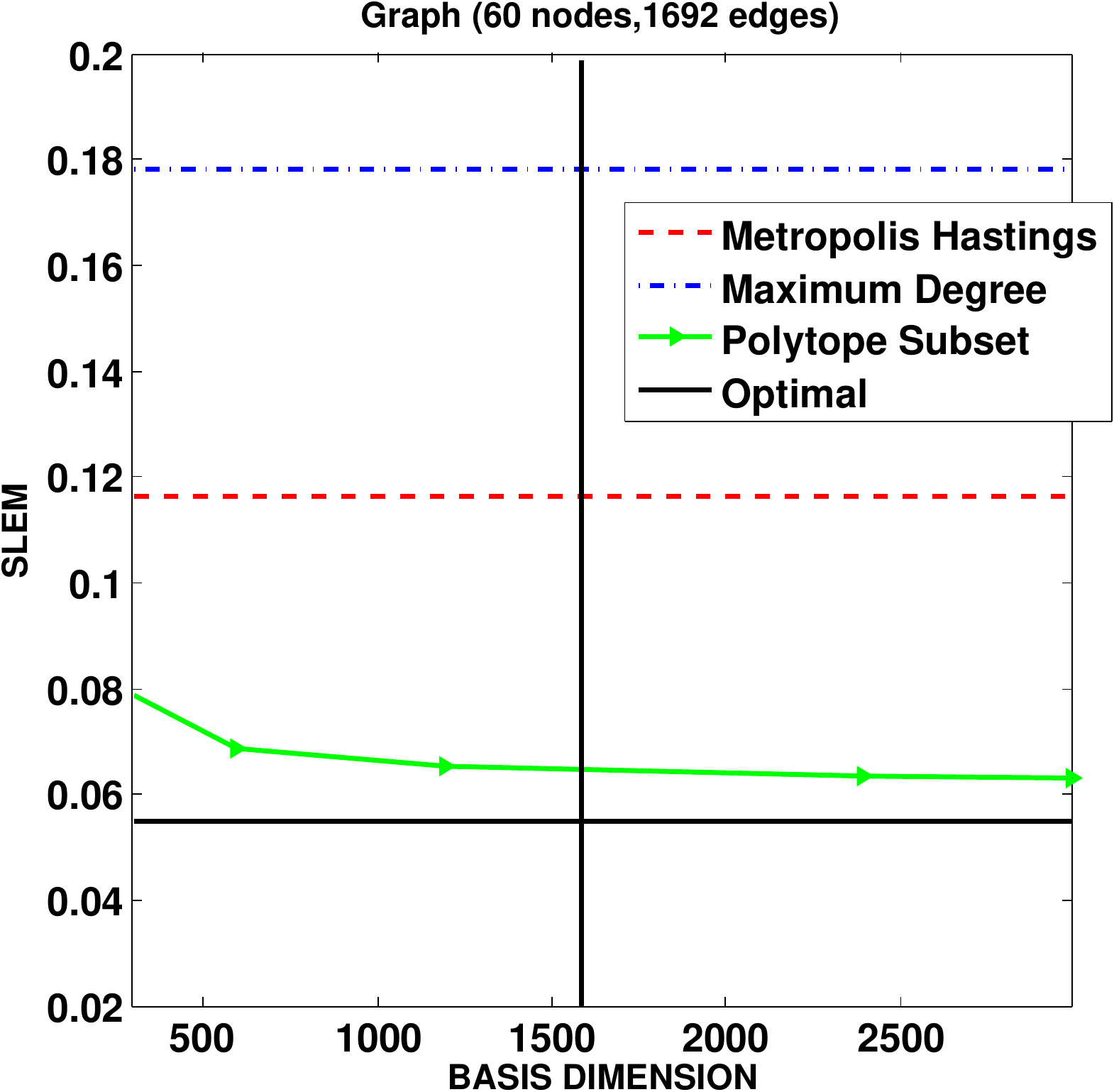}
\caption{SLEM for a fixed graph with varying basis dimension size for our method.  The horizontal axis is the number of basis
elements used, i.e. the number of variables being optimized.  The
vertical line is the number of edges, i.e. the number of variables
being optimized in a direct optimization approach.The parallel lines include the corresponding SLEM values for Metropolis Hastings and the Optimal. The SLEM values comes close to the optimal as we increase the basis dimension.}
\label{fig:slem}
\end{figure}

\begin{figure}[htp!]
\centering
\begin{tabular}{c}
\includegraphics[width=4in]{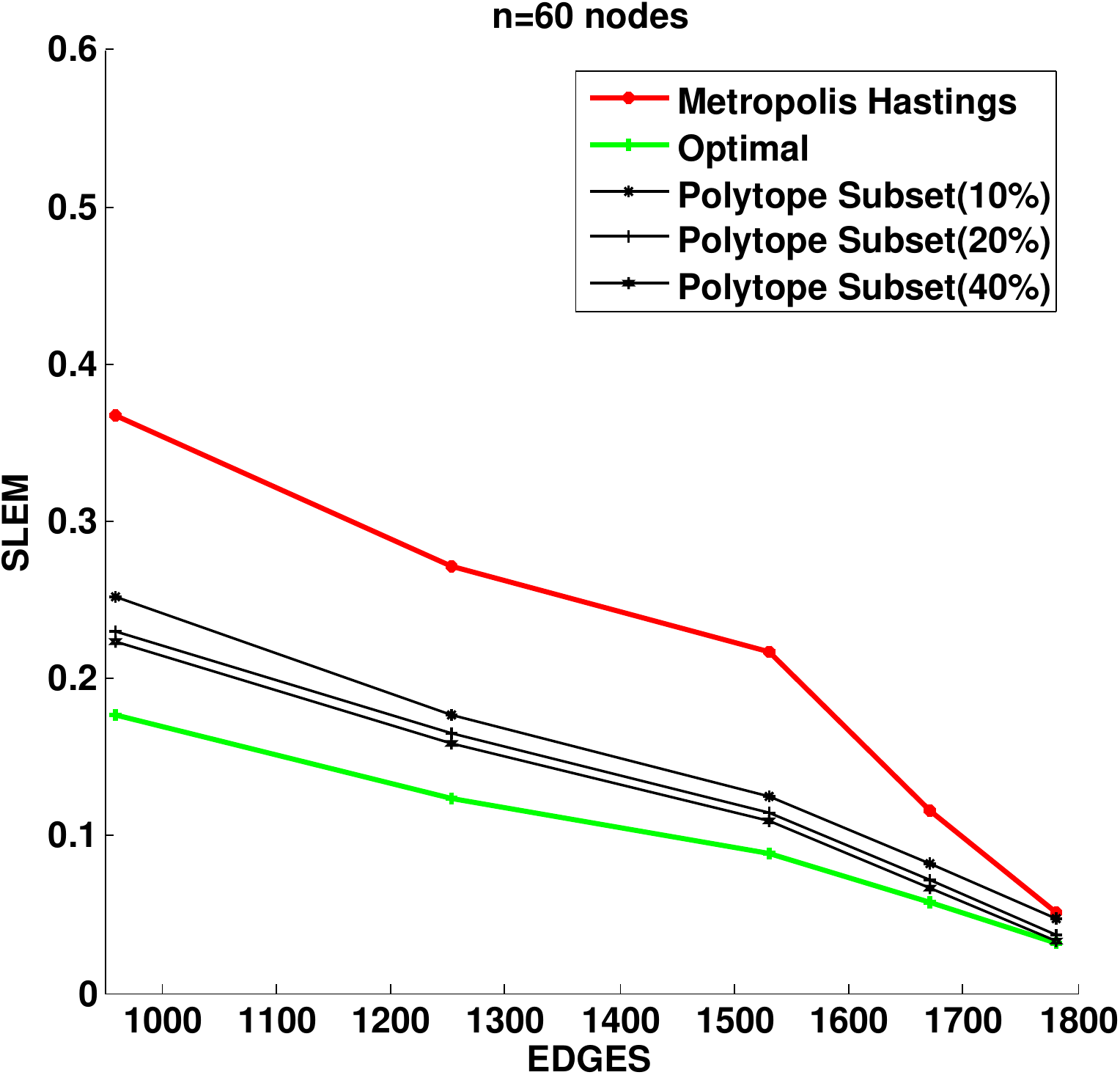}\\
\vspace{-11em}\\
\includegraphics[width=5in]{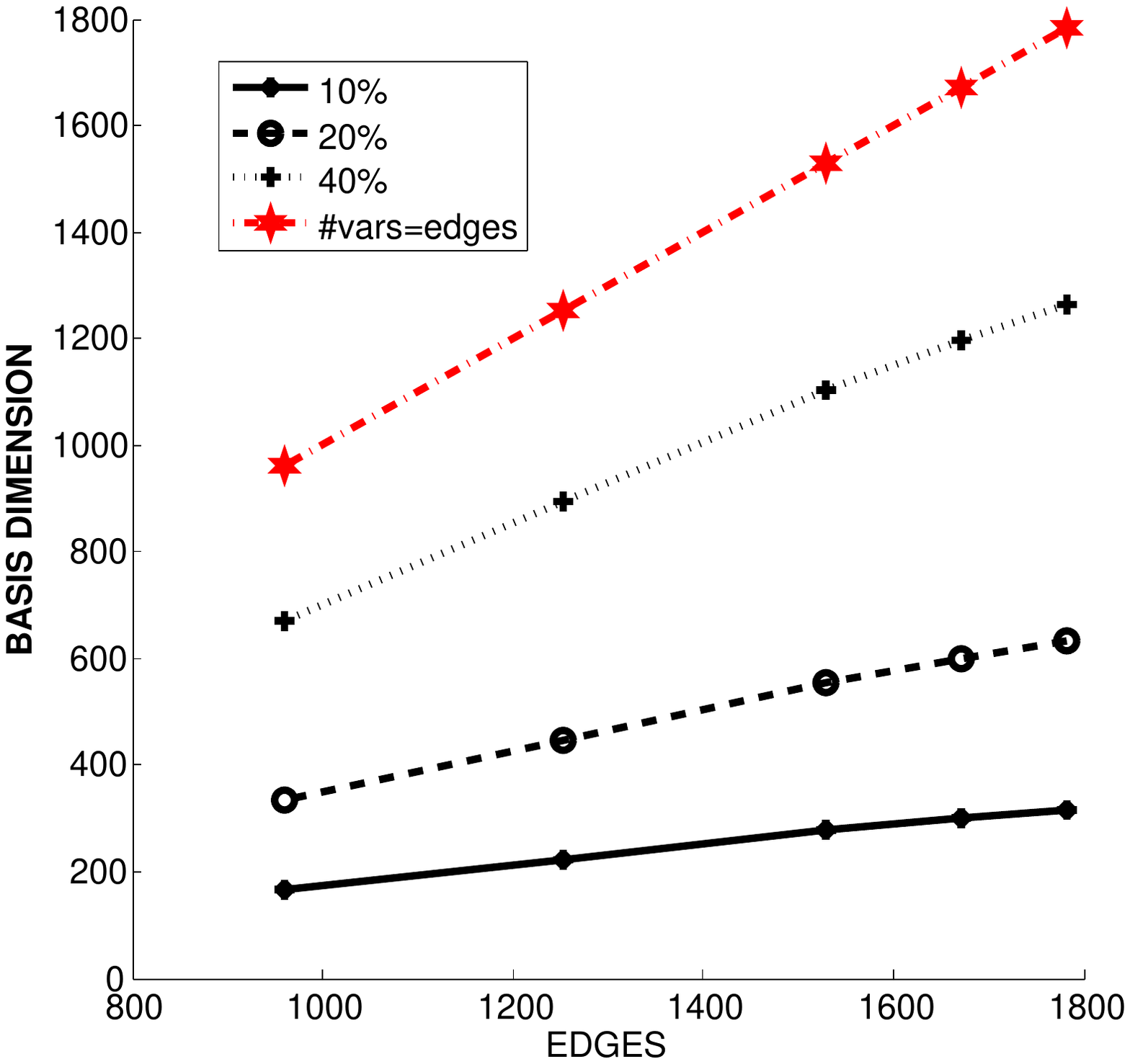}
\end{tabular}
\vspace{-7em}

\caption{a) SLEM for graphs with varying no. of edges. The \% indicates the proportion of total basis matrices used. Even with $10\%$ of the total variables, performance is much better than the Metropolis Hastings chain. b) The plot shows the basis dimension (problem size)(Y) for each of the graphs (edges (Y)) in the previous plot for each of the \% levels. The number of variables are considerably less than edges in the graph at the 10\% level. The top line shows the no. of variables as the number of edges in the direct optimization approach.}
\label{fig:varslem}
\end{figure}

\begin{figure}
\centering
\includegraphics[width=4.0in]{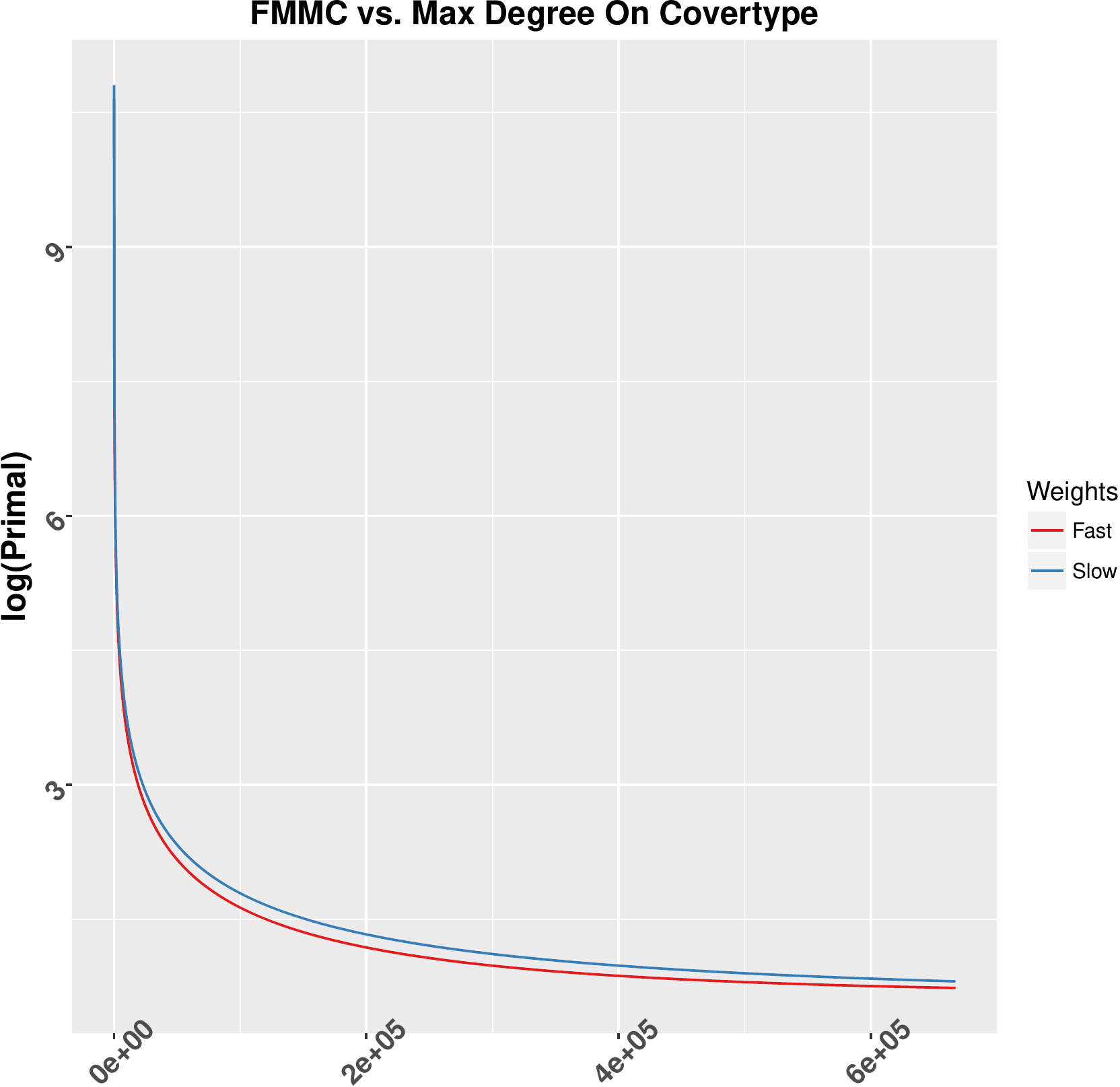}
\caption{Objective value on \ctype~dataset with the Max Degree and the Fast Mixing Markov Walk. The objective decays slightly faster than the Max Degree chain.}
\label{fig:FMMCvsMaxDegree-crop}
\end{figure}

\subsection{Simulation}

\removed{Figure \eqref{fig:mixing} shows the effect of Algorithm \eqref{alg:DiSCO} on a dataset with the weights set according to the fast mixing Markov chain versus the Metropolis Hastings graph. We get slightly faster convergence on fast mixing graph. Here $m=256$ and the topology is a $k$-regular graph. The results were averaged over $5$ iterations.}

Next we present results on a small graph with $60$ nodes generated uniformly as described in fastest mixing paper \cite{DiaconisBoyd} . Figure \eqref{fig:slem} shows a graph with $n=60$ nodes and the corresponding SLEM's as we increase the basis dimension or equivalently the number of permutation basis matrices used. The middle line indicates the number of edges in the graph. It can be seen that even with a limited basis dimension we do considerably better than the Metropolis-Hastings chain and stay relatively close to the optimal. Thus the basis dimension can also be used as a knob to control the accuracy vs efficiency tradeoff.

Figure \eqref{fig:varslem}(a) shows the SLEM values on graphs with $n=60$ nodes and varying number of edges. We can see that even with $10\%$ of permutation basis we stay relatively close to the optimal value and significantly better than the Metropolis Hastings chain. In figure 2(b) we can see in the top line the gain we get if were to solve the optimization problem with the number of variables equal to the number of edges as opposed to a fraction of the total basis dimension.

Finally, in Figure \eqref{fig:FMMCvsMaxDegree-crop} we show the performance of the fastest mixing Markov chain on Algorithm \eqref{alg:DiSCO} and compare it to the max-degree Markov chain defined in \cite{DiaconisBoyd}. On the dataset \ctype~ with $m=512$ nodes the fmmc graph seems to perform slightly better. Each trajectory was averaged over $10$ runs.


%

\section{Summary}

In this chapter we presented an efficient subgradient algorithm based on the Birkhoff-von Neumann decomposition to get the approximate fastest mixing rate Markov chain on a graph. This gives us relatively efficient way to set weights on a network topology that is applicable to all consensus based algorithms.



\chapter{Conclusion} 
\label{Chapter7}
\lhead{\emph{Summary and Conclusions}} 

In this thesis we proposed alternative analysis of distributed stochastic optimization schemes and showed that the performance of these methods, under the homogeneity assumption, depends on data dependent properties. We first analyzed a consensus based SGD scheme (Algorithm \ref{alg:DiSCO}) and showed that the rate of convergence involves the spectral properties of two matrices: the standard spectral gap of a weight matrix from the network topology and a new term depending on the spectral norm of the sample covariance matrix of the data (Theorem \ref{theorem:mainThrm}). This dependence allowed us to understand interactions between the data and the network properties, thereby allowing for network regimes better suited for certain distributions. Existing literature does not make any assumption on the data and we showed that under the homogeneity assumption we can obtain better convergence properties.

Since the consensus schemes communicate at every iteration we proposed a novel mini-batch scheme to achieve faster convergence in relatively fewer iterations (Theorem \ref{theorem:mainThrm:mbatch}). Another consequence of this analysis was to show that data distributions with low $\rho^2$ are more tolerant of skipped communication rounds. More surprisingly we were able to show that in the asymptotic regime, for smooth loss functions, the network and data effect completely disappears and we can get optimal performance from Algorithm \ref{alg:DiSCO} (Theorem \ref{theorem:mainThrmLargeApplication}).

Additionally we provided a data dependent analysis of speedups for the non-smooth mini-batch SGD case. Theoretically and empirically we showed that distributions with smaller value of $\rho^2$ are better suited for distributed optimization. We noted that these methods perform better on sparse datasets.

From an empirical perspective the most important contribution of this thesis is to show that the convergence and parallelization benefits of Distributed SGD algorithms depend on the largest eigenvalue of the sample covariance which can be estimated by using the power method. If it is too large (closer to $1$ than $0$) it indicates a high degree of repetition or correlation in the features. A key point is that for distributed SGD we gain more if there is more variety across the samples distributed in a network. For e.g. - for duplicates of samples distributed on different machines we won't gain anything since samples on different nodes offer no new information. Thus distributing a dataset is only really going to help if there is less redundancy in the dataset. This is captured by the largest eigenvalue of the sample covariance. 

We believe that this thesis provides a foundation of such analysis and the data dependence of several other distributed schemes could greatly explain their empirical properties. In the future we plan to explore the data dependent aspects of other learning schemes such as distributed one-shot averaging and even sequential SGD.



\addtocontents{toc}{\vspace{2em}} 





\addtocontents{toc}{\vspace{2em}}  
\backmatter

\label{Bibliography}
\bibliographystyle{abbrvnat}
\bibliography{opt}
\newpage 

\end{document}